\documentclass[11pt,reqno]{amsart}
\usepackage[margin=1in,letterpaper]{geometry}
\usepackage{graphicx}
\usepackage{amssymb}
\usepackage{amsthm}
\usepackage{mathrsfs}
\usepackage{stmaryrd}
\usepackage{accents} 
\usepackage{enumitem} 
\usepackage{microtype}
\usepackage{xcolor}
\usepackage{colonequals}

\usepackage[bookmarksopen,bookmarksdepth=2]{hyperref} 

\allowdisplaybreaks

\makeatletter\let\over\@@over\makeatother

\numberwithin{equation}{section}
\theoremstyle{plain} 
\newtheorem{theorem}{Theorem}[section] 
\newtheorem{proposition}[theorem]{Proposition} 
\newtheorem{corollary}[theorem]{Corollary}
\newtheorem{lemma}[theorem]{Lemma} 
\theoremstyle{remark}
\newtheorem{remark}[theorem]{Remark}
\theoremstyle{definition}

\newcommand{\be}{\begin{equation}}
\newcommand{\ee}{\end{equation}}%
\newcommand{\bse}{\begin{subequations}}
\newcommand{\ese}{\end{subequations}}

\newcommand{\dist}{\operatorname{dist}}
\newcommand{\realpart}{\operatorname{Re}}
\newcommand{\imagpart}{\operatorname{Im}}

\newcommand{\signum}[1]{\operatorname{sgn}{#1}}

\newcommand{\range}{\operatorname{rng}}
\newcommand{\kernel}{\operatorname{ker}}
\newcommand{\linspan}{\operatorname{span}}

\newcommand{\id}{\operatorname{id}}

\newcommand{\placeholder}{\,\cdot\,}

\newcommand{\n}[2][]{#1\lVert #2 #1\rVert}

\newcommand{\dell}{\partial}

\newcommand{\bdd}{\mathrm{b}}       
\newcommand{\loc}{{\mathrm{loc}} } 
\newcommand{\even}{{\mathrm{e}} }         
\newcommand{\odd}{{\mathrm{o}} }        
\newcommand{\per}{\mathrm{per}}  
\newcommand{\het}{\mathrm{f}}

\newcommand\G{\mathscr G}    

\newcommand\F{\mathscr F}    

\newcommand\cwave{c_{\mathrm{w}}}

\newcommand\rhos{\rho_{\mathrm{s}}}

\newcommand\Wspace{\mathscr W}
\newcommand\Xspace{\mathscr X}
\newcommand\Yspace{\mathscr Y}
\newcommand\Zspace{\mathscr Z}

\newcommand{\cm}{{\mathscr C}}  
       
\newcommand{\km}{{\mathscr K}}            

\newcommand\fluidD{\mathscr{D}}
\newcommand\fluidS{\mathscr{S}}
\newcommand{\fluidB}{\mathscr{B}}

\newcommand{\Nconf}{N_{\mathrm{c}}}
\newcommand{\Nvel}{N_{\mathrm{v}}}

\newcommand{\confD}{\mathcal{D}}
\newcommand{\confS}{\Gamma_1}
\newcommand{\confB}{\Gamma_0}

\newcommand{\sol}{\mathrm{sol}}

\newcommand{\vfield}{\mathbf{v}}    
\newcommand{\qfield}{\mathbf{q}}
\newcommand{\ellfunc}{\mathcal{E}}

\begin{document}

\title[Compressible steady water waves]{Compressible steady water waves}

\date{\today}

\author[A. Ramadan]{Abba Ramadan}
\address{Department of Mathematics, University of Alabama, Tuscaloosa, AL 35401} 
\address{Faculty Affiliate, CONSERVE Center, Alabama Water Institute, The University of Alabama, Tuscaloosa, AL, USA}
\email{aramadan@ua.edu} 
\author[S. Walsh]{Samuel Walsh}
\address{Department of Mathematics, University of Missouri, Columbia, MO 65211} 
\email{walshsa@missouri.edu}

\begin{abstract}
In this paper, we consider the existence of small- and large-amplitude compressible water waves. These are solutions of the two-dimensional free boundary isentropic Euler equations subjected to a constant gravitational force and neglecting the effects of surface tension. We study periodic and solitary waves beneath vacuum as well as front-type solutions of the two-phase problem in a channel. Using local and global bifurcation techniques, we construct large families of these waves beginning at an incompressible solution and treating the Mach number (defined in a suitable sense) as the bifurcation parameter. This gives a rigorous version of the classical Janzen--Rayleigh expansion, but also furnishes solution curves that may limit to transsonic flow.
\end{abstract}

\maketitle

\setcounter{tocdepth}{1}
\tableofcontents

\section{Introduction}
\label{introduction section}

Waves in water are typically assumed to be incompressible on the grounds that the characteristic velocity is quite small relative to the speed of sound.  However, there are certain physical situations in which even slight compressible effects can play a significant role. In pioneering work, for example, Longuet-Higgins~\cite{longuet1950theory} attributed the creation of seabed microseisms to them, which was later corroborated by field observations~\cite{kedar2008origin}. More recently, a number of authors have studied (linear) instabilities created by the resonant interactions between hydrodynamic gravity waves and acoustic waves; see, for example, \cite{kadri2013generation,kadri2016resonant,yang2018steady,kadri2021approximate,andrade2025acoustic}. Omira et al.~\cite{omira2022global} have argued that this dynamic is partially responsible for the surprising reach of the meteotsunami created by the Hunga Tonga-Hunga Ha'apai volcanic eruption in 2022.

All of these works, however, treat weakly nonlinear models and rely on formal methods.  It is somewhat surprising that in the over $75$ years since it was proposed by Longuet-Higgins, there have been no rigorous results on traveling wave solutions to the full isentropic Euler system. The objective of this paper is to use modern analysis techniques from local and global bifurcation theory to construct a plethora of exact compressible water waves. We will see that this can be done readily by exploiting the quite well-developed literature on irrotational incompressible waves.

\subsection{Governing equations}

\label{governing equations section}

For now, consider a single inviscid fluid occupying the time-dependent domain $\fluidD(t) \subset \mathbb{R}^2$; in Section~\ref{front section} we discuss a two-fluid problem.  By convention we write $(X,Y)$ for the spatial coordinates in the lab frame. For simplicity, we assume that the fluid is bounded below by a rigid, purely horizontal bed $\fluidB$, with axes chosen so that it is at $Y = -d$ for some depth $d >0$. The upper boundary of $\fluidD(t)$ is free and denoted by $\fluidS(t)$. Note that we will not assume a priori that $\fluidS(t)$ is the graph of a single-valued function of $x$. Above the fluid lies a region of constant pressure, normalized to be $0$.   

We will assume that the flow in $\fluidD(t)$ is irrotational but compressible and \emph{isentropic}, meaning that the density $\rho = \rho(t,\placeholder) \colon \overline{\fluidD(t)} \to \mathbb{R}_+$ and pressure $p = p(t,\placeholder) \colon \overline{\fluidD(t)} \to \mathbb{R}$ are functionally related according to the power law equation of state
\be
\label{definition of P}
	p = P(\rho) \colonequals  \frac{\alpha^2}{\gamma} \left(\frac{\rho}{\rhos}\right)^\gamma -  \frac{\alpha^2}{\gamma},
\ee
where $\gamma \geq 1$ is the \emph{adiabatic index}, a physical parameter representing the ratio of the specific heats, $\rhos > 0$ is the (constant) density at the surface, and $\alpha > 0$ is a proportionality constant. The speed of sound $c$ in the water is given by 
\begin{equation}
\label{definition sound speed}
	c^2 \colonequals P'(\rho) =  \alpha^2 \frac{\rho^{\gamma-1}}{\rhos^\gamma}.
\end{equation}
The \emph{isothermal} case corresponds to $\gamma = 1$ so that the equation of state is given by Boyle's law. By \eqref{definition sound speed}, it then follows that $c^2$ takes on the constant value $\alpha^2$ throughout the bulk of the fluid; this is the setting considered by Longuet-Higgins~\cite{longuet1950theory} and Omira et al.~\cite{omira2022global}. Our methodology also extends easily to the \emph{polytropic} case $\gamma > 1$, so we do not make this restriction.

Because the velocity field is curl-free, it can be represented as the gradient of a potential $\Phi = \Phi(t,\placeholder) \colon \overline{\fluidD(t)} \to \mathbb{R}$.  The isentropic Euler equations then take the form 
\begin{equation}
\label{isentropic Euler equations}
\left\{
\begin{aligned}
	\partial_t (\rho \nabla \Phi) + \nabla \cdot \left( \rho \nabla \Phi \otimes \nabla \Phi \right) + \nabla P(\rho) & = -g\rho \mathbf{e}_2 \\
	\partial_t \rho + \nabla \cdot (\rho \nabla \Phi) & = 0
\end{aligned}
\right.
\qquad \textrm{in } \fluidD(t),
\end{equation}
where $g > 0$ is the constant of gravitational acceleration. The first equation represents the conservation of momentum and the second is the conservation of mass. Substituting the mass equation into the momentum allows us to integrate, leading to the 
unsteady Bernoulli principle
\begin{equation}
\label{intro unsteady Bernoulli equation}
		\partial_t \Phi + \frac{1}{2} |\nabla \Phi|^2 + gY + H(\rho)  = B(t) \qquad  \textrm{in } \fluidD(t).
\end{equation}
Here, $B = B(t)$ is in general a time-dependent but spatially constant quantity and $H$ is the \emph{enthalpy} defined by
\be
\label{definition enthalpy}
	H(r) \colonequals  
	\left\{
	\begin{aligned}
		   \frac{\alpha^2}{\gamma-1} \left(\frac{r}{\rhos} \right)^{\gamma-1} & \qquad \textrm{if } \gamma > 1 \\
		\alpha^2  \log{\frac{r}{\rhos}} & \qquad \textrm{if } \gamma = 1.
	\end{aligned}
	\right.
\ee

Following Blokhintsev~\cite{blokhintsev1956acoustics}, the density and gravity can then be eliminated from~\eqref{intro unsteady Bernoulli equation} by taking the material derivative. For $\gamma > 1$, this calculation yields
\begin{equation}
\label{dimensional BLL equation}
\begin{aligned}
	0 & = \left(\partial_t  + \nabla \Phi \cdot \nabla \right) \left( \partial_t \Phi + \frac{1}{2} |\nabla \Phi|^2 +gY \right) + \alpha^2 \frac{\rho^{\gamma-2}}{\rhos^{\gamma-1}} (\partial_t + \nabla \Phi \cdot \nabla ) \rho \\
		& = \left(\partial_t  + \nabla \Phi \cdot \nabla \right) \left( \partial_t \Phi + \frac{1}{2} |\nabla \Phi|^2 +gY \right) + (\gamma-1) \left( \partial_t \Phi + \frac{1}{2} |\nabla \Phi|^2 + gY -B \right) \Delta \Phi.
\end{aligned}
\end{equation}
To obtain the second line from the first, we used the conservation of mass in \eqref{isentropic Euler equations}, then reexpressed $\rho$ in terms of $\Phi$ and $B$ via the Bernoulli equation \eqref{intro unsteady Bernoulli equation}. 

When $\gamma = 1$, the same idea gives us instead that
\begin{align*}
	0 
	& = \left(\partial_t  + \nabla \Phi \cdot \nabla \right) \left( \partial_t \Phi + \frac{1}{2} |\nabla \Phi|^2 +gY \right)  -\alpha^2 \Delta \Phi,
\end{align*}
which was first derived by Rayleigh~\cite{rayleigh1916flow}.

On the boundary $\partial \fluidD(t)$ of the fluid domain, we impose two conditions.  The \emph{kinematic condition} is the requirement that the velocity there be purely tangential, while the  \emph{dynamic condition} asks that the pressure be continuous across the free boundary $\fluidS(t)$. Because we are assuming that the region above $\fluidS(t)$ is at constant pressure (normalized to $0$), the dynamic condition simply asserts that $P$ must vanish identically along the free boundary. Note that this is consistent with our ansatz for $P$ in~\eqref{definition of P}, and it justifies the earlier statement that $\rho$ is constant on $\fluidS(t)$.  From \eqref{intro unsteady Bernoulli equation}, we therefore have
\begin{equation}
\label{unsteady dynamic condition}
	\partial_t \Phi + \frac{1}{2} |\nabla \Phi|^2 +g Y = Q(t) \qquad \textrm{on } \fluidS(t),
\end{equation}
for the \emph{Bernoulli constant} 
\[
	Q(t) \colonequals B(t) - H(\rhos). 
\]

\label{steady problem section}
Assume now that the flow is \emph{steady} in that it is independent of time when viewed in a reference frame that is translating at some fixed velocity $\cwave$ along the positive $X$-axis.  We write $(x,y)$ for the coordinates in the corresponding moving frame and, abusing notation slightly, denote by $\fluidD$ the (now steady) fluid domain, and likewise for $\fluidB$ and $\fluidS$. Let $\phi = \phi(x,y) \colon \overline{\fluidD} \to \mathbb{R}$ denote the potential for the relative velocity field, so that
\[
	\Phi(t,X,Y) =  \phi(x,y) + \cwave x.
\]
We will consider two classes of solutions to this system: \emph{periodic waves}, for which the unknowns are each periodic in the $x$-variable, and \emph{solitary waves}, for which the free boundary is asymptotically flat upstream ($x \to -\infty$) and downstream ($x \to +\infty$) with the same limiting depth in both directions. Note that in the latter case, the limiting velocity field is necessarily horizontal and of constant magnitude. Later, in Section~\ref{front section} we study a related two-phase problem, and construct traveling \emph{fronts} where the interface between the two phases is asymptotically flat but with distinct heights upstream and downstream. For now, though, we focus on just the one-phase model. These are illustrated in Figure~\ref{types of waves figure}(a)(b).

\begin{figure}
\includegraphics[width=0.8\linewidth]{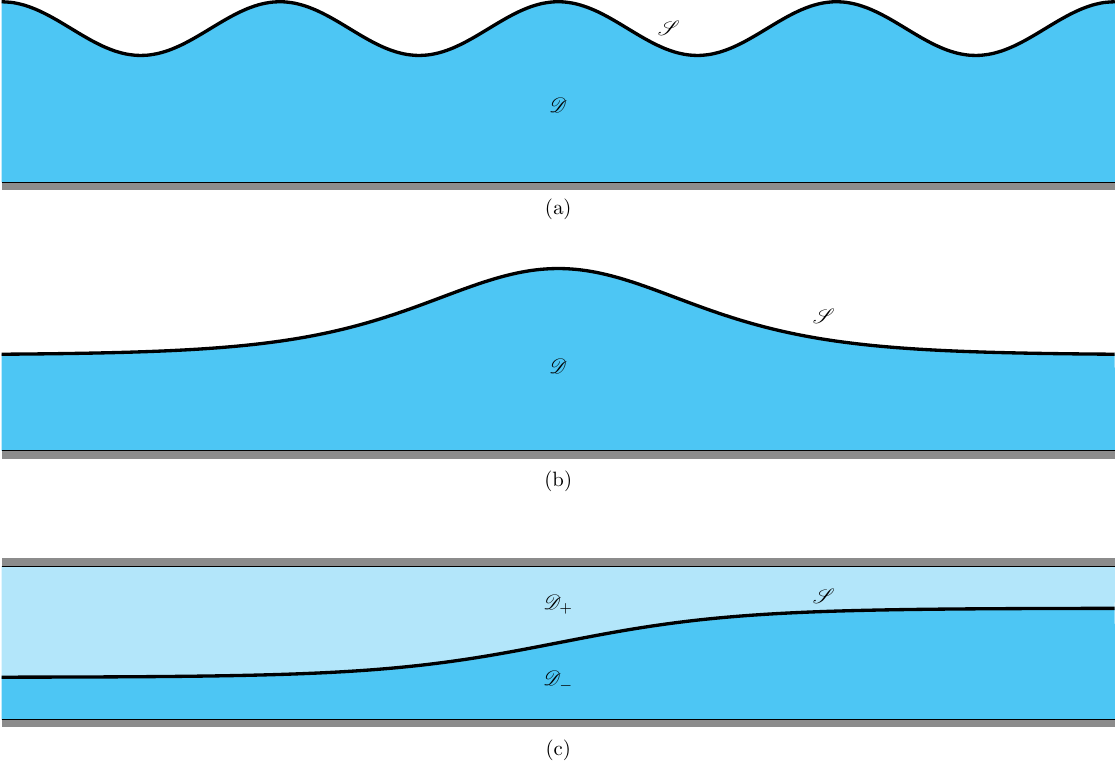}
\caption{The three species of traveling wave considered. The unshaded regions are vacuum at constant pressure $0$, while the light and dark blue shaded regions are fluid with positive density. The gray horizontal bars represent the rigid boundaries. The thick shaded line is the free boundary $\fluidS$.  (a) A \emph{periodic} or \emph{Stokes wave} in finite depth water. (b) A \emph{solitary wave} for which $\fluidS$ limits to the same horizontal line in the upstream and downstream limits. (c) A \emph{front} in a two-phase system, where the light $\fluidD_-$ and dark blue $\fluidD_+$ regions correspond to immiscible fluids separated by the free boundary $\fluidS$. }
\label{types of waves figure}
\end{figure}

 The next step is to nondimensionalize \eqref{dimensional BLL equation}. For periodic waves, we take the characteristic length scale $L$ to be the undisturbed mean depth $d$ of the water and the speed of the wave $\cwave$ as the characteristic velocity $U$. For solitary waves, we instead define $L$ to be the limiting depth and set $U$ to be the limiting (horizontal) speed. Then
 \begin{equation}
 \label{nondimensionalization space and velocity}
 \begin{aligned} 
 	 (\tilde{x},\tilde{y}) &\colonequals \left( \frac{x}{L},\frac{y}{L}\right),&\quad\tilde{\phi} &\colonequals \frac{\phi}{LU}, &\quad \tilde \rho &\colonequals \frac{\rho}{\rhos}, & \quad \tilde B & \colonequals \frac{B}{U^2}, \\
	 \tilde Q & \colonequals \frac{Q}{U^2}, &\quad \tilde c^2 &\colonequals \frac{c^2}{U^2}, &\quad \tilde\alpha^2 &\colonequals \frac{\alpha^2}{U^2}
\end{aligned}
\end{equation}
are the corresponding nondimensionalized physical variables, unknowns, and parameters. 

Rewriting \eqref{dimensional BLL equation} in terms of $\tilde{\phi}$ and dropping the tildes leads us to the governing equation
\bse
\label{steady Euler system}
\begin{equation}
\label{steady Euler}
	\Delta \phi = \mu  \left( \nabla \phi \cdot \nabla + (\gamma-1) \Delta \phi \right) \left( \frac{1}{2}  |\nabla \phi|^2 +  \frac{1}{F^2} y \right) \qquad \textrm{in } \fluidD
\end{equation}
where 
\be
\label{definition Mach number}
	\mu \colonequals \left\{
		\begin{aligned}
			\frac{1}{ (\gamma-1) B} & \qquad \textrm{if } \gamma \neq 1\\
			\frac{1}{\alpha^2} & \qquad \textrm{if } \gamma = 1
		\end{aligned}
	\right.
\ee
 is the (square) \emph{Mach number}, we are again abusing notation by writing $\fluidD$ for the nondimensionalized (steady) fluid domain, and 
\[
 	F^2	\colonequals \frac{U^2}{gL}
\]
is the (square) Froude number. 

Thus, formally $\mu = 0$ corresponds to the incompressible case where the relative potential is harmonic. The Froude number measures the relative influence of gravitational and inertial effects.  

Applying the same nondimensionalization to the dynamic condition \eqref{unsteady dynamic condition} yields
\begin{equation}
\label{steady dynamic condition}
	\frac{1}{2} |\nabla \phi|^2 + \frac{1}{F^2}y = Q \qquad \text{on } \fluidS.
\end{equation}

Finally, for the kinematic condition we have that
\begin{equation}
\label{steady kinematic condition}
	\nu \cdot\nabla\phi=0 \qquad \textrm{on } \fluidS,
\end{equation}
where $\nu$ is the outward unit normal to $\fluidD$ along $\fluidS$, and 
\begin{equation}
\label{steady kinematic condition-bed}
	\partial_y \phi = 0 \qquad \textrm{on } \fluidB.
\end{equation}
\ese
Note that the nondimensionalized form of Bernoulli's law is 
\be
\label{nondimensional Bernoulli law}
	\frac{1}{2} |\nabla \phi|^2 + \frac{1}{F^2}y + H(\rho) = B \qquad \textrm{in } \fluidD,
\ee
where, after nondimensionalization, the enthalpy now has the simpler form
\[
	H(r) = 
	\left\{
	\begin{aligned}
		   \frac{\alpha^2}{\gamma-1} r^{\gamma-1} & \qquad \textrm{if } \gamma > 1 \\
		\alpha^2  \log{r} & \qquad \textrm{if } \gamma = 1.
	\end{aligned}
	\right.
\]
For there to be a physical solution, the density $\rho$ needs to be strictly positive on $\overline{\fluidD}$, and thus in the polytropic case it must be that
\be
\label{intro subsonic flow}
	\inf_{\fluidD} \left( B - \frac{1}{2} |\nabla \phi|^2 - \frac{1}{F^2} y \right) > 0.
\ee
In particular, $B$ is strictly positive. We will see that this requirement is enforced by the assumption that the flow is subsonic, and will therefore be intimately related to the ellipticity of the system.

\subsection{Statement of results}

We now describe the main contributions of the paper. This will be done somewhat informally for the time being, because the precise statement requires making a change of variables.

Our first set of results establishes the existence of large families of compressible periodic and solitary surface water waves. These are \emph{lightly compressible} in the sense that the corresponding square Mach number $\mu$ is in a neighborhood of $0$; they are not necessarily small amplitude in that $\fluidS$ may be far from flat.

\begin{theorem}[Lightly compressible surface water waves] 
Let $(\phi_0, \fluidS_0, Q_0, F_0)$ be a solution to the water wave problem~\eqref{steady Euler system} that is incompressible in that $\mu = 0$.
\begin{enumerate}[label=\rm(\alph*)]
	\item \textup{(Periodic waves)} Suppose that for some $k > 0$,  $\phi_0$ and $\fluidD_0$ are $2\pi/k$-periodic in the $x$-direction. If the wave is nondegenerate as a solution of the incompressible problem, then there exists a real-analytic curve 
	\[
		\cm_\loc^\per = \left\{ (\phi^\mu,\fluidS^\mu, Q^\mu, \mu) : |\mu| \ll 1  \right \}
	\]
	of solutions to the compressible problem~\eqref{steady Euler system} that are also $2\pi/k$-periodic in $x$ with fixed Froude number $F_0$ and such that $(\phi^0,\fluidS^0,Q^0) = (\phi_0, \fluidS_0, Q_0)$.
	\item \textup{(Solitary waves)}
	If instead $(\phi_0, \fluidS_0,Q_0, F_0)$ is an incompressible solitary wave that is nondegenerate, then there exists a corresponding real-analytic curve
	\[
		\cm_\loc^\sol = \left\{ (\phi^\mu, \fluidS^\mu, \mu) : |\mu| \ll 1 \right\}
	\]
	of solutions to the compressible problem~\eqref{steady Euler system} with fixed Froude number $F_0$ and Bernoulli constant $Q_0 = 1/2$ that are also solitary waves and such that $(\phi^0,\fluidS^0) = (\phi_0, \fluidS_0)$.

\end{enumerate}

\label{intro JRE theorem}
\end{theorem}

\begin{remark}
See Theorems~\ref{periodic JRE theorem} and \ref{solitary JRE theorem} for rigorous versions of this result. The precise definition of nondegenerate is given in~\eqref{definition non-degenerate periodic} for periodic waves and~\eqref{definition non-degenerate solitary} for solitary waves. It holds generically in the sense that for the many families of incompressible waves constructed via analytic global bifurcation theory, all but a discrete set of solutions are nondegenerate. For example, the famous curve of Stokes waves begins at a laminar flow and terminates at an ``extreme wave'' with a $120^\circ$ corner at its crest. Applying a corollary of Theorem~\ref{intro JRE theorem} along this family furnishes lightly compressible but large-amplitude periodic waves that are likewise arbitrarily close to having the same corner singularity.
\end{remark}

\begin{remark}
\label{intro JRE remark}
Formally, one expects to be able to write $\phi^\mu$ as the incompressible relative potential $\phi_0$ with an $O(\mu)$ correction, in an appropriate norm, to account for the compressible effects. More generally, this suggests a power series expansion
\[
	\phi^\mu = \phi_0 + \sum_{m \geq 1} \mu^m \phi_m ,
\]
where each term $\phi_m$ is found by solving the hierarchy of equations that results from inserting this ansatz into the system~\eqref{steady Euler system} and doing matched asymptotics. This very natural approach was introduced by Janzen~\cite{janzen1913beitrag} and Rayleigh~\cite{rayleigh1916flow}, and is used extensively in the applied literature. It is somewhat more complicated for free boundary problems as one must also make corrections to the domain. Nonetheless, our construction gives a \emph{rigorous} justification for the Janzen--Rayleigh expansion: the validity of the power series representation is an immediate consequence of the analytic implicit function theorem. The equations defining the first-order corrections to the potential and domain are given in~\eqref{periodic JRE u} and \eqref{periodic JRE w}, respectively, but this can in principle be carried out to arbitrary order in $\mu$.
\end{remark}
\begin{remark}
While we only consider gravity waves, a straightforward adaptation of the argument applies to capillary-gravity waves. Similarly, one can construct lightly compressible three-dimensional waves of various types. The conformal mapping reformulation adopted in Section~\ref{local bifurcation section} will not work for this, but it is possible to instead use a naïve flattening transform as we do in Section~\ref{front section}. It would also be interesting to consider the case of compressible water waves with vorticity. Recently, Xie and Xin~\cite{xie2010global,xie2010existence} were able to construct large families of subsonic rotational flows through a rigid nozzle (absent gravity). We believe a version of their approach could be used in the free boundary setting as well by incorporating our techniques.
\end{remark}

The next theorem takes these lightly compressible families and continues them into the finite Mach number regime. For simplicity, we present only the periodic case, but similar methods can be applied --- with some additional difficulties --- to solitary waves and fronts as well. 

Before stating the result, let us introduce some terminology to help describe the possible limiting behavior as one follows these families to their extreme.  As we detail further in Section~\ref{lightly compressible periodic waves section}, we will formulate the Euler equations~\eqref{steady Euler system} in a fixed domain using a conformal mapping. If the physical domain $\fluidD$ is even with respect to the $y$-axis and $2\pi/k$-periodic, then there exists a unique $h > 0$ and holomorphic function $f = f(\zeta)$ whose domain is the slab
\[
	\confD \colonequals \{ \zeta \in \mathbb{C} : -kh < \imagpart{\zeta} < 0 \},
\] 
with upper and lower boundary components denoted by $\Gamma_1$ and $\Gamma_0$, respectively, and such that $f$ satisfies the following:
\begin{itemize}
	\item $f$ admits a $C^{2+\beta}(\overline{\confD})$ extension as a homeomorphism between $\overline{\confD}$ and $\overline{\fluidD}$ with
\[
	 \fluidD = f(\confD), \qquad f(\Gamma_0) = \fluidB, \qquad f(\Gamma_1) = \fluidS, \qquad f(\placeholder + 2\pi) = f + \frac{2\pi}{k};
\]
	\item $\partial_\zeta f$ is non-vanishing on $\overline{\confD}$; and 
	\item  $f$ is injective on $\overline{\confD}$. 
\end{itemize}	
See~\cite[Appendix A]{constantin2011constant}. 
One possibility is that, as one continues the solution curve, this conformal description degenerates. Specifically, this means that along the global curve, the quantity  
\be
\label{intro conformal degeneracy}
	\Nconf(\fluidD) \colonequals \sup_{\Gamma_1} { {|\partial_\zeta f|}} +  \sup_{\substack{\zeta, \zeta^\prime \in \partial\confD \\ \zeta\neq\zeta^\prime}}{\frac{|\zeta-\zeta^\prime|}{|f(\zeta)-f(\zeta^\prime)|} }  \qquad \textrm{is unbounded}.
\ee
Blowup of the first term is suggestive of solutions arbitrarily close to manifesting a surface singularity. Unboundedness of the second term, on the other hand, would occur if conformality fails on the boundary of $\confD$, or if $f$ fails to be injective on $\partial\confD$. The latter case would arise if $\fluidS$ self-intersects or comes into contact with $\fluidB$, at which point the physicality of the model breaks.

Another alternative is that the velocity field degenerates, by which we mean that along the global curve the quantity 
\be
\label{intro velocity degeneracy}
	\Nvel(\phi, \mu) \colonequals \| \phi \|_{C_\per^{2+\beta}(\fluidD)} + |\mu| \qquad \textrm{is unbounded}.
\ee
The blowup of the first term corresponds to blowup of the velocity field in $C^{1+\beta}$, which is a higher-regularity norm than one often sees in global bifurcation theoretic arguments of this type for incompressible waves. Typically, one hopes to be able to do a bootstrap argument to show that boundedness of the velocity field in $C^1$ implies boundedness of the higher-regularity norm as well. As we will see in Section~\ref{lightly compressible periodic waves section}, however, in the compressible setting the system is quasilinear with a complete coupling between the unknowns at the leading order. This additional structural complexity makes it difficult to apply a similar argument. We note that Xie and Xin~\cite{xie2010existence} likewise have $C^{1+\beta}$ blowup of the velocity as one of their unboundedness alternatives for the nozzle problem.

\begin{theorem}[Finite Mach number periodic waves]
\label{intro global periodic theorem}
Let $\cm_\loc^\per$ be the real-analytic curve of lightly compressible periodic waves furnished by Theorem~\ref{intro JRE theorem}. There exists a curve $\cm_\per$ of compressible periodic water waves admitting the $C^0$ parameterization
\[
	\cm_\per = \left\{ (\phi(s),\fluidS(s),Q(s),\mu(s) ) : s \in [0,\infty) \right\} 
\]
with $\cm_\loc^\per \subset \cm_\per$. In the limit $s \to \infty$, one of the following alternatives must occur:
		\begin{enumerate}[label=\rm(A\arabic*)]
			\item \label{intro periodic unboundedness alternative} \textup{(Unboundedness)} Conformal degeneracy~\eqref{intro conformal degeneracy} or velocity degeneracy occurs~\eqref{intro velocity degeneracy}: 
			\[
				 \Nconf(\fluidD(s)) + \Nvel(\phi(s), \mu(s)) \to \infty,
			\]
			where $\fluidD(s)$ is the fluid domain determined by $\fluidS(s)$.
			\item \label{intro transsonic limit alternative} \textup{(Sonic point)} At a point in $\overline{\fluidD(s)}$, the flow approaches transsonic:
			\be
			\label{intro transsonic limit}
		\lim_{s \to \infty} \inf_{\fluidD(s)} \left(c(s)^2 - |\nabla \phi(s)|^2 \right) = 0,
			\ee
			where $c(s)$ is the speed of sound at the parameter value $s$.
			\item \label{intro stagnation alternative} \textup{(Surface stagnation)} At a point on the free boundary $\fluidS(s)$, the flow approaches stagnation:
			\[
				\lim_{s \to \infty} \inf_{\fluidS(s)} |\nabla \phi(s)| = 0.
			\]
			\item \label{intro periodic loop alternative} \textup{(Closed loop)} There exists $T > 0$ such that 
			\[
				\left(\phi(s+T),\, \fluidS(s+T),\, Q(s+T),\, \mu(s+T) \right)= \left(\phi(s),\, \fluidS(s), \, Q(s),\, \mu(s) \right)
			\]
			for all $ s \in [0,\infty)$. If this occurs, then necessarily $\cm_\per$ contains two distinct incompressible solutions.
		\end{enumerate}
\end{theorem}
\begin{remark}
Thanks to a classical result of Spielvogel~\cite[Theorem 3b]{spielvogel1970variational}, it is known that incompressible irrotational gravity water waves cannot overhang, meaning here that the interface $\fluidS$ can always be written as the graph of a single-valued function of $x$. That argument, however, does not apply in the compressible setting. Because we are using a conformal mapping to fix the domain, the curve $\cm_\per$ potentially includes overhanging waves. Establishing this analytically, however, is extremely challenging.
\end{remark}

The second part of the paper deals with the related problem of traveling fronts moving along an interface dividing two immiscible fluids; see Figure~\ref{types of waves figure}(c) for an illustration. There are now two fluid domains, call them $\fluidD_+$ and $\fluidD_-$, which are separated by a free boundary $\fluidS$ across which the pressure is continuous but the density can jump. The flow in each fluid is governed by the isentropic irrotational Euler equations; we take the equation of state to be isothermal, but allow the speed of sound to be different in each layer. We suppose that the free boundary is given globally as the graph of an unknown function $h = h(x)$, with impermeable horizontal boundaries at $y = 0$ and $y = 1$ (in the nondimensionalized variables). Using a similar argument to eliminate the pressure and density, the problem can be formulated in terms of the relative potentials $\phi_\pm$ in $\fluidD_\pm$ and $h$. After suitable nondimensionalization, there is a square Mach number $\mu_\pm$ associated to each fluid.

Solutions of this type are physically interesting for a number of reasons. As one example, the meteotsunami model considered by Omira et al.~\cite{omira2022global} concerns a two-fluid system, with an initial pressure forcing in the upper fluid (the air) generating traveling waves in the lower fluid (the ocean). To keep the presentation here concise, we postpone a detailed discussion of the governing equations until Section~\ref{front section}. For variety, and because they are analytically more challenging, we focus on the case of front-type solutions, meaning that the upstream limiting depth and downstream limiting depth differ. These are often called (smooth) hydrodynamic bores. 

Using a thematically similar implicit function theorem argument, we are able to prove that a Janzen--Rayleigh-type expansion can be carried out at a generic incompressible front solution, which gives rise to large families of lightly compressible fronts. Again, these are not small amplitude, and in fact include solutions that are arbitrarily close to having an overturned interface. See Theorem~\ref{front theorem} and Corollary~\ref{front corollary}. A very informal statement of this result is given below.

\begin{theorem}[Lightly compressible fronts]
Suppose that $(\phi_{0+}, \phi_{0-}, h_0)$ is a solution to the internal front problem~\eqref{two-phase isentropic Euler system} that is incompressible in that $\mu = 0$ and nondegenerate. There exists a real-analytic curve 
	\[
		\cm_\loc^\het = \left\{ (\phi_+^\mu,\phi_-^\mu, h^\mu, \mu) : |\mu| \ll 1  \right \}
	\]
	of front-type solutions to the compressible problem~\eqref{two-phase isentropic Euler system} and such that $\phi_\pm^0 = \phi_{0\pm}$ and $h^0 = h_0$.
\end{theorem}

\subsection{Plan of the article}

We begin in Section~\ref{preliminaries section} with some preliminary material, including a list of notational conventions used throughout the paper. For the convenience of the reader, we also recall the analytic global implicit function theorem, which will serve as the main tool for constructing finite Mach number flows.

Section~\ref{local bifurcation section} is devoted to the proof of Theorem~\ref{intro JRE theorem} on the existence of lightly compressible periodic and solitary waves and the validity of the Janzen--Rayleigh expansion. The main work is simply to formulate the problem correctly so that the (local) analytic implicit function theorem can be applied at a nondegenerate incompressible wave. We show that large families of such waves exist, and in fact that these give rise to two-dimensional manifolds of lightly compressible waves. In Section~\ref{global bifurcation section}, we take up the proof of Theorem~\ref{intro global periodic theorem}. The basic strategy is to make a global continuation of the families of lightly compressible periodic waves. This machinery relies crucially on certain compactness properties that in turn are derived from the ellipticity of the system in the Agmon--Douglis--Nirenberg sense. We prove that this relates directly to the flow being subsonic~\eqref{intro subsonic flow}, which is precisely why the development of a sonic point~\eqref{intro transsonic limit} is one of the alternatives for the limiting behavior of the global curve. This analysis is strongly inspired by the work of Haziot and Wheeler~\cite{haziot2023large} on large-amplitude (incompressible) solitary water waves with constant vorticity, who also used conformal mappings and elliptic systems regularity theory. Here, however, the problem is quasilinear, whereas in their paper it is semilinear, which in particular breaks a number of bootstrapping arguments.

Finally, in Section~\ref{front section}, we consider the existence of lightly compressible fronts. The underlying argument is again based on applying the implicit function theorem. However, there are substantial new difficulties in this setting. For instance, conformal mappings are no longer an attractive means of fixing the domain, as each fluid layer would in principle require its own mapping, and these will in general give conflicting parameterizations of the shared interface. We instead use a naïve flattening transform, performing an $x$-dependent vertical scaling to force the layers to be of uniform width. The resulting equations are naturally far more difficult to analyze. Because fronts have distinct limits  upstream and downstream, the functional analytic setup is likewise more subtle than for either the periodic or solitary wave cases. Ultimately, we are able to show that nondegeneracy in the stream function formulation of the incompressible problem implies nondegeneracy in our formulation, so that the implicit function theorem yields the lightly compressible fronts.

\section{Preliminaries}
\label{preliminaries section}

\subsection{Notation}
\label{notation section}

For $\Omega \subset \mathbb{R}$ or $\mathbb{R}^2$, $k \geq 0$, and $\beta \in (0,1)$, the space $C^{k+\beta}(\Omega)$ consists of $k$-times continuously differentiable functions on $\Omega$ whose $k$-th order derivatives are Hölder continuous with exponent $\beta$.

 For any of these spaces, a subscript of ``c'' indicates the subset of functions whose support is compactly contained inside $\Omega$. When $\Omega$ is unbounded, for clarity we use a subscript of ``$\bdd$'' to indicate functions that are uniformly bounded in that their norm on $\Omega$ is finite. Conversely, a subscript of ``$\loc$'' denotes the set of functions whose norm is bounded on each compact subset of $\Omega$. Thus, for example, 
\[
	C_\bdd^{k+\beta}(\Omega) \colonequals \left\{ u \in C^{k+\beta}(\Omega) : \| u \|_{C^{k+\beta}(\Omega)} < \infty\right\},
\]
where $\| \placeholder \|_{C^{k+\beta}(\Omega)}$ is the standard Hölder norm on $\Omega$.

\subsection{Analytic global bifurcation theory}

Here, we state the version in \cite[Theorem B.1]{chen2023global}, which essentially rephrases the classical analytic global bifurcation theory of Dancer~\cite{dancer1973bifurcation,dancer1973globalstructure} and Buffoni and Toland~\cite{buffoni2003analytic}. 

\begin{theorem}[Global implicit function theorem] \label{homoclinic global ift} 
Let $\mathscr{W}$ and $\mathscr{Z}$ be Banach spaces, $\mathcal{O} \subset \mathscr{W} \times \mathbb{R}$ an open set containing a point $(w_0, \lambda_0)$. Suppose that  $\mathscr{H} \colon \mathcal{O} \to \mathscr{Z}$ is real analytic and satisfies 
  \begin{equation}
    \mathscr{H}(w_0, \lambda_0) = 0, \qquad \mathscr{H}_w(w_0, \lambda_0) \colon \mathscr{W} \to \mathscr{Z} \quad \textrm{is an isomorphism}.  \label{homoclinic global ift assumptions} 
  \end{equation}
 Then there exist a curve $\mathscr{K}$ that admits the global $C^0$ parameterization  
 \[ \mathscr{K} \colonequals  \left\{ (w(s), \lambda(s)) : s \in \mathbb{R}  \right \} \subset \mathscr{H}^{-1}(0) \cap \mathcal{O},\]
  and satisfies the following.  
  \begin{enumerate}[label=\rm(\alph*)]
  \item \label{K well behaved} At each $s \in \mathbb{R}$, the linearized operator $\mathscr{H}_w(w(s), \lambda(s)) \colon \mathscr{W} \to \mathscr{Z}$ is Fredholm index $0$.
  \item \label{K alternatives} One of the following alternatives holds as $s \to \infty$ and $s \to -\infty$.
    \begin{enumerate}[label=\rm(A\arabic*$'$)]
    \item  \label{K blowup alternative}
      \textup{(Blowup)}  The quantity 
      \begin{align}
        \label{K blowup}
        N(s)\colonequals  \n{w(s)}_{\mathscr{W}} + |\lambda(s)| +\frac 1{\dist((w(s),\lambda(s)), \, \dell \mathcal{O})} \to \infty.
      \end{align}
    \item \label{K loss of compactness alternative} \textup{(Loss of compactness)} There exists a sequence $s_n \to \pm\infty$ with $\sup_n N(s_n) < \infty$, but $( w(s_n), \lambda(s_n) )$ has no convergent subsequence in $\mathscr{W} \times \mathbb{R}$.  
      \item \label{K loss of fredholmness alternative} \textup{(Loss of Fredholmness)}    There exists a sequence $s_n \to \pm\infty$
        with $\sup_{n} N(s_n) < \infty$ and so that $(w(s_n), \lambda(s_n)) \to (w_*, \lambda_*) \in \mathcal{O}$ in $\mathscr{W} \times \mathbb{R}$, however $\mathscr{H}_w(w_*, \lambda_*)$ is not semi-Fredholm.  
   \item \label{K loop} \textup{(Closed loop)} There exists $T > 0$ such that $(w(s+T), \lambda(s+T)) = (w(s), \lambda(s))$ for all $s \in (0,\infty)$. 
        \end{enumerate}
          \item \label{K reparam} Near each point $(w(s_0),\lambda(s_0)) \in \mathscr{K}$, we can locally reparametrize $\mathscr{K}$ so that $s\mapsto (w(s),\lambda(s))$ is real analytic.
 \item \label{K maximal part} The curve $\mathscr{K}$ is maximal in the sense that, if $\mathscr{J} \subset \mathscr{H}^{-1}(0) \cap \mathcal{O}$ is a locally real-analytic curve containing $(w_0,\lambda_0)$ and along which  $\mathscr H_w$ is Fredholm index $0$, then $\mathscr{J} \subset \mathscr{K}$. 
  \end{enumerate}
\end{theorem}
This theorem follows from a straightforward adaptation of \cite[Theorem 6.1]{chen2018existence}.  A sketch of the argument can be found in \cite[Appendix B]{chen2023global}.  Note that here \ref{K loss of fredholmness alternative} refers to the loss of semi-Fredholmness, which by continuity of the index is equivalent to the index ceasing to be $0$ as stated in~\cite{chen2023global}.

Let us also mention some further structural properties of the global curve $\km$ from Theorem~\ref{homoclinic global ift}. Denote by
\[
	\mathscr{N} \colonequals \{ (w,\lambda) \in \mathscr{H}^{-1}(0) : \mathscr{H}_w(w,\lambda) \colon \Wspace \to \Zspace \textrm{ is an isomorphism} \}
\]
the set of nondegenerate points along the curve. The maximal connected components of $\mathscr{N}$ are called \emph{distinguished arcs}. By the implicit function theorem, each of these is a real-analytic one-dimensional manifold. In the proof of Theorem~\ref{homoclinic global ift}, it is shown that there are at most countably many distinguished arcs, $\{ \mathscr{A}_j \}$, and $\km = \bigcup_{j} \overline{\mathscr{A}_j}.$
The degenerate points $\km \setminus \mathscr{N}$ form a discrete set.

\section{Lightly compressible periodic and solitary water waves}
\label{local bifurcation section}

In this section, we prove that one can ``generically'' take a periodic or solitary wave solution to the incompressible problem~\eqref{steady Euler system} with $\mu = 0$, and find nearby traveling waves of the same type that are lightly compressible in that $0 < |\mu| \ll 1$. These solutions are, moreover, parameterized locally by the Mach number and we rigorously justify the validity of expressing them as an (infinite) power series in $\mu$, as was formally carried out by Janzen~\cite{janzen1913beitrag} and Rayleigh~\cite{rayleigh1916flow}.

\subsection{Periodic waves}
\label{lightly compressible periodic waves section}

Consider first the case of periodic water waves.

\subsubsection*{Conformal reformulation}

\begin{figure}
	\includegraphics{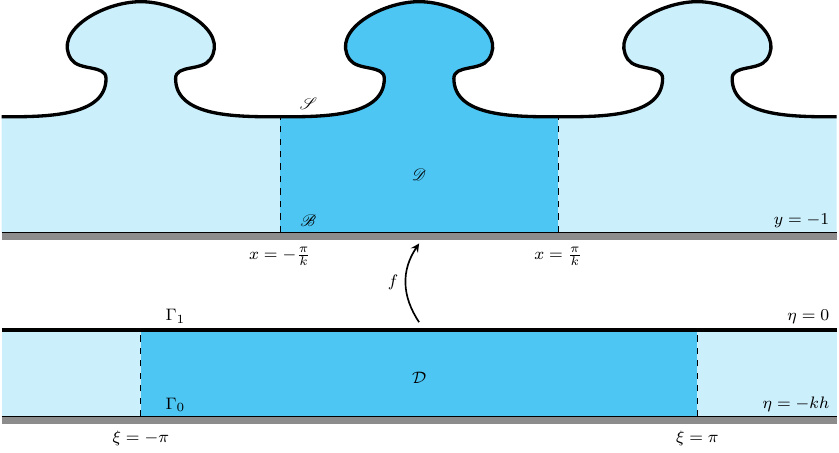}
	\caption{The physical domain $\fluidD$ in the $(x,y)$-plane and conformal domain $\confD$ in the $\zeta$-plane. A single period in each domain is shaded in darker blue. Note that in general, the upper boundary $\fluidS$ in the physical coordinates need not be the graph of a single-valued function of $x$. We are depicting an overhanging interface (in the physical variables) to emphasize that these are permitted with our formulation.}
		\label{periodic setup figure}
\end{figure}

As usual, a preliminary step in analyzing free boundary problems such as~\eqref{steady Euler system} is to make a change of variables recasting it in terms of functions defined on a fixed domain. There are many ways to achieve this, each with their own advantages. Here, we exploit the two-dimensionality of the problem to use a conformal mapping. That is, we identify the fluid domain $\fluidD$ in the physical $(x,y)$ variables with a subset of the complex plane, and view it as the image under an a priori unknown conformal mapping $f$ of a canonical reference domain $\confD$, which is a subset of the complex $\zeta$-plane. We will always write $\zeta = \xi + i \eta$, identifying the complex $\zeta$-plane with the real $(\xi,\eta)$-plane.

Suppose that we are given a $2\pi/k$-periodic solution to the irrotational and incompressible water wave problem for some wave number $k > 0$.  There exists a unique $h > 0$, called the \emph{conformal depth} of the fluid domain, for which there is a conformal mapping $f$ with domain 
\be
\label{periodic conformal domain definition}
	\confD \colonequals \left\{ \zeta = \xi + i \eta \in \mathbb{C} : -k h < \eta < 0 \right\},
\ee
and such that
\be
\label{correspondence conformal and physical domains}
	 \fluidD = f(\confD), \qquad f(\confB) = \fluidB, \qquad f(\confS) = \fluidS, \qquad f(\placeholder + 2\pi) = f + \frac{2\pi}{k}.
\ee
Here, 
\[
	\Gamma_0 \colonequals \{ \eta = -kh \}, \qquad \Gamma_1 \colonequals \{ \eta = 0 \}
\]
are the boundary components of $\confD$. In fact, one can verify that
\[
	h = \frac{1}{2\pi} \int_{-\pi}^{\pi} \imagpart f(\xi,0) \, d\xi. 
\] 
This setup is illustrated in Figure~\ref{periodic setup figure}.

For irrotational incompressible flow, the relative potential in the conformal variables will be a (real) scalar multiple of $\xi$. With that in mind, for the compressible case, we write
\[
	\phi \circ f \equalscolon u + \frac{m}{kh} \xi,
\]
where the new unknown $u = u(\xi,\eta)$ will be $2\pi$-periodic in $\xi$ and normalized so that it has mean zero over one period on the upper boundary. The constant factor multiplying $\xi$ is motivated by the observation that in the incompressible setting where $u = 0$, the relative stream function will be the harmonic conjugate of $\phi \circ f$, namely $\zeta \mapsto m\eta/kh$. Thus $-m$ is the volumetric mass flux in that case. To simplify notation, denote
\[
	\lambda \colonequals \frac{m}{kh},
\]
which serves as the bifurcation parameter in the existence theory for incompressible periodic waves.

We now take as our unknowns the function $u$ given above, and 
\[
	w \colonequals \imagpart{\left(f - \frac{1}{kh}\id\right)}.
\]
Note that $w$ is then $2\pi$-periodic in $\xi$, and it can be used to reconstruct the entire conformal mapping $f$.  In light of~\eqref{steady Euler} and the analyticity of $f$, in the interior of the conformal domain we see that  
\bse
\label{conformal governing equations periodic}
\begin{equation}
\label{qpde periodic}
\left\{
\begin{aligned}
	\Delta u & =  \mu \left(  (\gamma-1) \Delta u + (\nabla u + \lambda \mathbf{e}_1)  \cdot \nabla \right) \left( \frac{1}{2} \frac{|\nabla u + \lambda \mathbf{e}_1|^2}{|\nabla w+\frac{1}{kh}\mathbf{e}_2|^2} + \frac{1}{F^2} (w+\eta)  \right)  \\
	\Delta w & = 0
\end{aligned}
\right.
\qquad \textrm{in } \confD.
\end{equation}
The kinematic condition~\eqref{steady kinematic condition}--\eqref{steady kinematic condition-bed} on both boundary components likewise becomes 
\begin{equation}
\label{kinematic condition periodic}
	\partial_\eta u = 0 \qquad \textrm{on } \partial \confD.
 \end{equation}
 To ensure that the bottom boundary $\Gamma_0$ is mapped to the physical bed at $\{ y=-1\}$, we must have that
 \be
 \label{bed condition periodic}
 	w =0 \qquad \text{on } \Gamma_0.
\ee
The dynamic condition~\eqref{steady dynamic condition} accordingly takes the form
\begin{equation}
\label{cbernoulli periodic}
    \frac{1}{2} \frac{| \nabla u + \lambda \mathbf{e}_1|^2}{|\nabla w+\frac{1}{kh}\mathbf{e}_2|^2} + \frac{1}{F^2}w = Q \qquad \textrm{on } \Gamma_1.
\end{equation}
For periodic waves, $Q$ will be treated as one of the unknowns that varies along the solution family to be constructed. Lastly, to maintain the spatial period, we fix the conformal mean depth along the solution curve.  This is accomplished by requiring
\be
\label{conformal mean depth equation}
	\frac{1}{2\pi} \int_{-\pi}^\pi u(\xi,0) \, d\xi  = \frac{1}{2\pi} \int_{-\pi}^\pi w(\xi,0) \, d\xi = 0.
\ee  
\ese

Likewise, the full Bernoulli law~\eqref{intro unsteady Bernoulli equation}, non-dimensionalized, specialized to the steady problem, and pulled back to conformal domain is
\be
\label{conformal Bernoulli law}
	\frac{1}{2} \frac{|\nabla u + \lambda \mathbf{e}_1|^2}{|\nabla w+\frac{1}{kh}\mathbf{e}_2|^2} + \frac{1}{F^2} (w+\eta) + H(\varrho) = B  \qquad \textrm{in } \confD,
\ee
where $\varrho = \rho \circ f$ is the pullback of the density to the conformal domain. Given $(u,w,Q)$, the above equation allows us to reconstruct $\varrho$. 

\subsubsection*{Abstract operator equation}

In order to apply the implicit function theorem (and eventually the global implicit function theorem), we now reconfigure the governing equations~\eqref{conformal governing equations periodic} as an abstract operator equation of an appropriate type for the unknown $(u,w,Q)$.  

The kinematic condition~\eqref{kinematic condition periodic} as well as the normalization of the bed \eqref{bed condition periodic} and conformal mean depth \eqref{conformal mean depth equation} are linear and can therefore be encoded in the function spaces.  Define
\be
\label{definition X space and Y space periodic}
\begin{aligned}
	\Xspace_\per & \colonequals \left\{ (u, w) \in \mathring{C}_{\odd,\per}^{2+\beta}(\overline{\confD}) \times  \mathring{C}_{\even,\per}^{2+\beta}(\overline{\confD}) : ~ \partial_\eta u|_{\Gamma_0\cup\Gamma_1}=0, ~w|_{\Gamma_0}=0  \right\}, \\
		& \equalscolon \Xspace_{1\per} \times \Xspace_{2\per} \\
	\Yspace_{\per} &\colonequals  C_{\odd,\per}^{\beta}(\overline{\confD})\times C_{\even,\per}^{\beta}(\overline{\confD}) \times C_{\even,\per}^{1+\beta}(\Gamma_1) \equalscolon \Yspace_{1\per} \times \Yspace_{2\per} \times \Yspace_{3\per} .
\end{aligned}
\ee
Here, the subscript $\per$ indicates $2\pi$-periodicity in $\xi$, and $\mathring{C}_\per^{\ell + \beta}$ denotes the subspace of $C_\per^{\ell+\beta}$ consisting of functions that have vanishing mean over a period on $\Gamma_1$.  Similarly, subscripts of ``e'' or ``o'' indicate the subspace of functions that are even or odd in $\xi$, respectively. This restriction is needed to remove the translation symmetry.  

The full problem~\eqref{conformal governing equations periodic} is then equivalent to the abstract operator equation
\begin{equation}
\label{meqn periodic}
\left\{
\begin{aligned}
      \F(u,w,Q,\lambda; \mu) &= 0 \\
      \G(u,w,Q,\lambda; \mu) &= 0, 
\end{aligned}
\right.
\end{equation}
where $(\F,\G)$ is the real-analytic mapping $\mathcal{O} \subset \Xspace_{\per} \times \mathbb{R}^3 \to \Yspace_{\per}$ given by 
\be
\label{definition abstract F}
	\F(u,w,Q,\lambda; \mu) \colonequals
		\begin{pmatrix}
			\Delta u - \mu \left(  (\gamma-1) \Delta u + (\nabla u  + \lambda \mathbf{e}_1) \cdot \nabla  \right)  \left( \frac{1}{2} \frac{|\nabla u + \lambda \mathbf{e}_1|^2}{|\nabla w+\frac{1}{kh}\mathbf{e}_2|^2} + \frac{1}{F^2} (w+\eta) \right) \\
			\Delta w
		\end{pmatrix},
\ee
and
\[
	\G(u,w,Q,\lambda; \mu) \colonequals \left( \frac{1}{2} \frac{|\nabla u + \lambda \mathbf{e}_1|^2}{|\nabla w+\frac{1}{kh}\mathbf{e}_2|^2} + \frac{1}{F^2} w -Q \right)\Big|_{\Gamma_1}.
\]
The domain of $(\F,\G)$ is the open set $\mathcal{O}$; it is chosen so that membership enforces the ellipticity of the system, which physically relates to be the flow being subsonic, and also that the mapping $f$ is conformal and globally injective. The latter of these facts is necessary in order to reconstruct solutions of the original problem from that of the conformal formulation. Specifically, for $\delta > 0$, we set
\be
\label{definition periodic O delta}
\begin{aligned}
	\mathcal{O}_\delta  \colonequals  \Big\{ (u,w,Q,\lambda, \mu) & \in \Xspace_\per \times \mathbb{R}^3  :   \\
	& \qquad \sup_{\Gamma_1} \frac{1}{w+\eta+kh}  + \sup_{\substack{\zeta, \zeta^\prime \in  \Gamma_1 \\ \zeta \neq \zeta^\prime}} \frac{|\zeta - \zeta^\prime|}{|w(\zeta) + \zeta - w(\zeta^\prime) - \zeta^\prime|}   < \frac{1}{\delta},  \\
	&  \qquad \sup_{\confD} \left( \frac{\gamma+1}{2}\frac{|\nabla u+\lambda\mathbf{e}_1|^2}{|\nabla w+\frac{1}{kh}\mathbf{e}_2|^2} + \frac{(\gamma-1)(w+\eta)}{F^2} \right)
< \frac{1}{|\mu|} - \delta, \\
	& \qquad \inf_{\Gamma_1} \frac{(\partial_\xi u + \lambda)^2}{|\nabla w + \frac{1}{kh}\mathbf{e}_2|^3} > \delta \Big\}, 
\end{aligned}
\ee
where the third inequality is assumed to hold automatically when $\mu = 0$. We then define $\mathcal{O} \colonequals \bigcup_{\delta > 0} \mathcal{O}_\delta$. Notice that the first term controlled by $\delta$ in~\eqref{definition periodic O delta} gives (quantitative) bounds on the distance between the free boundary and the bed. The second term is a chord-arc distance, which gives both a lower bound on $|\partial_\zeta f|$ along $\Gamma_1$ and implies injectivity of $f$ along $\Gamma_1$. The third term is an ellipticity condition in the Agmon--Douglis--Nirenberg sense; see Section~\ref{ellipticity section}. Of course, it will hold automatically for $|\mu| \ll 1$, but it will become relevant for the global bifurcation argument to be undertaken later. The fourth condition excludes (relative) stagnation points on the free surface; it is needed for the Lopatinskii condition, see Proposition~\ref{lopat prop}, and is used in the proof of Lemma~\ref{schauder lemma}.

\begin{lemma}[Equivalence for periodic waves]
\label{periodic equivalence lemma}
Fix $\gamma \geq 1$ and suppose that $(u,w,Q,\lambda,\mu) \in \mathcal{O}$ satisfies the abstract operator equation~\eqref{meqn periodic} with $\mu > 0$. Assume further that either
\begin{enumerate}[label=\rm(\roman*)]
	\item \label{f is conformal case} $\inf_{\confD} |\nabla w+\frac{1}{kh}\mathbf{e}_2| > 0$ or 
	\item \label{connected component case} $(u,w,Q,\lambda,\mu)$ is in a connected component of $\mathcal{O}$ containing a laminar solution $(u_*,0,Q_*,\lambda_*, \mu_*)$.
\end{enumerate}
There exists $\alpha^2 > 0$, a $C^{2+\beta}$  fluid domain $\fluidD$, and relative velocity potential $\phi \in C^{2+\beta}(\overline{\fluidD})$ solving the (non-dimensionalized) steady isentropic Euler equations~\eqref{steady Euler system} and for which the velocity field and domain are $2\pi/k$-periodic in $x$. 
\end{lemma}
\begin{proof}
Let $f$ be the unique holomorphic mapping on $\confD$ with imaginary part $w$ and such that $f - \id$ is $2\pi$-periodic in $\xi$. By the definition of $\mathcal{O}$ in \eqref{definition periodic O delta}, we see that $f$ is injective on $\partial\confD$. If~\ref{f is conformal case} holds, then $f$ is conformal on $\overline{\confD}$; we claim the same is true in the case~\ref{connected component case} by a continuity argument. 

Indeed, suppose that there is a continuous curve $\{ (u(s), w(s), Q(s), \lambda(s), \mu(s)) : s \in [0,1] \} \subset \mathcal{O}$ with $w(0) =0$ and $w(s) = w$. Denote by $f(s)$ the corresponding holomorphic mapping obtained from $w(s)$. Let $\mathcal{D}^0 \colonequals (-\pi,\pi) \times (-kh,0)$ be a single period of the conformal domain and consider the contour $\mathcal{C}$ corresponding to its boundary given positive orientation. For each $s \in [0,1]$, we can slightly deform the lateral boundaries of $\mathcal{C}$ so that they avoid any zeros of $\partial_\zeta f(s)$, and each side is a $2\pi$ translation of the other in the $\xi$ direction; call the resulting contour $\mathcal{C}(s)$.  By the definition of $\mathcal{O}$, we know that $\partial_\zeta f(s)$ is nonvanishing on the upper boundary $\Gamma_1^0$ of $\confD^0$. Since $w(s) + \eta + kh$ is harmonic in $\confD$, vanishes on $\Gamma_0$, and is bounded below by $\delta > 0$ on $\Gamma_1$ by the definition of $\mathcal{O}$, the maximum principle gives $w(s)+\eta+kh > 0$ in $\confD$. The Hopf lemma then implies $\partial_\eta (w(s) + \eta) > 0$ on $\Gamma_0$, and hence $|\partial_\zeta f(s)| > 0$ on the lower boundary $\Gamma_0^0$ of $\confD^0$, meaning $\mathcal{C}(s)$ contains no zeros of $\partial_\zeta f(s)$. 
On the other hand, $\partial_\zeta f(s)$ is bounded, so clearly it has no poles. By periodicity and the Cauchy argument principle, we therefore have that
\[
	\frac{1}{2\pi i} \int_{\Gamma_0^0} \frac{\partial_\zeta^2 f(s)}{\partial_\zeta f(s)} \, d\zeta - \frac{1}{2\pi i}\int_{\Gamma_1^0} \frac{\partial_\zeta^2 f(s)}{\partial_\zeta f(s)} \, d\zeta = \frac{1}{2\pi i}\int_{\mathcal{C}(s)} \frac{\partial_\zeta^2 f(s)}{\partial_\zeta f(s)} \, d\zeta = \# \{ \zeta \in \mathcal{D}^0 : \partial_\zeta f(s)(\zeta) = 0 \}. 
\]
The quantity on the far right-hand side is the winding number. As it is  integer-valued, by continuity it is independent of $s$. Evaluating the far left-hand side at $f(0) = \id$, we infer that the winding number vanishes for all $s \in [0,1]$, meaning that $\partial_\zeta f(s)$ is nonvanishing for all such $s$.  

Therefore, in both cases~\ref{f is conformal case} and \ref{connected component case}, the mapping $f$ is conformal in $\confD$.  As we have shown $f$ is injective on $\partial\confD$, the Picard--Darboux theorem now implies that $f$ is globally injective on $\overline{\confD}$. We may then define the physical domain $\fluidD$ and its boundary components by~\eqref{correspondence conformal and physical domains}, and set $\phi \colonequals (u + \lambda \xi) \circ f^{-1}$.

In the isothermal case $\gamma=1$, we have that $\alpha^2 = 1/\mu$ and $B = Q$, since $H(1) = 0$. For the polytropic case $\gamma >1$, the value of $\mu$ instead determines $B$, and to find $\alpha$ we use the relation
\[
	Q = B - H(1) = \frac{1}{\gamma-1} \left( \frac{1}{\mu} - \alpha^2 \right),
\]
which gives
\[
	\alpha^2 = \frac{1}{\mu} - (\gamma-1) Q.
\]
Note that this formula holds also in the isothermal case. 

We can now recover the density by inverting the enthalpy in Bernoulli's law:
\[
	\rho =  H^{-1}\left(B-\frac{1}{2} |\nabla \phi|^2 - \frac{1}{F^2} y\right) = \left\{
		\begin{aligned}
			\left( \frac{\gamma-1}{\alpha^2} (B -\frac{1}{2} |\nabla \phi|^2 - \frac{1}{F^2} y ) \right)^{\frac{1}{\gamma-1}} & \qquad \textrm{if } \gamma > 1 \\
			\exp{\left( \frac{1}{\alpha^2}(Q-\frac{1}{2} |\nabla \phi|^2 - \frac{1}{F^2} y ) \right)} & \qquad \textrm{if } \gamma = 1.
		\end{aligned}
		\right.
\]
The argument of $H^{-1}$ on the right-hand side above is positive by the definition of $\mathcal{O}$, and so this gives a positive density. The dynamic condition is therefore satisfied on $\fluidS$, as $\rho =  1$ there thanks to the Bernoulli condition $\G(u,w,Q,\lambda;\mu) = 0$. To see that this satisfies the continuity of mass equation, we rewrite $\F_1(u,w,Q,\lambda;\mu) = 0$ in the physical variables. If $\gamma = 1$, this yields
\[
	\Delta \phi = \mu \nabla \phi \cdot \nabla \left( \frac{1}{2} |\nabla \phi|^2 + \frac{1}{F^2} y \right) = -\frac{1}{\alpha^2} \nabla \phi \cdot \nabla H(\rho) = -\frac{1}{\rho} \nabla \phi \cdot \nabla \rho.
\]
The polytropic case is similar:
\begin{align*}
	\Delta \phi & = \frac{1}{(\gamma-1)B} \left( \nabla \phi \cdot \nabla + (\gamma-1) \Delta \phi \right) \left( B - H(\rho) \right) \\
		& =  \left(1-\frac{H(\rho)}{B} \right) \Delta \phi -\frac{1}{(\gamma-1) B}  \nabla \phi \cdot \frac{\nabla P(\rho)}{\rho}.
\end{align*}
Rearranging slightly and recalling the definition of $P$ in~\eqref{definition of P}, this becomes
\begin{align*}
	0 &  = \frac{\alpha^2}{(\gamma-1)B} \rho^{\gamma-1} \Delta \phi + \frac{\alpha^2}{(\gamma-1) B} \rho^{\gamma-2} \nabla \phi \cdot \nabla \rho,
\end{align*}
which is clearly equivalent to the conservation of mass. Finally, the conservation of momentum follows immediately from taking the gradient of Bernoulli's law and applying conservation of mass.
\end{proof}

\subsubsection*{Local bifurcation}

The well-studied incompressible problem is equivalent to solving the corresponding compressible problem $(\F,\G) = 0$ in the subspace where $u =0$ and $\mu = 0$. We say that $(0,w_0,Q_0,\lambda_0,0) \in \mathcal{O}$ represents a \emph{nondegenerate} incompressible periodic wave provided that 
\be
\label{definition non-degenerate periodic}
\begin{aligned}
	 (0,w_0,Q_0,\lambda_0,0) &\in  (\F_2,\G)^{-1}(0) \cap \mathcal{O}, \\
	 D_{(w,Q)} (\F_2,\G)(0,w_0,Q_0,\lambda_0,0) & \colon \Xspace_{2\per} \times \mathbb{R} \to \Yspace_{2\per} \times \Yspace_{3\per} \textrm{ is an isomorphism}.
\end{aligned}
\ee
The existence of global families of incompressible irrotational traveling waves originating at shear flows and extending to large amplitude is classical~\cite{toland1996stokes}, and this has more recently been extended to allow for vorticity, critical layers, stratification, surface tension, and other physical effects; see the survey~\cite{haziot2022traveling}. For concreteness, we include here a lemma outlining how this can be done using the conformal variable reformulation. As a basic consequence of the analytic global bifurcation theoretic construction, all but a discrete subset of this curve of solutions are nondegenerate in the sense of~\eqref{definition non-degenerate periodic}. 

\begin{lemma}[Incompressible periodic waves]
Fix a positive wave number $k$ and conformal mean depth $h > 0$.
\begin{enumerate}[label=\rm(\alph*\rm)]
	\item \label{CVS existence part} There exists a locally real-analytic curve $\km_{k,h}$ of solutions to the incompressible irrotational water wave problem admitting the global $C^0$ parameterization
	\be
	\label{global curve incompressible periodic}
		\km_{k,h} = \left\{ (0,w(s), Q(s),\lambda(s), 0) : s \in \mathbb{R} \right\} \subset (\F,\G)^{-1}(0) \cap \mathcal{O},
	\ee
	with 
	\[
		(0,w(0),Q(0),\lambda(0),0) = \left(0,0,\tfrac{(\lambda_1^*)^2}{2},\lambda_1^*,0 \right)
	\]
	being a uniform flow.
	\item There exists a discrete set $S \subset \mathbb{R}$ such that for all $s \in \mathbb{R} \setminus S$ we have that $(0,w(s),Q(s),\lambda(s),0)$ is non-degenerate~\eqref{definition non-degenerate periodic}.
\end{enumerate}

\end{lemma}
\begin{proof}
We only sketch the argument in the interest of keeping the presentation self-contained. First, observe that for any $\lambda$, we have the trivial solution 
\[
	(u, w, Q, \lambda,\mu) =  (0, 0, \tfrac{\lambda^2}{2}, \lambda, 0),
\]
which physically corresponds to an incompressible water wave for which the interface is completely flat and the velocity field is purely constant and horizontal with speed $\lambda$. Linearizing $(\F,\G)$ at this point yields the mapping
\be
\label{linearized F periodic}
	D_{(u,w,Q)} \F(0, 0, \tfrac{\lambda^2}{2},\lambda;0) 
		\begin{pmatrix} 
			\dot u \\
			\dot w \\
			\dot Q
		\end{pmatrix}
			=
		\begin{pmatrix}
			\Delta \dot u \\
			\Delta \dot w
		\end{pmatrix},
\ee
and
\be
\label{linearized G periodic}
		D_{(u,w,Q)} \G(0,0,\tfrac{\lambda^2}{2},\lambda;0) 
		\begin{pmatrix} 
			\dot u \\
			\dot w \\
			\dot Q
		\end{pmatrix}
			=  \left( -\lambda^2 \partial_\eta \dot w + \lambda \partial_\xi \dot u + \frac{1}{F^2} \dot w - \dot Q \right)\Big|_{\Gamma_1}.
\ee
Let 
\[
\begin{aligned}
	& \mathscr{H}(w, R; \lambda) 
	\\
	& \qquad \colonequals \left(\F_2(0,w,R+\tfrac{\lambda^2}{2},\lambda; 0),~ \G(0,w,R+\tfrac{\lambda^2}{2},\lambda; 0) \right) \colon \tilde{\mathcal{O}} \subset \Xspace_{2\per} \times \mathbb{R}^2 \to \Yspace_{2\per} \times \Yspace_{3\per},
\end{aligned}
\]
for the neighborhood
\[
	\tilde{\mathcal{O}} \colonequals \{ (w,R,\lambda) \in \Xspace_{2\per} \times \mathbb{R}^2 : (0,w,R+\tfrac{\lambda^2}{2},\lambda,0) \in \mathcal{O} \}.
\]
Each element $(w,R,\lambda)$ of the zero-set of $\mathscr{H}$ represents an incompressible periodic wave with $Q = R+\lambda^2/2$. Clearly, $\mathscr{H}$ is real analytic, and 
\[
	\mathscr{H}(0,0,\lambda) = 0 \quad \textrm{for all } \lambda \in \mathbb{R}, \qquad 
	D_{(w,R)} \mathscr{H}(0,0,\lambda) \begin{pmatrix} \dot w \\ \dot R \end{pmatrix}
		 = 
		 \begin{pmatrix} 
		 	\Delta \dot w \\
			\left( -\lambda^2 \partial_\eta \dot w + \frac{1}{F^2} \dot w -\dot R\right)|_{\Gamma_1}
		\end{pmatrix}.
\] 
Elliptic regularity implies that this is a Fredholm operator. To compute its index, we expand $\dot w$ as a Fourier cosine series. A routine calculation then shows that the dimension of the kernel is given exactly by the number of positive integer solutions $n$ to the so-called \emph{dispersion relation}:
\[ 
	0  = \mathfrak{d}(n,\lambda) \colonequals -\lambda^2 + \frac{1}{F^2} \frac{\tanh{(nkh)}}{n}. 
\]
For each $n \geq 1$, there exists a unique positive root $\lambda_n^*$ of $\mathfrak{d}(n,\placeholder)$. We compute that
\[
	\kernel{D_{(w,R)} \mathscr{H}(0,0,\lambda_n^*)} = \linspan{\{\dot w_n\}} \oplus \{ 0\}, \qquad \dot w_n(\xi,\eta) \colonequals \cos{(n\xi)} \sinh{\left( n(\eta+kh) \right)}.  
\]
Exploiting again the Fourier multiplier expression for $D_w \mathscr{H}(0,0,\lambda_n^*)$, it is readily confirmed that the range is characterized by 
\[
	\range{D_{(w,R)} \mathscr{H}(0,0,\lambda_n^*)} = \left\{ (Y_2, Y_3) \in \Yspace_{2\per} \times \Yspace_{3\per} : \int_{\confD^0} Y_2 \dot w_n \, d\xi \, d\eta + \frac{1}{(\lambda_n^*)^2} \int_{\Gamma_{1}^0} Y_3 \dot w_n \, d\xi = 0 \right\},
\]
where $\confD^0$ and $\Gamma_{1}^0$ are a single period of the conformal domain and the preimage of the free boundary, respectively. Thus, the range is codimension $1$ and so the linearized operator is Fredholm index $0$. The above characterization of the range also allows us to infer that
\[
	D_{(w,R)} D_\lambda \mathscr{H}(0,0,\lambda_n^*) \dot w_n = \begin{pmatrix} 0 \\ -2 \lambda_n^* \partial_\eta \dot w_n|_{\Gamma_1} \end{pmatrix} \not\in \range{D_{(w,R)} \mathscr{H}(0,0,\lambda_n^*)}.
\]
Applying the (analytic) Crandall--Rabinowitz theorem, we conclude that there exists a real-analytic curve $\tilde{\km}_n^\loc \subset \mathscr{H}^{-1}(0)$ so that locally the zero-set of $\mathscr{H}$ is the union of $\tilde\km_{n}^\loc$ and the trivial solution curve $\{ (0,0,\lambda) : \lambda \in \mathbb{R} \}$.  

We focus on the case $n=1$. Undoing the change of variables, the above argument gives a real-analytic curve 
\[
	\km_{k,h}^\loc \subset (\F,\G)(0,\placeholder, \placeholder, \placeholder;\, 0)^{-1}(0)
\]
of nontrivial solutions bifurcating from the trivial solution $(0,0,(\lambda_1^*)^2/2, \lambda_1^*,0)$. Perhaps shrinking the curve slightly, the linearized operator is invertible at each point except the point of bifurcation. Then, applying Theorem~\ref{homoclinic global ift} furnishes the global curve $\km_{k,h}$. Using Schauder estimates and the fact that the domain is compact, we can prove that $(\F,\G)$ is relatively locally proper on the zero-set and it is semi-Fredholm; this is essentially done in the more general compressible setting in Lemmas~\ref{local properness lemma} and \ref{periodic Fredholmness lemma}. Thus, the loss of compactness~\ref{K loss of compactness alternative} and loss of Fredholmness~\ref{K loss of fredholmness alternative} alternatives can be eliminated. The closed loop alternative~\ref{K loop} is ruled out through the analysis of nodal properties, leaving only the possibility of blowup~\ref{K blowup alternative}.
\end{proof}

We are now prepared to prove the main result of this section, namely that there exists a family of lightly compressible periodic waves bifurcating from each nondegenerate incompressible periodic wave, and these compressible solutions are given by the Janzen--Rayleigh expansion.

\begin{theorem}[Lightly compressible periodic waves]
\label{periodic JRE theorem}
Suppose that $(0,w_0,Q_0)$ is a solution to the periodic incompressible wave problem~\eqref{conformal governing equations periodic} that is nondegenerate \eqref{definition non-degenerate periodic} with wave number $k > 0$, conformal depth $h > 0$, and Froude number $F$. There exists a curve $\cm_\loc^\per$ of solutions to the compressible periodic wave problem \eqref{steady Euler system} with the real-analytic parameterization
\[
	\cm_\loc^\per = \left\{ (u^{\mu} , w^{\mu} , Q^\mu, \lambda_0,\mu) \in \Xspace_\per \times \mathbb{R}^3 : ~ |\mu| < \mu_0 \right\} \subset \mathcal{O} \cap (\F, \G)^{-1}(0),
\] 
for some $\mu_0 > 0$. The solutions along $\cm_\loc^\per$ admit the infinite power series representation 
\be
\label{JRE}
	u^\mu = \sum_{m \geq 1} \mu^m u_m, \qquad w^\mu = w_0 + \sum_{m \geq 1} \mu^m w_m, \qquad Q^\mu = Q_0 + \sum_{m \geq 1} \mu^m Q_m
\ee
where the leading-order terms $(u_1, w_1,Q_1)$ are the unique solutions in $\Xspace_\per \times \mathbb{R}$ to the homogeneous Poisson problem
\begin{equation}
    \label{periodic JRE u}
    \left\{
    \begin{aligned}
        \Delta u_1 & = \lambda_0 \partial_\xi \left( \frac{1}{2} \frac{\lambda_0^2}{|\nabla w_0+\mathbf{e}_2|^2} + \frac{1}{F^2}w_0 \right) & \qquad& \text{in } \confD \\
        \partial_\eta u_1 &=0 &\qquad & \text{on } \partial\confD
      \end{aligned} \right.
\end{equation}
and homogeneous Laplace equation with Robin condition 
\begin{equation}
\label{periodic JRE w}
      \left\{
      \begin{aligned}
        \Delta w_1 & =0 &\qquad & \text{in }\confD\\
        -\frac{\lambda_0^2( \nabla w_0+\mathbf{e}_2) \cdot\nabla w_1 }{|\nabla w_0+\mathbf{e}_2|^4}+\frac{1}{F^2} w_1 & = -\frac{\lambda_0 \partial_\xi u_1}{|\nabla w_0 + \mathbf{e}_2|^2} + Q_1 & \qquad & \text{on } \Gamma_1 \\
        w_1 & = 0 & \qquad & \textrm{on } \Gamma_0.
    \end{aligned}
    \right.
\end{equation}
\end{theorem}
\begin{proof}
It follows from the general formula~\eqref{linsystem} for the linearized operator (evaluated at $u = 0$, $\mu = 0$) and the assumption that $(0,w_0,Q_0,\lambda_0,0)$  is nondegenerate~\eqref{definition non-degenerate periodic} that
\begin{align*}
	\mathscr{L}_0 & \colonequals D_{(u,w,Q)} (\F,\G)(0,w_0,Q_0,\lambda_0,0) \\
	& =  
	\begin{pmatrix}
		\Delta_{\odd} & 0 & 0\\
		0 & \Delta_\even  & 0 \\
		\frac{\lambda_0}{|\nabla w_0+\mathbf{e}_2|^2} \partial_\xi|_{\Gamma_1} & \left( -\frac{\lambda_0^2}{|\nabla w_0+\frac{1}{kh}\mathbf{e}_2 |^4} (\nabla w_0+\frac{1}{kh}\mathbf{e}_2) \cdot \nabla + \frac{1}{F^2} \right)|_{\Gamma_1} & -1 
	\end{pmatrix},
\end{align*}
where 
\begin{align*}
	\Delta_\odd  \colonequals \Delta\Big|_{\Xspace_{1\per}}^{\Yspace_{1\per}} \qquad \textrm{and} \qquad 
	\Delta_{\even} \colonequals \Delta\Big|_{\Xspace_{2\per}}^{\Yspace_{2\per}}
\end{align*}
 are the Laplacian with the appropriate domain and codomain restrictions. In particular, $\Delta_\odd$ is invertible, the lower right block operator is $D_{(w,Q)}(\F_2,\G)$, which is invertible by the nondegeneracy assumption~\eqref{definition non-degenerate periodic}, and the lower left entry is clearly a bounded operator. Therefore
\[
	 D_{(u,w,Q)} (\F,\G)(0,w_0,Q_0,\lambda_0,0) \colon \Xspace_\per \times \mathbb{R} \to \Yspace_\per \quad \textrm{is an isomorphism.}
\]
The existence of the curve $\cm_\loc^\per$ follows simply from the real analytic implicit function theorem, which also justifies the power series expansion. To compute the first-order terms, we differentiate the equation in $\mu$ to see that
\[
\mathscr{L}_0	\begin{pmatrix}
		u_1 \\  w_1 \\ Q_1
	\end{pmatrix}
		= 
	- 
		\begin{pmatrix} 
			D_\mu\F(0,w_0,Q_0,\lambda_0,0) \\
			D_\mu\G(0,w_0,Q_0,\lambda_0,0) 
		\end{pmatrix}
		= 
			\begin{pmatrix}
				\lambda_0 \partial_\xi \left( \frac{1}{2} \frac{\lambda_0^2}{|\nabla w_0+\frac{1}{kh}\mathbf{e}_2|^2} + \frac{1}{F^2}w_0 \right) \\
				0 \\ 0
			\end{pmatrix}.
\]
This leads directly to the PDEs in~\eqref{periodic JRE u} and \eqref{periodic JRE w}.
\end{proof}

The above result allows us to take a single given nondegenerate incompressible periodic wave and do a local Janzen--Rayleigh expansion. But we can of course perform this same procedure at each point along the distinguished arcs of the global families $\km_{k,h}$, resulting in a collection of $2$-manifolds.

\begin{corollary}[Manifolds of lightly compressible periodic waves] 
\label{periodic JRE manifold corollary}
Let $\{\mathscr{A}_j\}$ be the (at most countable) distinguished arcs of the global family of incompressible periodic water waves $\km_{k,h}$ given in~\eqref{global curve incompressible periodic}. There exists a corresponding at most countable ordered list $\{ \mathscr{M}_j\} \subset \mathcal{O}$ of two-dimensional real-analytic manifolds of periodic compressible water waves. Each $\mathscr{M}_j$ has the structure of a (real-analytic) fiber bundle over the base $\mathscr{A}_j$.
\end{corollary}
\begin{proof}
Consider a distinguished arc $\mathscr{A}_j$ of $\km_{k,h}$, which in view of Theorem~\ref{homoclinic global ift} admits the locally real-analytic parameterization 
\[
	\mathscr{A}_j = \{ (0, w(s), Q(s), \lambda(s), 0) : s \in (s_j, s_{j+1}) \}
\]
for some $0 < s_j < s_{j+1} \leq \infty$. By definition, each solution on $\mathscr{A}_j$ is nondegenerate in the sense of~\eqref{definition non-degenerate periodic}. Applying Theorem~\ref{periodic JRE theorem}, we conclude that there exists a two-dimensional manifold $\mathscr{M}_j$ of lightly compressible solutions which in the neighborhood of $\mathcal{N}_*$ of $(0,w(s_*), Q(s_*), \lambda(s_*), 0) \in \km_{k,h}$ has the  real-analytic coordinate chart:
\[
	\mathscr{M}_j \cap \mathcal{N}_* = \{ (u(s,\mu), w(s,\mu), Q(s,\mu), \lambda(s), \mu) : |s - s_*| + |\mu| \ll 1\} \subset (\F,\G)^{-1}(0),
\] 
where
\[
	u(s,\mu) = \sum_{m \geq 1} \mu^m u_m(s), \quad w(s,\mu) = w(s) + \sum_{m \geq 1} \mu^m w_m(s), \quad Q(s,\mu) = Q(s) + \sum_{m \geq 1} \mu^m Q_m(s),
\]
and each of the terms in the sum are obtained via the Janzen--Rayleigh expansion with base point $(0,w(s),Q(s), \lambda(s), 0)$.  This is clearly a fiber bundle over $\mathscr{A}_j$ with the trivialization map given by projection to $\mu = 0$.
\end{proof}

\subsection{Solitary waves} 

We next turn to the solitary wave case, which only requires we make an appropriate change to the functional analytic setting.

\subsubsection*{Conformal reformulation}

\begin{figure}
	\includegraphics{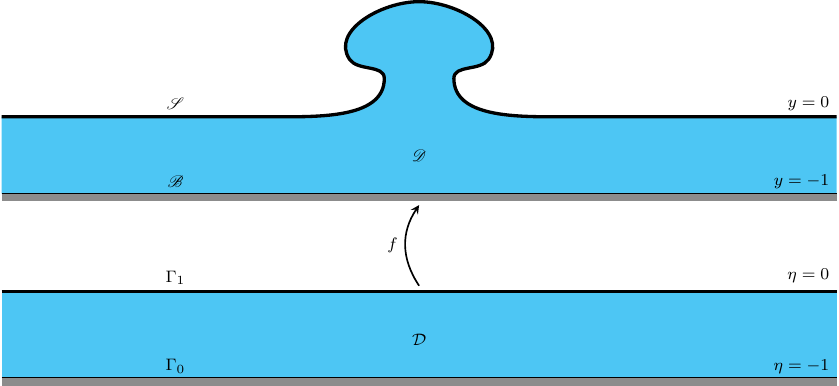}
	\caption{The physical domain $\fluidD$ in the $(x,y)$-plane and conformal domain $\confD$ in the $\zeta$-plane for solitary waves. Again, the conformal formulation of the problem allows for the free surface $\fluidS$ to be overhanging.}
	\label{solitary setup figure}
\end{figure}

Assume that we have an incompressible and irrotational solitary wave, meaning a traveling wave solution to the Euler equations~\eqref{steady Euler system} with $\mu=0$ and such that $\fluidS$ limits asymptotically as $|x| \to \infty$ to a straight line parallel to $\fluidB$. Thanks to our nondimensionalization~\eqref{nondimensionalization space and velocity}, the limiting depth is $1$.  As before, we require that $\fluidS$ is a globally injective curve with uniform $C^{2+\beta}$ regularity.  Reusing notation, we take the conformal domain to be the slab region 
\be
\label{solitary conformal domain definition}
	\confD \colonequals \{ \zeta = \xi + i \eta \in \mathbb{C} : -1 < \eta < 0 \},
\ee
with boundary components $\Gamma_0 \colonequals \{ \eta = -1\}$ and $\Gamma_1 \colonequals \{ \eta =0 \}$. There exists a (unique) holomorphic mapping $f \colon \confD \to \fluidD$ that admits an extension to $C^{2+\beta}(\overline{\confD})$, is conformal on $\overline{\confD}$, and satisfies
\[
	f(\confD) = \fluidD, \qquad f(\Gamma_0) = \fluidB, \qquad  f(\Gamma_1) = \fluidS, \qquad f(\zeta) - \zeta \to 0 \quad \textrm{as } \realpart{\zeta} \to \pm\infty.
\]
The requirement that $\fluidS$ limit to the horizontal line $\{ y = 0 \}$ means that
\be
\label{asymptotic condition w solitary wave}
	\lim_{\xi \to \pm\infty} \imagpart{f(\xi, 0)} = 0.
\ee 
See Figure~\ref{solitary setup figure} for an illustration.

We treat as the new unknowns
\[
	\phi \circ f \equalscolon u + \xi, \qquad w \colonequals \imagpart{(f-\id)}. 
\]
As in the periodic case, $w$ allows us to reconstruct $f$, whereas $u$ measures the deviations of the relative potential in the conformal variables from its far-field limit. Arguing as before, we find that the steady problem~\eqref{steady Euler system} written now in the conformal strip becomes the elliptic system
\bse
\label{conformal governing equations solitary}
\begin{equation}
\label{qpde solitary}
\left\{
\begin{aligned}
	\Delta u & =  \mu \left(  (\gamma-1) \Delta u + (\nabla u + \mathbf{e}_1)  \cdot \nabla \right) \left( \frac{1}{2} \frac{|\nabla u + \mathbf{e}_1|^2}{|\nabla w+\mathbf{e}_2|^2} + \frac{1}{F^2} (w+\eta)  \right)  \\
	\Delta w & = 0
\end{aligned}
\right.
\qquad \textrm{in } \confD.
\end{equation}
together with the boundary conditions
\begin{equation}
\label{boundary conditions solitary}
\left\{ \begin{aligned}
	\partial_\eta u & = 0 & \qquad & \textrm{on } \Gamma_0 \cup \Gamma_1 \\
	\frac{1}{2} \frac{| \nabla u + \mathbf{e}_1|^2}{|\nabla w+\mathbf{e}_2|^2} + \frac{1}{F^2}w &= \frac{1}{2} &\qquad & \textrm{on } \Gamma_1 \\
	w & =0 &\qquad& \text{on } \Gamma_0.
\end{aligned}
\right.
 \end{equation}
The first two of these correspond to the kinematic and dynamic conditions, respectively, while the third ensures $\Gamma_0$ is mapped to the physical bed at $\{ y=-1\}$. Note that we have explicitly evaluated the Bernoulli constant using the limiting values of $u$ and $w$, namely:
\be
\label{conformal far-field condition}
	\lim_{\xi \to \pm\infty} u(\xi,\eta) = 0, \qquad \lim_{\xi \to \pm\infty}  w(\xi,\eta)  = 0,
\ee  
where both limits are uniform in $\eta$. In contrast to the periodic case,  the Bernoulli constant here is determined by the far-field conditions, and hence not a varying parameter. In its place, we allow the Froude number $F$ to vary, as this agrees with the typical construction of incompressible solitary waves. 
\ese

\subsubsection*{Abstract operator equation}

The procedure for the solitary wave case is largely the same as the periodic case, except that now we work in Banach spaces of functions defined on the unbounded strip $\confD$ defined in~\eqref{solitary conformal domain definition} and with appropriately prescribed behavior at infinity. Specifically, let 
\be
\label{definition X space and Y space}
\begin{aligned}
	\Xspace_\bdd & \colonequals \Big\{ (u, w) \in C_{\bdd,\odd,0}^{2+\beta}(\overline{\confD}) \times  C_{\bdd,\even,0}^{2+\beta}(\overline{\confD}) : ~
		\partial_\eta u|_{\Gamma_0\cup\Gamma_1}=0, ~w|_{\Gamma_0}=0   \Big\}, \\
		& \equalscolon \Xspace_{1\bdd} \times \Xspace_{2\bdd} \\
	\Yspace_{\bdd} &\colonequals  C_{\bdd,\odd,0}^{\beta}(\overline{\confD})\times C_{\bdd,\even,0}^{\beta}(\overline{\confD}) \times C_{\bdd,\even,0}^{1+\beta}(\Gamma_1)  \equalscolon \Yspace_{1\bdd} \times\Yspace_{2\bdd} \times \Yspace_{3\bdd}.
\end{aligned}
\ee
Here, $\beta\in(0,1)$ is fixed but arbitrary, and the subscript ``b'' indicates that the corresponding H\"older norms are finite. The subscript $0$ indicates that the functions vanish as $|\xi| \to \infty$ and uniformly in $\eta$, which ensures that the conformal mapping limits to the identity upstream and downstream and the background velocity field is uniform~\eqref{conformal far-field condition}.  

The compressible solitary wave problem \eqref{steady Euler system} can then be written as the abstract operator equation
\begin{equation}
\label{meqn}
\left\{
\begin{aligned}
      \F(u,w,F; \mu) &= 0 \\
      \G(u,w,F; \mu) &= 0, 
\end{aligned}
\right.
\end{equation}
where $(\F,\G) : \mathcal{O}_\bdd \subset \Xspace_\bdd \times \mathbb{R}^2  \to \Yspace_\bdd$ is the real-analytic mapping with $\F$ given by
\be
\label{definition abstract F solitary}
	\F(u,w,F; \mu) \colonequals
		\begin{pmatrix}
			\Delta u - \mu \left(  (\gamma-1) \Delta u + (\nabla u  + \mathbf{e}_1) \cdot \nabla  \right)  \left( \frac{1}{2} \frac{|\nabla u + \mathbf{e}_1|^2}{|\nabla w+\mathbf{e}_2|^2} + \frac{1}{F^2} (w+\eta) \right) \\
			\Delta w
		\end{pmatrix},
\ee
and
\[
	\G(u,w,F; \mu) \colonequals \left( \frac{1}{2} \frac{|\nabla u +  \mathbf{e}_1|^2}{|\nabla w+\mathbf{e}_2|^2} + \frac{1}{F^2} w - \frac{1}{2}  \right)\Big|_{\Gamma_1}.
\]
Again, the open set $\mathcal{O}_\bdd$ is chosen so that it quantitatively implies ellipticity holds and that  $f$ is conformal and injective on the boundary. We do this by defining for each $\delta > 0$ the open set $\mathcal{O}_\delta^\bdd$ consisting of $(u,w,F,\mu) \in \Xspace_\bdd \times \mathbb{R}^2$ satisfying the conditionals in~\eqref{definition periodic O delta}, then taking 
$\mathcal{O}_\bdd \colonequals \bigcup_{\delta > 0} \mathcal{O}_\delta^\bdd$.

\begin{lemma}[Equivalence for solitary waves]
Fix $\gamma \geq 1$ and suppose that $(u,w,F,\mu) \in \mathcal{O}_\bdd$ satisfies the abstract operator equation~\eqref{meqn} with $\mu > 0$. Assume further that either
\begin{enumerate}[label=\rm(\roman*)]
	\item \label{solitary f is conformal case} $\inf_{\confD} |\nabla w+\mathbf{e}_2| > 0$ or 
	\item \label{solitary connected component case} $(u,w,F,\mu)$ is in a connected component of $\mathcal{O}_\bdd$ containing a laminar solution $(u_*,0,F_*,\mu_*)$.
\end{enumerate}
There exists $\alpha > 0$, a $C^{2+\beta}$  fluid domain $\fluidD$, and relative velocity potential $\phi \in C_\loc^{2+\beta}(\overline{\fluidD})$ solving the (nondimensionalized) steady isentropic Euler equations~\eqref{steady Euler system}. This is a solitary wave in the sense that the free boundary approaches the line $\{y = 0\}$ asymptotically and $\nabla\phi \to \mathbf{e}_1$ as $x \to \pm\infty$. 

\end{lemma}
\begin{proof}
This follows as in the periodic case detailed in Lemma~\ref{periodic equivalence lemma}, recalling that the far-field conditions imply that the Bernoulli constant $Q = 1/2$. 
\end{proof}

\subsubsection*{Local bifurcation}

Call a solution of the incompressible solitary wave problem \emph{nondegenerate} provided it satisfies 
\be
\label{definition non-degenerate solitary}
\begin{aligned}
	 (0,w_0,F,0) &\in  (\F_2,\G)^{-1}(0) \cap \mathcal{O}_\bdd, \\
	 D_{w} (\F_2,\G)(0,w_0,F,0) & \colon \Xspace_{2\bdd}  \to \Yspace_{2\bdd} \times \Yspace_{3\bdd}  \textrm{ is an isomorphism}.
\end{aligned}
\ee
For bifurcation curves constructed through analytic global continuation, nondegeneracy is generic as it holds outside a discrete set. This well-known fact is recorded in the lemma below.

\begin{lemma}[Incompressible solitary waves]
	There exists a locally real-analytic curve $\km_\sol$ of solutions to the incompressible irrotational solitary water wave problem admitting the global $C^0$ parameterization
	\be
	\label{global curve incompressible solitary}
		\km_{\sol} = \left\{ (0,w(s), F(s), 0) : s \in (0,\infty) \right\} \subset (\F,\G)^{-1}(0) \cap \mathcal{O}_\bdd,
	\ee
	with $w(s) \to 0$ and $F(s) \searrow 1$ as $s \searrow 0$.  There exists a discrete set $S \subset \mathbb{R}$ such that for all $s \in (0,\infty)\setminus S$ we have that $(0,w(s),F(s),0)$ is nondegenerate~\eqref{definition non-degenerate solitary}.
\end{lemma}
\begin{proof}
Versions of this result have been obtained through a variety of means, and in much more general physical settings, by many authors. For a survey of this literature, see \cite[Section 5]{haziot2022traveling}. We mention that our formulation is inspired by the recent work of Haziot and Wheeler~\cite{haziot2023large} on (incompressible) solitary waves with constant vorticity. Again, we only briefly outline the argument in the interests of keeping the paper somewhat self-contained. 

Consider the trivial solution $(0,0,F,0)$. The linearized operator there is generally \emph{not} Fredholm in this case due to the unboundedness of the domain. In place of local bifurcation theory, one can construct small-amplitude solutions via a center manifold reduction; see, for example~\cite{dias2003handbook}. The result of this analysis is a curve $\km_\loc^\sol$ that contains in its closure the trivial solution. Furthermore, inspecting the behavior of the linearized operator along $\km_\loc^\sol$, it is possible to verify that sufficiently small-amplitude solutions that are nontrivial are in fact nondegenerate~\eqref{definition non-degenerate solitary}. The basic idea is that the linearized system is in one-to-one correspondence with the linearization of the reduced system on the center manifold; see, for example, \cite[Theorem 4.1]{wheeler2013solitary}. We may then apply the global implicit function theorem as in Theorem~\ref{homoclinic global ift} to obtain the global curve $\km_\sol$. 
\end{proof}

\begin{theorem}[Lightly compressible solitary waves]
\label{solitary JRE theorem}
Suppose that $(0,w_0)$ is a solution to the incompressible solitary wave~\eqref{conformal governing equations solitary} for a Froude number $F > 1$ and that is nondegenerate~\eqref{definition non-degenerate solitary}. There exists a curve $\cm_\loc^\sol$ of solutions to the compressible solitary wave problem \eqref{steady Euler system} with the real-analytic parameterization
\[
	\cm_\loc^\sol = \left\{ (u^\mu, w^\mu, \mu) \in \Xspace_\bdd \times \mathbb{R} : |\mu| < \mu_0 \right\} \subset \mathcal{O}_\bdd \cap (\F, \G)^{-1}(0),
\] 
for some $\mu_0 > 0$. The solutions along $\cm_\loc^\sol$ admit the infinite power series representation~\eqref{JRE} 
\be
\label{JRE solitary}
	u^\mu = \sum_{m \geq 1} \mu^m u_m, \qquad w^\mu = w_0 + \sum_{m \geq 1} \mu^m w_m,
\ee
where $F$ is fixed along the curve and the leading-order terms $(u_1, w_1)$ are the unique solutions in $\Xspace_\bdd$ of~\eqref{periodic JRE u} and \eqref{periodic JRE w} (with $\lambda_0 = 1$ and $Q_1 = 0$).  
\end{theorem}
\begin{proof}
This follows from a straightforward adaptation of the proof of Theorem~\ref{periodic JRE theorem}. \end{proof}

\section{Subsonic compressible periodic water waves}
\label{global bifurcation section}

Our objective in this section is to make an analytic global continuation of the curves of lightly compressible periodic waves constructed in Section~\ref{lightly compressible periodic waves section}. These families extend beyond the perturbative Mach number regime and up until either a singularity develops or the curve returns to the point of bifurcation. Our main result is as follows.

\begin{theorem}[Subsonic periodic waves]
\label{periodic global bifurcation theorem}
Fix a wave number $k > 0$, conformal depth $h > 0$, Froude number $F > 0$, and a corresponding nondegenerate solution $(0,w_0,Q_0)$ to the periodic incompressible water wave problem.  Let $\cm_\loc^\per$ be the real-analytic curve of lightly compressible periodic waves furnished by Theorem~\ref{periodic JRE theorem}. There exists a curve $\cm_\per$ of compressible periodic water waves admitting the $C^0$ parameterization
\[
	\cm_\per = \left\{ (u(s), w(s), Q(s),\lambda_0, \mu(s) ) : s \in [0,\infty) \right\} \subset (\F,\G)^{-1} (0) 
\]
and satisfying the following.
\begin{enumerate}[label=\rm(\alph*)]
	\item \label{periodic extension part} $\cm_\loc^\per \subset \cm_\per$ and $(u(0), w(0), Q(0), \lambda_0, \mu(0)) = (0,w_0,Q_0,\lambda_0, 0)$.
	\item \label{periodic physical part} Every solution on $\cm_\per$ is physical in that the corresponding mapping $f(s)$ is conformal and injective on $\overline{\confD}$.
	\item \label{periodic alternatives part} In the limit $s \to \infty$, one of the following alternatives must occur:
		\begin{enumerate}[label=\rm(A\arabic*)]
			\item \label{periodic unboundedness alternative} \textup{(Unboundedness)} The quantity
			\[
				 \|u(s) \|_{C_\per^{2+\beta}(\confD)} + \| w(s) \|_{C_\per^{2+\beta}(\confD)} + |\mu(s)| \to \infty. 
			\]
			\item \label{boundary degeneracy alternative} \textup{(Boundary degeneracy)} Either
			\[
				\dist{(\fluidS(s), \fluidB)} \to 0 \qquad \textrm{or} \qquad \lim_{s\to\infty} \sup_{\substack{\zeta, \zeta^\prime \in \Gamma_1 \\ \zeta \neq \zeta^\prime}} \frac{|\zeta-\zeta^\prime|}{|f(\zeta) - f(\zeta^\prime)|} = \infty,
			\]
			where $f(s)$ is the conformal mapping derived from $w(s)$ and $\fluidS(s)$ is the corresponding free boundary in the physical domain.

			\item \label{transsonic limit alternative} \textup{(Transsonic flow)} At a point in $\confD$, the flow approaches transsonic:
			\be
			\label{transsonic limit}
		\lim_{s \to \infty} \inf_{\confD} \left(c(s)^2 - \frac{|\nabla u(s) + \lambda_0 \mathbf{e}_1|^2}{|\nabla w(s) + \frac{1}{kh}\mathbf{e}_2|^2}  \right) = 0.
			\ee
			\item \label{stagnation alternative} \textup{(Stagnation)} A relative stagnation point forms on the free surface:
			\[
				\lim_{s \to \infty} \inf_{\Gamma_1} |\partial_\xi u(s) + \lambda_0| = 0.
			\]

			\item \label{periodic loop alternative} \textup{(Closed loop)} There exists $T > 0$ such that 
			\[
				\left(u(s+T),\, w(s+T),\, Q(s+T),\, \mu(s+T) \right)= \left(u(s),\, w(s), \, Q(s),\, \mu(s) \right)
			\]
			for all $ s \in [0,\infty)$. If this occurs, then necessarily $\cm_\per$ contains two distinct incompressible solutions.
		\end{enumerate}
	\item At each parameter value $s$, $\cm_\per$ admits a local reparameterization that is real analytic. 
\end{enumerate}
\end{theorem}

Over the next subsections, we will show that the more expansive set of alternatives in the general setting of Theorem~\ref{homoclinic global ift} can be reduced down to those above. To slightly simplify the formulas, throughout the rest of the section we simply take $kh = 1$. The general case can be obtained through a trivial modification of the argument.

\subsection{Linear a priori estimates}
\label{ellipticity section}

We first establish linear Schauder-type estimates for the linearization of the abstract operator $(\F,\G)$ from~\eqref{definition abstract F} at a solution $(u,w,Q,\lambda;\mu)\in\mathcal{O}$. The main objective is to verify that the corresponding system is elliptic in the Agmon--Douglis--Nirenberg (ADN) sense and satisfies the Lopatinskii (complementing) condition, and to derive the resulting Schauder estimates that rule out the loss of compactness and Fredholmness alternatives in the global continuation. While we carry this out for the water wave problem at hand, the argument uses only the structure of the linearized system, and therefore extends to a wider class of quasilinear elliptic systems sharing the same features; see Remark~\ref{generality remark}.

To lighten the notation, throughout this section we write
\[
    \vfield \colonequals \nabla u + \lambda \mathbf{e}_1,
    \qquad
    \qfield \colonequals \nabla w + \mathbf{e}_2,
\]
and introduce the \emph{Bernoulli function}
\be
\label{definition Bernoulli function}
    \mathfrak{B} = \mathfrak{B}(u,w)
    \colonequals \frac{1}{2} \frac{|\vfield|^2}{|\qfield|^2}
    + \frac{w+\eta}{F^2},
\ee
in terms of which
\[
    \F(u,w,Q,\lambda;\mu)
    = \big( 1 - \mu(\gamma-1) \mathfrak{B} \big) \Delta u
    - \mu\, \vfield \cdot \nabla \mathfrak{B},
    \qquad
    \G(u,w,Q,\lambda;\mu)
    = \big( \mathfrak{B} - Q \big)\big|_{\Gamma_1}.
\]
Observe also that Bernoulli's law~\eqref{conformal Bernoulli law}
states precisely that $\mathfrak{B} = B - H(\varrho)$ in
$\overline{\confD}$ whenever $(u,w,Q,\lambda;\mu)$ is in the
zero-set of $(\F,\G)$.

Differentiating $\mathfrak{B}$ at $(u,w)$ in the direction
$(\dot u, \dot w)$ yields the \emph{linearized Bernoulli form}
\be
\label{linearized Bernoulli form}
    \dot{\mathfrak{B}} \colonequals
    D_{(u,w)} \mathfrak{B}(u,w)
    \begin{pmatrix} \dot u \\ \dot w \end{pmatrix}
    = \frac{\vfield \cdot \nabla \dot u}{|\qfield|^2}
    - |\vfield|^2\, \frac{\qfield \cdot \nabla \dot w}{|\qfield|^4}
    + \frac{\dot w}{F^2},
\ee
and hence the linearized operator at $(u,w,Q,\lambda;\mu)$ is
given by
\bse
\label{linsystem}
\begin{align}
    D_{(u,w,Q)} \F
    \begin{pmatrix} \dot u \\ \dot w \\ \dot Q\end{pmatrix}
    & =
    \begin{pmatrix}
        \big(1 - \mu(\gamma-1)\mathfrak{B}\big) \Delta \dot u
        - \mu(\gamma-1)\, \Delta u \; \dot{\mathfrak{B}}
        - \mu\, \nabla \mathfrak{B} \cdot \nabla \dot u
        - \mu\, \vfield \cdot \nabla  \dot{\mathfrak{B}}  \\
        \Delta \dot w
    \end{pmatrix}, \label{Lqpde}\\
    D_{(u,w,Q)} \G
    \begin{pmatrix} \dot u \\ \dot w \\ \dot Q\end{pmatrix}
    & = \big( \dot{\mathfrak{B}}  - \dot Q\big)\Big|_{\Gamma_1}.
    \label{Lcbernoulli}
\end{align}
\ese

 Note that in~\eqref{Lqpde}, the final two terms involve a second-order coupling between $\dot u$ and $\dot w$; this is one of the main technical obstacles for the compressible case.

Reading off the second-order part of~\eqref{Lqpde}, the principal
symbol of the interior operator is the upper triangular matrix
\[
    A^0(\boldsymbol{\xi},\boldsymbol{\Xi})
    = \begin{pmatrix} A_{11}^0 & A_{12}^0 \\ 0 & A_{22}^0
    \end{pmatrix},
    \qquad \textrm{where} \quad  \boldsymbol{\Xi} = (\Xi_1, \Xi_2) \in \mathbb{R}^2,
\]
with entries
\be
\label{linsys:coeffs}
    A_{11}^0  =  \big(1-\mu(\gamma-1)\mathfrak{B}\big)
    |\boldsymbol{\Xi}|^2
    - \mu\, \frac{(\vfield \cdot \boldsymbol{\Xi})^2}{|\qfield|^2},
    \quad
    A_{12}^0 = \mu\, |\vfield|^2\,
    \frac{(\vfield\cdot\boldsymbol{\Xi})
    (\qfield\cdot\boldsymbol{\Xi})}{|\qfield|^4},
    \quad
    A_{22}^0  = |\boldsymbol{\Xi}|^2.
\ee
It is convenient to introduce the \emph{ellipticity function}
$\ellfunc = \ellfunc(u,w,\lambda; \mu)$ defined by
\be
\label{definition ellipticity function}
    \ellfunc \colonequals
    1 - \mu(\gamma-1)\mathfrak{B}
    - \mu \frac{|\vfield|^2}{|\qfield|^2}
    = 1 - \mu \left( \frac{\gamma+1}{2}
    \frac{|\vfield|^2}{|\qfield|^2}
    + \frac{\gamma-1}{F^2}(w+\eta) \right).
\ee
Since $A^0$ is triangular, $\det A^0 = |\boldsymbol{\Xi}|^2 A_{11}^0$, and, as $\mu \geq 0$, the Cauchy--Schwarz inequality gives
\be
\label{sharp ellipticity bound}
    A_{11}^0(\boldsymbol{\xi}, \boldsymbol{\Xi})
    \geq \ellfunc \, |\boldsymbol{\Xi}|^2,
\ee
with equality when $\boldsymbol{\Xi}$ and $\vfield$ are
collinear. The bound~\eqref{sharp ellipticity bound} is therefore
sharp, and uniform ADN ellipticity of the linearized system is
\emph{equivalent} to the interior condition
\be
\label{interior ellipticity condition}
    \inf_{\confD} \ellfunc > 0,
\ee
which is exactly the quantity controlled by the third condition
in the definition~\eqref{definition periodic O delta} of
$\mathcal{O}_\delta$. Since
$|\vfield|^2 = (\partial_\xi u + \lambda)^2
+ (\partial_\eta u)^2$, we have the identity
\be
\label{interior boundary identity}
    1-\mu(\gamma-1)\mathfrak{B}
    - \mu \frac{(\partial_\xi u+\lambda)^2}{|\qfield|^2}
    = \ellfunc + \mu \frac{(\partial_\eta u)^2}{|\qfield|^2},
\ee
so the coefficient of $\Xi_1^2$ in $A_{11}^0$ dominates
$\ellfunc$ pointwise, with equality precisely on
$\partial\confD$ where $\partial_\eta u = 0$.

\begin{lemma}[Subsonic implies ellipticity]
\label{subsonic ADN lemma}
Suppose that $(u,w,Q,\lambda,\mu) \in (\F,\G)^{-1}(0)$
with $\mu > 0$ is subsonic in that
\[
    \inf_{\confD} \left(c^2
    - \frac{|\nabla u+\lambda \mathbf{e}_1|^2}
    {|\nabla w + \mathbf{e}_2|^2} \right) > 0.
\]
Then, the corresponding linearized operator
$D_{(u,w)} (\F,\G)(u,w,Q,\lambda; \mu)$ is elliptic in
the ADN sense.
\end{lemma}
\begin{proof}
Translating back to the physical variables, we have shown
that ellipticity in the ADN sense holds provided
\be
\label{ADN ellipticity condition physical}
    \frac{1}{\mu} > \frac{\gamma+1}{2} |\nabla \phi|^2
    + \frac{\gamma-1}{F^2} y
    = |\nabla \phi|^2
    + (\gamma-1) \left( \frac{1}{2} |\nabla \phi|^2
    + \frac{1}{F^2} y \right).
\ee
First consider the case $\gamma = 1$, so that in light of
the definition of $\mu$ in~\eqref{definition Mach number},
this becomes $\alpha^2 > |\nabla \phi|^2$.
Recall that for isothermal flow, $\alpha^2 = c^2$ is the
(constant) speed of sound. Likewise, when $\gamma > 1$, we
find instead from Bernoulli's law~\eqref{nondimensional
Bernoulli law} that the
condition~\eqref{ADN ellipticity condition physical} can be
rewritten as $|\nabla \phi|^2 < c^2$.
\end{proof}

On $\Gamma_1$, where $\partial_\eta u = 0$, the quantity
$\frac{1}{2} (\partial_\xi u + \lambda)^2/|\qfield|^2$ is
half the squared relative fluid speed, and by the kinematic
condition the relative velocity is tangent to the free
surface. Thus $\partial_\xi u + \lambda = 0$ at a point of
$\Gamma_1$ means that the relative velocity vanishes there
entirely, meaning there is a (relative) \emph{stagnation point}. We emphasize that
the absence of stagnation is \emph{not} implied by
ellipticity: at $\mu = 0$ one has $\ellfunc \equiv 1$ while
$\partial_\xi u + \lambda$ may still vanish. As we will see
in Proposition~\ref{lopat prop}, it is the stagnation set
--- and not the ellipticity function $\ellfunc$ --- that
governs the boundary (Lopatinskii) condition.

\subsubsection*{Proper ellipticity and the Lopatinskii condition}

Fix a background state
$(u,w,\lambda,Q,\mu) \in (\F,\G)^{-1}(0)$ with $\mu \geq 0$.
The Schauder theory we are after gives a priori estimates in Hölder norms for the linear equation problem
\be
\label{genlinsystem}
\left\{
\begin{aligned}
    D_{(u,w,Q)}\F_1(u,w,\lambda, Q; \mu) (\dot u, \dot w, \dot Q)^T
    & = h_1 & \qquad & \textrm{in } \confD, \\
    \Delta \dot w & = h_2
    & \qquad & \textrm{in } \confD, \\
    \dot{\mathfrak{B}}(\dot u, \dot w) & = f
    & \qquad & \textrm{on } \Gamma_1, \\
    \partial_\eta \dot u & = g
    & \qquad & \textrm{on } \Gamma_1, \\
    \partial_\eta \dot u = 0, \quad \dot w & = 0
    & \qquad & \textrm{on } \Gamma_0,
\end{aligned}
\right.
\ee
where $D_{(u,w,Q)}\F_1$ denotes the first component
of~\eqref{Lqpde}. Its coefficients lie in
$C^{\beta}_\per(\overline{\confD})$ with bounds depending on $|Q|$ and $\| (u,w)\|_{C_\per^{2+\beta}}$, and its
principal symbol row is $(A_{11}^0, A_{12}^0)$
from~\eqref{linsys:coeffs}. We assign Douglis--Nirenberg
weights $s_1 = s_2 = 0$ to the two equations, and $t_1 = t_2 = 2$ to the two unknowns $\dot u$, $\dot w$; on $\Gamma_1$ both boundary operators are first
order, corresponding to weights $r_1 = r_2 = -1$, while on
$\Gamma_0$ we take $r_1 = -1$ and $r_2 = -2$. By the sharp
lower bound~\eqref{sharp ellipticity bound}, it follows that
\be
\label{gen:elliptic}
    \det A^0(\boldsymbol{\xi},\boldsymbol{\Xi})
    = |\boldsymbol{\Xi}|^2\, A_{11}^0(\boldsymbol{\xi},
    \boldsymbol{\Xi})
    \geq \Big( \inf_{\confD} \ellfunc \Big)\,
    |\boldsymbol{\Xi}|^4,
\ee
so the interior condition~\eqref{interior ellipticity
condition} implies that the system is uniformly elliptic in
the Douglis--Nirenberg sense.

To verify proper ellipticity, fix a point of
$\partial\confD$, where $\partial_\eta u = 0$, and write
$\tau$ and $\nu$ for the frequencies dual to the tangential
variable $\xi$ and the normal variable $\eta$,
respectively. Introduce
\be
\label{definition kappa}
    \kappa \colonequals \sqrt{
    \frac{1-\mu(\gamma-1)\mathfrak{B}}{\ellfunc}} \geq 1.
\ee

The symbol $A_{11}^0$ at this point reduces to
$\ellfunc \tau^2 + \big( 1 - \mu(\gamma-1)\mathfrak{B}\big)
\nu^2$, so for fixed $\tau \neq 0$ the roots $\nu \in
\mathbb{C}$ of $\det A^0(\tau, \placeholder)$ are
\be
\label{proper ellipticity roots}
    \nu = \pm i \tau, \qquad
    \nu = \pm \frac{i \tau}{\kappa}.
\ee
All four roots are purely imaginary, and for each
$\tau \neq 0$ exactly two lie in the upper half-plane;
the system is therefore properly elliptic. We record the
identity
\be
\label{mu c2 identity}
    \mu (\partial_\xi u + \lambda)^2
    = \ellfunc  (\kappa^2 - 1)\, |\qfield|^2
    \qquad \textrm{on } \partial\confD,
\ee
which follows directly from the definitions of $\ellfunc$
and $\kappa$.

Fix a point of $\Gamma_1$ and freeze the coefficients there. To
shorten the formulas, write $\sigma \colonequals \partial_\xi u +
\lambda$, and recall that $|\vfield|^2 = \sigma^2$ on $\Gamma_1$.
Using~\eqref{mu c2 identity}, the principal symbol at this point
becomes
\[
	A^0(\tau,\nu) = \begin{pmatrix}
		\ellfunc \big( \tau^2 + \kappa^2 \nu^2 \big)
		& \dfrac{\sigma \ellfunc (\kappa^2-1)\, \tau
		\big( w_\xi \tau + (w_\eta+1) \nu \big)}{|\qfield|^2}
		\\[6pt]
		0 & \tau^2 + \nu^2
	\end{pmatrix},
\]
so that $\det A^0 = \ellfunc (\tau^2+\nu^2)(\tau^2+\kappa^2\nu^2)$,
with roots in the upper half-plane $\nu = i\tau$ and
$\nu = i\tau/\kappa$ for $\tau > 0$, as already noted
in~\eqref{proper ellipticity roots}. The principal parts of the
boundary operators on $\Gamma_1$ --- the linearized Bernoulli
form~\eqref{linearized Bernoulli form}, whose zeroth-order term
$\dot w/F^2$ is subprincipal for the weight $r = -1$, and the
kinematic row --- give the boundary symbol
\[
	\mathcal{B}(\tau,\nu) = \begin{pmatrix}
		\dfrac{\sigma \tau}{|\qfield|^2}
		& -\dfrac{\sigma^2 \big( w_\xi \tau
		+ (w_\eta + 1)\nu \big)}{|\qfield|^4} \\[6pt]
		\nu & 0
	\end{pmatrix}.
\]
Following the scheme of \cite{volpert2011book1} (see also
\cite{haziot2023large}), we form the Lopatinskii matrix
\be
\label{definition Lopatinskii matrix}
	\Lambda(\tau) \colonequals \frac{1}{2\pi i} \oint_{\gamma}
	\mathcal{B}(\tau,\nu)\, \big( A^0 \big)^{-1}(\tau,\nu)\,
	\big( I, \ \nu I \big)\, d\nu,
\ee
where $\gamma$ is a contour in the upper half-plane enclosing the
roots of $\det A^0(\tau, \placeholder)$. The Lopatinskii
(complementing) condition at this boundary point is that the
$2 \times 2$ column minors $M_{ij}$ of the $2\times 4$ matrix
$\Lambda$ do not all vanish; the uniform version required for
the Schauder theory is
$\inf_{\Gamma_1} \sum_{i<j} |M_{ij}(\pm 1)| > 0$.

\begin{proposition}[Lopatinskii condition]
\label{lopat prop}
Assume the interior ellipticity
condition~\eqref{interior ellipticity condition}. Then the
following hold.
\begin{enumerate}[label=\rm(\alph*)]
\item \label{lopat surface part}
At each point of $\Gamma_1$ and for each $\tau \neq 0$,
\be
\label{lopat determinant}
	\sum_{i<j} |M_{ij}(\tau)| \, \geq \, |M_{14}(\tau)|
	= \frac{|\tau|\,\sigma^2}{4\kappa \ellfunc \,|\qfield|^3},
\ee
and every minor $M_{ij}$ vanishes wherever $\sigma = 0$. In
particular, the Lopatinskii condition holds uniformly on
$\Gamma_1$ if and only if
\be
\label{definition sigma0}
	\sigma_0 \colonequals
	\inf_{\Gamma_1} \frac{\sigma^2}{|\qfield|^3} > 0.
\ee
\item \label{lopat bed part}
On $\Gamma_0$, where the boundary rows are
$(\partial_\eta \dot u, \dot w)$ with symbol rows $(\nu,0)$ and
$(0,1)$, the minor $M_{12}$ of the corresponding Lopatinskii
matrix equals $1/(4i\kappa^2 \ellfunc \tau) \neq 0$, so the
Lopatinskii condition holds there with no further assumptions.
\end{enumerate}
\end{proposition}
\begin{proof}
By~\eqref{mu c2 identity}, $\sigma \neq 0$ if and only if
$\kappa > 1$ (for $\mu > 0$), while at $\mu = 0$ one has
$\kappa = 1$ regardless of $\sigma$. First assume that $\sigma \neq 0$. For $\mu > 0$ the poles
$\nu = i\tau/\kappa$ and $\nu = i\tau$ of $\det A^0(\tau,
\placeholder)$ are simple and distinct, and evaluating the
residues in~\eqref{definition Lopatinskii matrix} gives
\[
\begin{aligned}
R_1
&\colonequals
\operatorname*{Res}_{\nu = i\tau/\kappa}
\mathcal{B}(A^0)^{-1} =
\frac{1}{2 i \kappa}
\begin{pmatrix}
\dfrac{\sigma}{\ellfunc |\qfield|^2}
&
-\dfrac{\kappa \sigma^2
\bigl(\kappa w_\xi + i(w_\eta+1)\bigr)}
{|\qfield|^4}
\\[10pt]
\dfrac{i}{\kappa \ellfunc}
&
-\dfrac{i\sigma
\bigl(\kappa w_\xi + i(w_\eta+1)\bigr)}
{|\qfield|^2}
\end{pmatrix},
\\[12pt]
R_2
&\colonequals
\operatorname*{Res}_{\nu = i\tau}
\mathcal{B}(A^0)^{-1} =
\begin{pmatrix}
0 & 0 \\[4pt]
0 &
\dfrac{\sigma\bigl(w_\xi+i(w_\eta+1)\bigr)}
{2|\qfield|^2}
\end{pmatrix}.
\end{aligned}
\]
so that $\Lambda(\tau) = (R_1+R_2 \ \ \tfrac{i\tau}{\kappa}
R_1 + i\tau R_2 )$. A direct check shows $R_1$ has rank
one --- its first row equals $-(i\kappa\sigma/|\qfield|^2)$
times its second, so $\det R_1 = 0$ --- and the first entry of
$R_2\mathbf{e}_2$ vanishes. Consequently, in the minor
\[
	M_{14} = \det \begin{pmatrix} R_1\mathbf{e}_1 & \tfrac{i\tau}{\kappa}R_1\mathbf{e}_2 + i\tau
R_2\mathbf{e}_2 \end{pmatrix},
\]
 the $R_1$-contribution cancels identically
and only the cross term survives:
\[
	M_{14} = i\tau \det \begin{pmatrix} R_1\mathbf{e}_1
	& R_2\mathbf{e}_2 \end{pmatrix} 
	= \frac{\tau\,\sigma^2\big(w_\xi+i(w_\eta+1)\big)}
	{4\kappa \ellfunc\,|\qfield|^4}.
\]
Since $|w_\xi + i(w_\eta+1)| = |\qfield|$, this
gives~\eqref{lopat determinant} for $\mu>0$; all
$\kappa$-dependence beyond the prefactor cancels. At the
incompressible endpoint $\mu=0$ (still $\sigma\neq0$) the poles
merge at $\nu=i\tau$, the off-diagonal entry of $A^0$ vanishes,
and a direct evaluation at the single pole returns the same
formula with $\kappa=1$, so~\eqref{lopat determinant} extends
continuously across $\mu=0$.

Next, suppose instead that $\sigma = 0$. Every entry of the first row of
$\mathcal{B}$ carries a factor of $\sigma$, so the first row of
$\Lambda(\tau)$ vanishes identically and every $2\times2$ minor
is zero, consistent with $|M_{14}|\to0$ as $\sigma\to0$.
Combining the two cases proves part~\ref{lopat surface part}.

For part~\ref{lopat bed part}, the first column of the
$\Gamma_0$ Lopatinskii matrix is
\[
	\begin{pmatrix} \displaystyle \operatorname*{Res}_{\nu=i\tau/\kappa} \frac{\nu}{A_{11}^0} \\ 
\displaystyle 0 \end{pmatrix}  = \begin{pmatrix}\displaystyle \frac{1}{2\kappa^2 \ellfunc} \\ 0\end{pmatrix}.
\]
A similar residue calculation shows that the second entry of the second column is
$1/(2i\tau)$. Since the lower-left entry vanishes, we therefore find that
\[
	M_{12} 	= \frac{1}{4i\kappa^2 \ellfunc\tau} \neq 0. \qedhere
\]
\end{proof}

\begin{remark}
\label{kappa degeneracy remark}
The symbol degeneracy $\kappa=1$ and the breakdown of the
Lopatinskii condition are distinct phenomena. By~\eqref{mu c2
identity}, $\kappa=1$ occurs either at the incompressible
endpoint $\mu=0$, where it is harmless, or, when $\mu\neq0$,
precisely at a stagnation point $\sigma=0$. Consequently, for a
compressible solution ($\mu>0$) the no-stagnation condition
$\sigma_0>0$ is equivalent to $\kappa>1$ holding uniformly on
$\Gamma_1$, and the resulting lower bound is continuous across
$\mu=0$, so the ellipticity and Lopatinskii constants remain
uniform along the entire bifurcation curve.
\end{remark}

\subsubsection*{Schauder estimates}
We conclude the subsection with a proof of the Schauder-type a priori estimate for the linearized system~\eqref{genlinsystem}.

\begin{lemma}[Schauder estimate]
\label{schauder lemma}
Fix $\delta > 0$ and let
$(u,w,Q,\lambda,\mu) \in \mathcal{O}_\delta$. Then, for all
$(\dot u, \dot w, \dot Q) \in \Xspace_\per \times \mathbb{R}$,
\be
\label{schauder2}
\begin{aligned}
    \| \dot u \|_{C_\per^{2+\beta}(\confD)}
    + \| \dot w \|_{C_\per^{2+\beta}(\confD)}
    & \leq C \bigg(
    \Big\| D_{(u,w)} \F(u,w,Q,\lambda;\mu)
    \big(\dot u, \dot w\big)
    \Big\|_{\Yspace_{1\per} \times \Yspace_{2\per}} \\
    & \qquad
    + \Big\| D_{(u,w)} \G(u,w,Q,\lambda;\mu)
    \big(\dot u, \dot w \big)
    \Big\|_{\Yspace_{3\per}} \\
    & \qquad
    + |\dot Q|
    + \| \dot u \|_{C^0(\confD)}
    + \| \dot w \|_{C^0(\confD)} \bigg),
\end{aligned}
\ee
where the constant $C > 0$ depends only on $\delta$,
$\beta$, $\gamma$, $F$, $kh$, and upper bounds for
$|\lambda|$, $|\mu|$, $|Q|$, and
$\| (u,w) \|_{\Xspace_\per}$.
\end{lemma}
\begin{proof}
We first verify that the hypotheses of the linear elliptic
theory hold with constants controlled by the stated
quantities. By the definition~\eqref{definition periodic
O delta} of $\mathcal{O}_\delta$, the ellipticity function
satisfies $\inf_\confD \ellfunc \geq |\mu|\delta$ when
$\mu \neq 0$; on the other hand, the elementary bound
$|\ellfunc - 1| \leq |\mu| C_1$ holds for all $\mu$, where
$C_1$ is controlled by $\gamma$, $F$, $kh$, and
$\|(u,w)\|_{\Xspace_\per}$. If $|\mu| \leq 1/(2C_1)$ the
latter gives $\ellfunc \geq 1/2$, while otherwise the
former gives $\ellfunc \geq \delta/(2C_1)$; in either case
\[
    \inf_\confD \ellfunc \geq \ellfunc_0
    \colonequals \min\Big\{ \tfrac12,\
    \tfrac{\delta}{2 C_1} \Big\} > 0.
\]
Likewise, the fourth condition in the definition of $\mathcal{O}_\delta$ gives
$\sigma_0 \geq \delta$. Since moreover
$\ellfunc \leq 1 + |\mu| C_1$ by the same elementary bound
and $\kappa^2 = 1 + \mu(\partial_\xi u + \lambda)^2/
(\ellfunc |\qfield|^2)$ is bounded above by the stated
quantities and $\ellfunc_0$, we conclude
from~\eqref{gen:elliptic},
\eqref{proper ellipticity roots}, and
Proposition~\ref{lopat prop} that the boundary value
problem~\eqref{genlinsystem} is uniformly elliptic in the
Douglis--Nirenberg sense, properly elliptic, and satisfies
the Lopatinskii condition uniformly on $\partial\confD$,
with all constants controlled as claimed. The classical
Schauder theory for elliptic systems satisfying the
complementing condition
\cite[Theorem 9.2]{agmon1964estimates} therefore yields, for
every solution $(\dot u, \dot w) \in
C_\per^{2+\beta}(\overline\confD)$
of~\eqref{genlinsystem} with data $h_1, h_2 \in
C_\per^{\beta}(\overline{\confD})$ and $f, g \in
C_\per^{1+\beta}(\Gamma_1)$,
\be
\label{adn schauder}
\begin{aligned}
    \| \dot u \|_{C_\per^{2+\beta}(\confD)}
    + \| \dot w \|_{C_\per^{2+\beta}(\confD)}
    & \leq C \Big(
    \| h_1 \|_{C_\per^{\beta}(\confD)}
    + \| h_2 \|_{C_\per^{\beta}(\confD)}
    + \| f \|_{C_\per^{1+\beta}(\Gamma_1)}
    + \| g \|_{C_\per^{1+\beta}(\Gamma_1)} \\
    & \qquad \qquad
    + \| \dot u \|_{C^0(\confD)}
    + \| \dot w \|_{C^0(\confD)} \Big).
\end{aligned}
\ee

It remains to match~\eqref{linsystem} with the data
in~\eqref{genlinsystem}. Given
$(\dot u, \dot w, \dot Q) \in \Xspace_\per \times
\mathbb{R}$, set
\[
\begin{aligned}
    h_1 & \colonequals D_{(u,w,Q)}\F_1(u,w,Q,\lambda; \mu)
    \begin{pmatrix} \dot u \\ \dot w \\ \dot Q\end{pmatrix},
    \\
    h_2 & \colonequals \Delta \dot w,
    \\
    f &\colonequals \dot{\mathfrak{B}}
    = D_{(u,w,Q)}\G(u,w,Q,\lambda; \mu)
    \begin{pmatrix} \dot u \\ \dot w \\ \dot Q \end{pmatrix} {}+ \dot Q,
    \end{aligned}
\]
where we have used~\eqref{Lcbernoulli}. The remaining side
conditions are homogeneous: the kinematic datum is
$g = \partial_\eta \dot u = 0$ on $\Gamma_1$, and
$\partial_\eta \dot u = 0$, $\dot w = 0$ on $\Gamma_0$,
all built into the space $\Xspace_\per$. Under the
identifications
$\Yspace_{1\per}\times\Yspace_{2\per}
= C_\per^{\beta} \times C_\per^{\beta}$ and
$\Yspace_{3\per} = C_\per^{1+\beta}(\Gamma_1)$,
the data lie in exactly the spaces required
by~\eqref{adn schauder}, and
\begin{align*}
    \| h_1 \|_{C_\per^{\beta}}
    + \| h_2 \|_{C_\per^{\beta}}
    + \| f \|_{C_\per^{1+\beta}(\Gamma_1)}
    & \leq
    \big\| D_{(u,w,Q)}\F
    (\dot u, \dot w, \dot Q)\big\|_{\Yspace_{1\per}
    \times\Yspace_{2\per}} \\
    & \quad
    + \big\| D_{(u,w,Q)}\G
    (\dot u, \dot w, \dot Q)\big\|_{\Yspace_{3\per}}
    + |\dot Q|,
\end{align*}
the term $|\dot Q|$ accounting for the splitting
$f = D_{(u,w,Q)}\G(\dot u,\dot w,\dot Q) + \dot Q$.
Applying~\eqref{adn schauder} with $g = 0$ and the data
bounded as above gives~\eqref{schauder2}.
\end{proof}

\begin{remark}[Generality]
\label{generality remark}
The proof of Lemma~\ref{schauder lemma} uses only the
following structural features of~\eqref{genlinsystem}:
the principal symbol is upper triangular with entries of
the form~\eqref{linsys:coeffs}, the interior coefficients
lie in $C^{\beta}_\per(\overline\confD)$ and the boundary coefficients in $C^{1+\beta}_\per(\Gamma_1)$, and the
boundary operators have the principal rows appearing in
the proof of Proposition~\ref{lopat prop}. The same
Schauder estimate, at any order $k \geq 0$ with data
$h_1, h_2 \in C^{k+\beta}_\per$ and $f, g \in
C^{k+1+\beta}_\per(\Gamma_1)$ and solutions in
$C^{k+2+\beta}_\per$, therefore holds for any linear
elliptic system of this class, with the constant depending
on the same quantities. The case $k=0$, stated above, is
the one used in the sequel; higher values of $k$ would
apply if the background solution $(u,w)$ and its
perturbation were assumed to have correspondingly higher
regularity.
\end{remark}

\subsection{Local properness and Fredholmness}

We now prove that as a consequence of the ellipticity of the system --- and specifically the resulting Schauder estimate of Lemma~\ref{schauder lemma} --- the abstract operator $(\F,\G)$ has the requisite compactness and Fredholmness properties needed to carry out (analytic) global continuation of the local curves constructed in Section~\ref{local bifurcation section}.

\begin{lemma}[Local properness]
\label{local properness lemma}
For each $\delta > 0$, the zero-set of $(\F,\G)$ is locally proper: for any closed and bounded set $K \subset \mathcal{O}_\delta$, we have that 
\[
	(\F,\G)^{-1}(0) \cap K \subset\subset \Xspace_\per \times \mathbb{R}^3.
\]
\end{lemma}
\begin{proof}
Fix $\delta > 0$. Let $\{ (u_j, w_j, Q_j, \lambda_j, \mu_j)\} \subset \Xspace_\per \times \mathbb{R}^3$ be a bounded sequence in $(\F,\G)^{-1}(0) \cap \mathcal{O}_\delta$; we must show that $\{ (u_j, w_j, Q_j, \lambda_j, \mu_j)\}$ has a subsequence converging in $\Xspace_\per \times \mathbb{R}^3$. Note that each term in the sequence is a solution of the PDE system
\[
\left\{
\begin{aligned}
	\Delta u_j & = \mu_j \left(  (\gamma-1) \Delta u_j + (\nabla u_j  + \lambda_j \mathbf{e}_1) \cdot \nabla  \right)  \left( \frac{1}{2} \frac{|\nabla u_j + \lambda_j \mathbf{e}_1|^2}{|\nabla w_j+\mathbf{e}_2|^2} + \frac{1}{F^2} (w_j+\eta) \right)   \\
	\Delta w_j & = 0
\end{aligned}
\right.
\qquad \textrm{in } \confD
\]
together with the boundary conditions 
\[
\left\{ \begin{aligned}
	\partial_\eta u_j & = 0 & \qquad & \textrm{on } \partial \confD \\
	\frac{1}{2} \frac{| \nabla u_j+\lambda_j \mathbf{e}_1|^2}{|\nabla w_j + \mathbf{e}_2|^2} + \frac{1}{F^2} w_j -Q_j  &= 0 & \qquad & \textrm{on } \Gamma_1 \\
	w_j & = 0 & \qquad & \textrm{on } \Gamma_0.
\end{aligned} \right.
 \]
 Of course, it suffices to show that $\{(u_j,w_j, Q_j,\lambda_j,\mu_j)\}$ has a convergent subsequence in $C_\per^{2+\beta}(\overline{\confD}) \times \mathbb{R}^3$. By boundedness we know that, passing to subsequences, $(Q_j,\lambda_j,\mu_j) \to (Q, \lambda, \mu)$ for some $Q, \lambda, \mu \in \mathbb{R}$. Moreover, by compactness of the embedding, we can pass to a further subsequence to conclude that there exists $u, w \in C_\per^{2+\beta}(\overline{\confD})$ such that $u_j \to u$ and $w_j \to w$ in $C_\per^2(\overline{\confD})$. We upgrade this to convergence in the full $C^{2+\beta}_\per$ norm using the following standard argument.
  
Because of translation invariance in the $\xi$-direction, $(\partial_\xi u_j, \partial_\xi w_j)$ satisfies the ``quasilinearized'' problem
\[
	D_{(u,w)} (\F,\G)(u_j,w_j,Q_j,\lambda_j,\mu_j) \begin{pmatrix} \partial_\xi u_j \\ \partial_\xi w_j \end{pmatrix} = \begin{pmatrix} 0 \\ 0 \\ 0 \end{pmatrix}.
\]
This is a system with interior coefficients uniformly bounded in $C^{\beta}_\per(\overline{\confD})$ and boundary coefficients uniformly bounded in $C^{1+\beta}_\per(\Gamma_1)$. Thanks to the definition of $\mathcal{O}_\delta$, it is moreover elliptic in the ADN sense with a constant depending only on $\delta$. Applying the Schauder estimates of Lemma~\ref{schauder lemma} (whose proof does not use the parity or mean-zero restrictions built into $\Xspace_\per$, so that it applies to $(\partial_\xi u_j, \partial_\xi w_j)$ even though these have the opposite parity), we therefore have that $\{ (\partial_\xi u_j, \partial_\xi w_j) \}$ is uniformly bounded in $C^{2+\beta}_\per(\overline{\confD})$. 
This gives convergence of all of the derivatives of $w_j$ of order up to two in $C^{\beta}_\per(\overline{\confD})$ except perhaps $\partial_\eta^2 w_{j}$. But, since each $w_j$ is harmonic, so too is the limit $w$, and we have
\[
	\partial_\eta^2 w_j = -  \partial_\xi^2 w_j \longrightarrow_{C^\beta_\per} - \partial_\xi^2 w = \partial_\eta^2 w.
\]

By the same reasoning, we have that all derivatives of $u_j$ up to order two except perhaps $\partial_\eta^2 u_j$ are convergent in $C_\per^{\beta}(\overline{\confD})$. Differentiating the equation $\F_1(u_j,w_j,Q_j,\lambda_j,\mu_j) = 0$ in $\eta$ shows that $\partial_\eta u_j$ solves
\[
	D_u \F_1(u_j, w_j, Q_j,\lambda_j, \mu_j)\, \partial_\eta u_j = g_j,
\]
where $\{g_j\}$ is uniformly bounded in $C^\beta_\per(\overline\confD)$, and $D_u\F_1$ is a uniformly elliptic PDE operator with coefficients bounded uniformly in $C^{\beta}_\per(\overline{\confD})$. As $\partial_\eta u_j$ satisfies homogeneous Dirichlet conditions on $\partial\confD$, standard (scalar) Schauder estimates imply that it is uniformly bounded in $C^{2+\beta}_\per(\overline{\confD})$, which completes the argument. 
\end{proof}

The next lemma records the fact that the linearization of $(\F,\G)$ is semi-Fredholm at any solution in $\mathcal{O}_\delta$; this is merely a consequence of the definition of $\mathcal{O}$ and linear Schauder estimates in Lemma~\ref{schauder lemma}.

\begin{lemma}[semi-Fredholmness]
\label{periodic Fredholmness lemma}
For any $(u,w,Q,\lambda,\mu) \in (\F,\G)^{-1}(0) \cap \mathcal{O}$, the linearized mapping $D_{(u,w,Q)}(\F,\G)$ is semi-Fredholm $\Xspace_\per \times \mathbb{R} \to \Yspace_\per$.
\end{lemma}
\begin{proof}
First note that it suffices to establish the semi-Fredholmness of $D_{(u,w)} (\F,\G)$ as mapping $\Xspace_\per \to \Yspace_\per$, since this differs from the linearized operator by a finite-rank perturbation. Let $(u_*,w_*,Q_*,\lambda_*,\mu_*) \in \mathcal{O}_\delta$ be given, and consider the corresponding linearized operator 
\[
	\mathscr{L} = ( \mathscr{L}_1, \mathscr{L}_2, \mathscr{L}_3 ) \colonequals D_{(u,w)}(\F,\G)(u_*,w_*,Q_*, \lambda_*, \mu_*) \colon \Xspace_{\per} \to \Yspace_\per.
\]
As a consequence of the Schauder estimates from Lemma~\ref{schauder lemma}, which hold in view of the definition of $\mathcal{O}_\delta$, we have the bound
\begin{align*}
	\| \dot u \|_{C_\per^{2+\beta}(\confD)} + \| \dot w \|_{C_\per^{2+\beta}(\confD)}  & \leq C \Big( \| \dot u \|_{C_\per^{\beta}(\confD)} + \| \dot w \|_{C_\per^{\beta}(\confD)} + \left\| \mathscr{L}_1 \begin{pmatrix} \dot u \\ \dot w \end{pmatrix} \right\|_{C_\per^{\beta}(\confD)} \\ 
	& \qquad   + \left\| \mathscr{L}_2 \begin{pmatrix} \dot u \\ \dot w \end{pmatrix} \right\|_{C_\per^{\beta}(\confD)}  +  \left\| \mathscr{L}_3  \begin{pmatrix} \dot u \\ \dot w \end{pmatrix} \right\|_{C^{1+\beta}(\Gamma_1)}\Big),
\end{align*}
where the constant $C > 0$ depends only on lower bounds on $\delta$ and upper bounds on $|Q_*|$, $|\lambda_*|$, $\mu_*$, and the $C^{2+\beta}$ norms of $u_*$ and $w_*$. Since the spaces represented by the norms on the left-hand side of this inequality are compactly embedded in the spaces represented by the (interior) norms on the right-hand side, a standard argument gives that $\mathscr{L}$ is semi-Fredholm. 
\end{proof}

\subsection{Proof of the main theorem}
\label{proof of the main theorem section}

We now complete the global continuation of the Janzen--Rayleigh expansion for the periodic case. The primary tool is the analytic global implicit function theorem. 

\begin{proof}[Proof of Theorem~\ref{periodic global bifurcation theorem}]
Applying Theorem~\ref{homoclinic global ift} at the initial incompressible solution $(0,w_0, Q_0, \lambda_0, 0)$ furnishes the continuous  curve $\cm_\per$. Local uniqueness implies that $\cm_\loc^\per \subset \cm_\per$, which proves the statement in part~\ref{periodic extension part}. Since each point on $\cm_\per$ is necessarily in the connected component of $\mathcal{O}$ containing $(0,w_0, Q_0,\lambda_0, 0)$, the physicality of the solutions asserted in part~\ref{periodic physical part} is a consequence of Lemma~\ref{periodic equivalence lemma}.

Next consider the limiting behavior as $s \to \infty$ along $\cm_\per$. Thanks to Lemmas~\ref{local properness lemma} and \ref{periodic Fredholmness lemma}, the loss of compactness alternative~\ref{K loss of compactness alternative} and loss of Fredholmness alternative~\ref{K loss of fredholmness alternative} can be excluded, respectively. This leaves only  blowup~\ref{K blowup alternative} and the closed loop alternative~\ref{K loop}. 

Assume first that the latter of these occurs. Then it must be that $\mu(T) = \mu(0) = 0$ for some $T > 0$. At the initial solution $(0,w_0,Q_0,\lambda_0,0)$, we have by local uniqueness that every solution lies on $\cm_\loc^\per$, and hence $\mu(s) < 0$ for $0 < T-s \ll 1$. Therefore, there must exist some $s_* \in (0,T)$ with $\mu(s_*) = 0$. Thus, the alternative~\ref{periodic loop alternative} occurs.

Suppose instead that we have blowup, which here means specifically that
\begin{align*}
	N(s) & \colonequals \|u(s) \|_{C_\per^{2+\beta}(\confD)} + \| w(s) \|_{C_\per^{2+\beta}(\confD)} + |Q(s)| +|\mu(s)| \\
		& \qquad + \frac{1}{\dist{((u(s), w(s), Q(s), \lambda_0, \mu(s)), \partial\mathcal{O})}} \\
		& \longrightarrow \infty.
\end{align*} 
Recalling the definition of $\mathcal{O}$ in~\eqref{definition periodic O delta}, in order for $\cm_\per$ to come arbitrarily close to $\partial \mathcal{O}$, thanks to Lemma~\ref{subsonic ADN lemma}, we must have that either~\ref{boundary degeneracy alternative}, \ref{transsonic limit alternative}, or \ref{stagnation alternative} occurs. On the other hand, if the distance between $\cm_\per$ and $\partial\mathcal{O}$ is uniformly positive and $(u(s),w(s))$ is uniformly bounded in $\Xspace_\per$, then taking the mean of the dynamic condition~\eqref{cbernoulli periodic} over a period reveals that
\[
	Q(s) = \frac{1}{2\pi} \int_{-\pi}^\pi \frac{1}{2} \frac{|\nabla u(s)(\xi,0) + \lambda_0 \mathbf{e}_1|^2}{|\nabla w(s)(\xi,0) + \mathbf{e}_2|^2} \, d\xi,
\]
hence $Q(s)$ must likewise be uniformly bounded. This allows us to pare down the unboundedness statement to~\ref{periodic unboundedness alternative}.
\end{proof}

\section{Lightly compressible internal fronts}
\label{front section}

In this section, we study lightly compressible \emph{internal waves}. These are solutions to the two-phase isentropic Euler equations in two dimensions, where the phases are separated by a free boundary across which the normal stress (and hence pressure) and normal velocity are continuous, but the tangential velocity and density need not be. Periodic and solitary waves for this system have been constructed by many authors in the incompressible regime; see, for example, the survey~\cite[Section 7]{haziot2022traveling}. One of the distinctive features of the two-phase problem is the existence of \emph{front-type} solutions, called \emph{bores}, that have different upstream and downstream limiting forms. We will restrict our attention to these, both because they are comparatively more difficult to analyze, and to complement the theory of lightly compressible solitary and periodic waves in the previous sections. 

\subsection{Governing equations}

\begin{figure}
	\includegraphics[scale=1.]{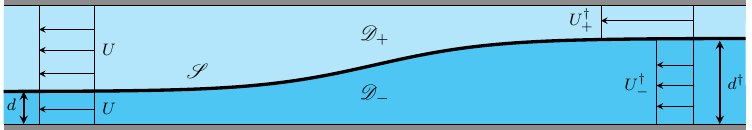}
	\caption{A front in a two-fluid system with the upper fluid region $\fluidD_+$ shaded lighter blue and the lower fluid region $\fluidD_-$ shaded darker blue. An impermeable bed bounds the configuration from below and an impermeable lid bounds it from above. The interface $\fluidS$ dividing the layers is bolded. In the upstream limit $x \to -\infty$, the relative velocity field in both layers tends to $(U,0)$; as is customary, the magnitude of the velocity is indicated by the length of the arrow. By contrast, in the downstream limit $x \to +\infty$, the relative velocity will be purely horizontal but take on different values in each layer.  }
	\label{front setup figure}
\end{figure}

Suppose that there are two immiscible fluids confined between impermeable horizontal barriers. Performing a standard nondimensionalization akin to~\eqref{nondimensionalization space and velocity}, we may choose Cartesian coordinates $(x,y)$ in a moving frame so that the line $\{y = 0\}$ corresponds to the rigid bed beneath the fluids and the line  $\{y = 1\}$ the rigid lid above them. This amounts to choosing the vertical height of the channel as the characteristic length scale; we will specify the characteristic velocity and density shortly. The region between the bed and lid is the disjoint union of an upper fluid layer occupying the region $\fluidD_+$ and a lower fluid layer occupying the region $\fluidD_-$. They are separated by a free boundary $\fluidS = \partial \fluidD_- \cap \partial \fluidD_+$. Assume that $\fluidS$ is given globally as the graph of a single-valued function $h = h(x)$, and that 
\be
\label{far-field depths}
	\lim_{x \to -\infty} h(x) = d, \qquad \lim_{x \to +\infty} h(x) = d^\dagger.
\ee
Thus $\fluidD_+$ lies above $\fluidS$ and below $\{ y = 1\}$, while $\fluidD_-$ lies below $\fluidS$ and above $\{ y = 0 \}$; see Figure~\ref{front setup figure}. We call $d$ the \emph{upstream depth} of the lower fluid layer and $d^\dagger$ the \emph{downstream depth} of the lower layer. Throughout this section, we will exclusively use a dagger superscript to denote downstream limiting quantities. As mentioned above, our main interest is in \emph{fronts}, meaning that $d \neq d^\dagger$.  Note that the upstream depth of the upper layer will then be $1-d$ and the corresponding downstream depth is $1-d^\dagger$.

The flow in each fluid layer is governed by the isentropic Euler equations, taking the equation of state in $\fluidD_\pm$ to be 
\[
	P_\pm(r) \colonequals \alpha_\pm^2 r.
\]
Note that the constants $\alpha_+, \alpha_- > 0$ may be distinct, but, for simplicity, we restrict ourselves to the isothermal case. Conservation of momentum and mass~\eqref{isentropic Euler equations} then take the form
\bse 
\label{two-phase isentropic Euler system}
\be
\label{two-phase steady Euler equation}
	\left\{\begin{aligned} 
		(\nabla \phi_\pm \cdot \nabla) \nabla \phi_\pm+ \frac{1}{\rho_\pm} \nabla P_\pm(\rho_\pm) & = -\frac{1}{F^2} \mathbf{e}_2  \\
		\nabla \cdot \left( \rho_\pm \nabla \phi_\pm \right) & = 0 
	\end{aligned} \right.
	\qquad \textrm{in } \fluidD_\pm,
\ee
where $\phi_\pm$ is the potential for the relative velocity in $\fluidD_\pm$, $\rho_\pm$ is the corresponding density, and $F$ is the Froude number.

As always, the kinematic condition requires that the relative velocity field is tangent along each boundary component, which in term of the relative potential becomes
\be
\label{two-phase kinematic condition}
	\left\{ \begin{aligned}
		\partial_y \phi_\pm & =  (\partial_x h) \partial_x \phi_\pm & \qquad & \textrm{on } \fluidS \\
		\partial_y \phi_+ & = 0 & \qquad & \textrm{on } \{ y = 1 \} \\
		\partial_y \phi_- & = 0 & \qquad & \textrm{on } \{ y = 0 \}.
	\end{aligned} \right.
\ee
The dynamic condition mandates that pressure is continuous over the interface $\fluidS$, that is
\be
\label{two-phase dynamic condition}
	P_+(\rho_+) = P_-(\rho_-) \qquad \textrm{on } \fluidS.
\ee
In particular, this implies that the traces of $\rho_+$ and $\rho_-$ are constant multiples of each other, though they need not themselves be constant:
\[
	\frac{\rho_+}{\rho_-} = \frac{\alpha_-^2}{\alpha_+^2} \equalscolon \rho_0 + \delta \qquad \textrm{on } \fluidS.
\]
Here, $\rho_0$ will be the (constant) ratio between the densities in the starting incompressible case and $\delta$ is a perturbative parameter.

Finally, to complete the system it is necessary to prescribe the asymptotic conditions upstream and downstream. The limiting velocity field in both directions is uniform and purely horizontal. For simplicity, we assume that in the upstream limit, the speed in both fluid layers is the same, and we denote it by $U$. In fact, we take $|U|$ to be the characteristic velocity, so that it suffices to assume $U=-1$. Nonetheless, we will often write $U$ so that its meaning in formulas is not obscured. In the downstream limit the limiting velocity in the layers $\fluidD_\pm$ is $U_\pm^\dagger$:
\be
\label{two-phase far-field condition}
	\partial_x \phi_\pm \to U \qquad \textrm{as } x \to -\infty, \qquad \partial_x \phi_\pm \to U_\pm^\dagger  \qquad \textrm{as } x \to +\infty.
\ee
\ese
There are not enough degrees of freedom to assume that $U_+^\dagger = U_-^\dagger$; see Corollary~\ref{conjugate flow corollary}. To complete the nondimensionalization, we take the characteristic density to be the upstream limiting value of $\rho_-$ on $\fluidS$.  

 Integrating the momentum equation in the $\fluidD_\pm$ yields Bernoulli's law
\be
\label{two-phase Bernoulli law}
	\frac{1}{2} |\nabla \phi_\pm|^2 + \frac{1}{F^2} y + H_\pm(\rho_\pm) = B_\pm \qquad \textrm{in } \fluidD_\pm,
\ee
with $B_+$ and $B_-$ being (possibly distinct) constants, and the enthalpy corresponding to the fluid in the region $\fluidD_\pm$ is 
\[
	H_\pm \colonequals  
			\alpha_\pm^2  \log{(\placeholder)}.
\]
Notice that the enthalpy, unlike the pressure, will have a jump across $\fluidS$:
\be
\label{enthalpy jump}
	(\rho_0 + \delta) H_+(\rho_+) - H_-(\rho_-) = \alpha_+^2 (\rho_0+\delta) \log{(\rho_0+\delta)} \qquad \textrm{on } \fluidS.
\ee
Importantly, the right-hand side above is constant.

As in the one-fluid case, we can eliminate the pressure and density to obtain a system for $\phi$ and $h$ alone. Taking the material derivative of Bernoulli's law~\eqref{two-phase Bernoulli law} and repeating the same process as in Section~\ref{governing equations section}, leads to  
\bse
\label{steady two-phase system}
\begin{equation}
\label{steady two-phase Euler}
	\Delta \phi_\pm = \mu_\pm  \nabla \phi_\pm \cdot \nabla  \left( \frac{1}{2}  |\nabla \phi_\pm|^2 +  \frac{1}{F^2} y \right) \qquad \textrm{in } \fluidD_\pm,
\end{equation}
where $\mu_\pm \colonequals 1/\alpha_\pm^2$ is the square Mach number for the flow in $\fluidD_\pm$. Note that $\mu_+/\mu_- = \rho_0+\delta$, so we will think of $\mu \colonequals \mu_-$ and $\delta$ as being the free parameters, and set
\[
	\mu_+ = (\rho_0+\delta) \mu_- \equalscolon (\rho_0+\delta) \mu.
\]
Using the enthalpy jump identity~\eqref{enthalpy jump} and Bernoulli's law, allows us to rewrite the dynamic condition on $\fluidS$ as 
\be
\label{steady two-phase dynamic}
	\frac{\rho_0+\delta}{2} |\nabla \phi_+|^2 - \frac{1}{2} |\nabla \phi_-|^2 + \frac{\rho_0 -1+\delta}{F^2} y = Q \qquad \textrm{on } \fluidS,
\ee
where we now define the Bernoulli constant
\be
\label{front Q relation}
	Q \colonequals (\rho_0+\delta) B_+ - B_- - \alpha_+^2 ( \rho_0+\delta)  \log{(\rho_0+\delta)} = \frac{\rho_0+\delta-1}{2} \left(U^2+\frac{d}{F^2} \right).
\ee
In the second equality, the value of $Q$ has been determined by evaluating the left-hand side of~\eqref{steady two-phase dynamic} using the far-field conditions~\eqref{far-field depths} and \eqref{two-phase far-field condition}. Finally, the kinematic condition on the interface and top and bottom rigid boundaries takes the form
\be
\label{steady surface two-phase kinematic}
	\nu_\pm \cdot\nabla\phi_\pm=0 \qquad \textrm{on } \fluidS  
\ee
with $\nu_\pm$ the outward unit norm to $\fluidD_\pm$ along $\fluidS$, and
\be
	\partial_y \phi_\pm = 0 \qquad \textrm{on } \fluidB_\pm.
\ee
\ese
Setting $(\mu,\delta) = (0,0)$, this reduces to the incompressible problem.

\subsection{Trivial solutions and conjugate flows}

We say a solution is \emph{trivial} provided that it satisfies~\eqref{two-phase steady Euler equation}--\eqref{two-phase dynamic condition} and has the form $\phi_\pm = U_\pm x$ for some constants $U_\pm$, which corresponds to the velocity field being completely horizontal in both fluids and with constant horizontal velocity $U_\pm$. Bernoulli's law~\eqref{two-phase Bernoulli law} then determines that the density has the form 
\be
\label{far-field density ansatz}
	\rho_\pm = R_\pm(y) \colonequals \exp{\left(  \mu_\pm (B_\pm-\frac{1}{2} U_\pm^2 - \frac{1}{F^2} y ) \right)}. 
	\ee

An essential consideration in both the existence and qualitative theory of fronts is a characterization of what Benjamin~\cite{benjamin1971unified} calls the \emph{conjugate flows} of the system. Observe that the upstream and downstream limiting relative potential of a front must themselves be trivial solutions.  The next proposition, which is well-known but stated here for convenience, shows that there exist a number of physical quantities that are constant on each vertical cross-section of the fluid domain.  Two trivial solutions that have the same value of all these conserved quantities are called \emph{conjugate}. In order for two trivial solutions to be the upstream and downstream limits of a front, they must be conjugate, which creates a necessary condition for the existence of a front connecting two far-field states.

\begin{proposition}[Conserved quantities]
For any classical solution to the steady isentropic Euler equations~\eqref{two-phase isentropic Euler system}, the mass flux in the lower layer $\mathfrak{m}_-$, mass flux in the upper layer $\mathfrak{m}_+$, and flow force $\mathfrak{s}$ defined by
\be
\label{conserved quantities}
\begin{aligned}
	\mathfrak{m}_-(x) & \colonequals \int_{\fluidD_-^x } \rho_- \partial_x \phi_- \, dy,  \qquad 
	\mathfrak{m}_+(x)  \colonequals  \int_{\fluidD_+^x} \rho_+ \partial_x \phi_+ \, dy, \\
	\mathfrak{s}(x) & \colonequals \int_{\fluidD_-^x} \left( P_-(\rho_-) + \rho_- (\partial_x \phi_-)^2 \right) \, dy  + \int_{\fluidD_+^x} \left( P_+(\rho_+) + \rho_+ (\partial_x \phi_+)^2 \right) \, dy
\end{aligned} 
\ee
are independent of $x$, where $\fluidD_\pm^x \colonequals \fluidD_\pm \cap ( \{ x \} \times (0,1)) $ is the $x$-section of the corresponding fluid layer.
\end{proposition}
\begin{proof}
The $x$-independence of $\mathfrak{m}_\pm$ is an immediate consequence of the conservation of mass in~\eqref{two-phase steady Euler equation} and kinematic boundary condition~\eqref{two-phase kinematic condition}: for any $x_1 < x_2$, we have
\[
	\mathfrak{m}_-(x_2) - \mathfrak{m}_-(x_1) = \int_{x_1}^{x_2} \int_0^{h(x)} \nabla \cdot \left( \rho_- \nabla \phi_- \right) \, dy \, dx = 0.
\]
The proof for $\mathfrak{m}_+$ is similar. Likewise, the conservation of the flow force $\mathfrak{s}$ results from writing the horizontal component of the momentum equations in conservative form
\[
	\partial_x P_\pm + \partial_x \left( \rho_\pm (\partial_x \phi_\pm)^2 \right) + \partial_y \left( \rho_\pm (\partial_x \phi_\pm) (\partial_y \phi_\pm) \right) = 0 \qquad \textrm{in } \fluidD_\pm
\]
and integrating, using the kinematic condition and the continuity of the pressure to eliminate the boundary terms on $\fluidS$.
\end{proof}

Evaluating the conserved quantities above at upstream and downstream states of a classical solution yields a system of algebraic equations relating the parameters $(U, U_+^\dagger, U_-^\dagger, d, d^\dagger)$. These are called the \emph{conjugate flow equations} and are recorded in the next corollary. For the incompressible case, the system is polynomial, whereas for isentropic Euler, it contains transcendental equations.

\begin{remark}
Recall that we took the upstream limit of trace of $\rho_-$ on $\fluidS$ to be the characteristic density. Thus, in our nondimensionalized variables, we should have $\rho_-(x,h(x)) \to 1$ as $x \to -\infty$. Since the far-field density limits are all of the form~\eqref{far-field density ansatz}, we find that
\[
	 B_- = \frac{1}{2} U^2 + \frac{d}{F^2}, \qquad (\rho_0+\delta)\mu B_+ =  (\rho_0+\delta)\mu \left( \frac{1}{2} U^2 + \frac{d}{F^2}\right) +  \log{(\rho_0+\delta)}.
\]
In particular, this implies that $\rho_\pm \to R_\pm$ as $x \to -\infty$, where 
\be
\label{upstream density}
	R_-(y) = e^{-\mu \frac{y-d}{F^2}}, \qquad R_+(y) = (\rho_0+\delta) e^{-(\rho_0+\delta) \mu \frac{y-d}{F^2}},
\ee
and in the downstream limit, $\rho_\pm \to R_\pm^\dagger$ as $x \to +\infty$, with
\be
\label{downstream density}
	R_-^\dagger(y) = e^{-\mu\left( \frac{1}{2} ((U_-^\dagger)^2 - U^2) + \frac{y-d}{F^2}  \right)}, \qquad R_+^\dagger(y) = (\rho_0+\delta) e^{-(\rho_0+\delta) \mu \left( \frac{1}{2} ((U_+^\dagger)^2 - U^2) + \frac{y-d}{F^2} \right)}.
\ee
\end{remark}

\begin{corollary}[Conjugate flow equations]
\label{conjugate flow corollary}
If $(\phi,h)$ is a classical solution to the steady isentropic Euler equations~\eqref{two-phase isentropic Euler system}, then the limiting depths $d$ and $d^\dagger$ and velocities $U$, $U_+^\dagger$, and $U_-^\dagger$ must satisfy the following system of equations. 

\begin{enumerate}[label=\rm(\alph*)]
\item \textup{(Conservation of energy)} The dynamic condition on $\fluidS$ implies that
\be
\label{conservation of energy}
\begin{aligned}
	\frac{\rho_0-1+\delta}{2} U^2 + \frac{\rho_0-1+\delta}{F^2} d & = Q = 
	\frac{\rho_0+\delta}{2} (U_+^\dagger)^2 - \frac{1}{2} (U_-^\dagger)^2 + \frac{\rho_0-1+\delta}{F^2} d^\dagger. 
\end{aligned}
\ee
\item \textup{(Conservation of mass flux)}  The conservation of the mass flux in the upper and lower layers implies that
\be
\label{conservation of mass flux} 
\left\{
\begin{aligned}
	 Ue^{-\frac{\mu}{2} U^2 } \frac{ 1 - e^{- \frac{\mu d}{F^2}} }{1- e^{- \frac{\mu d^\dagger}{F^2} }} & =
	U_-^\dagger e^{-\frac{\mu}{2} (U_-^\dagger)^2 } \\
	 U e^{-(\rho_0+\delta)\frac{\mu}{2}U^2} \frac{  e^{\frac{(\rho_0+\delta) \mu (1-d)}{F^2}} - 1}{e^{\frac{(\rho_0+\delta)\mu (1-d^\dagger)}{F^2} } -1 } & =
	U_+^\dagger e^{-(\rho_0+\delta) \frac{\mu}{2}(U_+^\dagger)^2  }. 
\end{aligned}
\right.
\ee
\item \textup{(Conservation of flow force)} Denoting by $R_\pm$ the (trivial) density for the upstream flow in $\fluidD_\pm$ and $R_\pm^\dagger$ the trivial density for the downstream flow, then conservation of the flow force implies 
\be
\label{conservation of flow force}
\begin{aligned}
	 & \left(1 + \mu U^2 \right) \frac{R_-(d) - R_-(0)}{\mu} + \left(1 + (\rho_0+\delta) \mu U^2 \right) \frac{R_+(1) - R_+(d)}{(\rho_0+\delta)^2 \mu}  \\
		 & \qquad\qquad = \left(1 + \mu (U_-^\dagger)^2 \right) \frac{R_-^\dagger(d^\dagger) - R_-^\dagger(0)}{\mu}     + \left(1 + (\rho_0+\delta) \mu (U_+^\dagger)^2 \right) \frac{R_+^\dagger(1) -  R_+^\dagger(d^\dagger)}{(\rho_0+\delta)^2 \mu}.
		\end{aligned}
\ee
\end{enumerate}
\end{corollary}

\begin{proof}
The conservation of energy~\eqref{conservation of energy} simply results from evaluating the Bernoulli condition~\eqref{steady two-phase dynamic} with the hypothetical upstream and downstream flows. Likewise, the first equation in~\eqref{conservation of mass flux} is just found by rearranging the equality
\[
	U \left( R_-(d) - R_-(0) \right) = \lim_{x\to -\infty} \mathfrak{m}_-(x) = \lim_{x \to +\infty} \mathfrak{m}_-(x) = U_-^\dagger \left( R_-^\dagger(d^\dagger) - R_-^\dagger(0) \right).
\]
The second conservation of mass equation in~\eqref{conservation of mass flux} follows from the same argument but in the upper layer.

Equating the upstream and downstream limits of the force results in the identity 
\begin{align*}
	\mathfrak{s} & = \int_0^d \left( P_-(R_-) + R_- U^2 \,  \right) dy + \int_d^1 \left( P_+(R_+) + R_+ U^2 \,  \right) dy \\
		& = \int_0^{d^\dagger} \left( P_-(R_-^\dagger) + R_-^\dagger (U_-^\dagger)^2 \,  \right) dy + \int_{d^\dagger}^1 \left( P_+(R_+^\dagger) + R_+^\dagger (U_+^\dagger)^2 \,  \right) dy.
\end{align*}
Using the explicit forms of the pressure and of the upstream~\eqref{upstream density} and downstream densities~\eqref{downstream density}, these integrals can be evaluated to yield
\begin{align*}
	 & \left(\frac{1}{\mu_-} + U^2 \right) \frac{1}{\mu_-} (R_-(d) - R_-(0)) + \left(\frac{1}{\mu_+} + U^2 \right) \frac{1}{\mu_+} ( R_+(1) - R_+(d) ) \\
		\qquad & = \left(\frac{1}{\mu_-} + (U_-^\dagger)^2 \right) \frac{1}{\mu_-}   (R_-^\dagger(d^\dagger) - R_-^\dagger(0) )   + \left(\frac{1}{\mu_+} + (U_+^\dagger)^2 \right) \frac{1}{\mu_+} (R_+^\dagger(1) -  R_+^\dagger(d^\dagger) ).
\end{align*}
Multiplying by $\mu$, we arrive at the identity~\eqref{conservation of flow force}.
\end{proof}

\begin{remark}
\label{conjugate flow remark}
Letting $\delta =0$ and sending $\mu \to 0$, this recovers the well-known conjugate flow equations for the incompressible case; see, for example, \cite[Appendix A]{Laget1997interfacial} or \cite[Section 7.3]{chen2022center}. In that setting, note that conservation of mass flux~\eqref{conservation of mass flux} in the lower layer forces 
\[
	 U_-^\dagger = U \frac{d}{d^\dagger}, \qquad  U_+^\dagger = U \frac{1-d}{1-d^\dagger},
\]
which allows one to formulate the conjugate flow equation entirely in terms of the unknowns $(U,d,d^\dagger,\rho_0,F)$. This remains possible in the compressible setting so long as $U_\pm^\dagger \in (-1/\sqrt{\mu_\pm},0)$, which is the interval on which $t \mapsto t \exp{(-\mu_\pm t^2/2)}$ is injective.
\end{remark}

\subsection{Flattening transform}

\begin{figure}
	\includegraphics{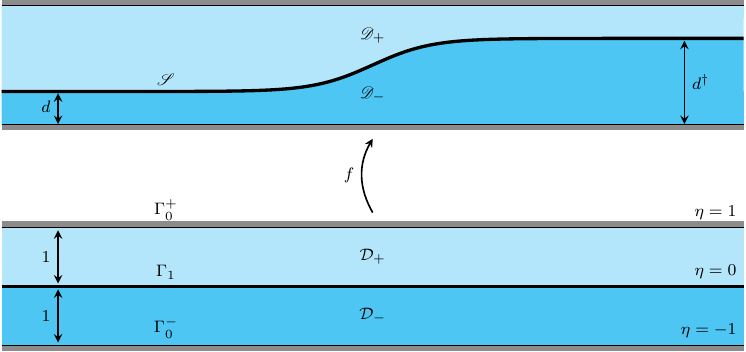}
	\caption{The flattening transform and canonical domain $\confD = \confD_+ \cup \confD_-$. Note that in contrast to the previous sections, $f$ is \emph{not} a conformal mapping.}
	\label{front flattening figure}
\end{figure}

To formulate the problem~\eqref{steady two-phase system} as an abstract operator equation, we once more make a change of coordinates to fix the domain. The conformal mapping approach of the previous section is less appealing here as one would need a separate conformal mapping for each fluid layer, and then a ``conformal welding'' problem arises as the two mappings give conflicting parameterizations of the interface. For that reason, we instead use a more robust, though less elegant, flattening transform that simply rescales the vertical variable along each vertical cross-section of the domain. We continue to use $(\xi,\eta)$ for the independent variables in the flattened domain, which is now the strip
\[
	\confD \colonequals \confD_+ \cup \confD_-, \qquad \confD_\pm \colonequals \{ (\xi,\eta) \in \mathbb{R}^2 : 0 < \pm \eta < 1 \}.
\]
The boundary components are denoted 
\[
	\Gamma_1 \colonequals \mathbb{R} \times \{ 0 \}, \qquad \Gamma_0 \colonequals \Gamma_0^+ \cup \Gamma_0^-, \qquad \Gamma_0^\pm \colonequals \mathbb{R} \times \{ \pm 1 \}.  
\]
Let 
\[
	w = w(\xi) \colonequals h(\xi) -d
\]
be the free surface profile in the flattened domain variables and normalized so that it vanishes in the upstream limit. Then, the coordinate map, which we denote again by $f$, is defined by
\be
\label{definition flattening transform}
	(x,y) = f(\xi,\eta) \colonequals (\xi, f_2(\xi,\eta)), \qquad 
	f_2(\xi,\eta) \colonequals 
		\left\{
		\begin{aligned}
			(1-\eta) (d+w(\xi) -1) + 1 & \qquad  \textrm{in } \confD_+ \\	
			(1+\eta)(d+w(\xi)) & \qquad  \textrm{in } \confD_-.
		\end{aligned}\right.
\ee
Note that $\fluidD_\pm = f(\confD_\pm)$, $\fluidS = f(\Gamma_1)$, and the top and bottom rigid boundaries are likewise the images under $f$ of $\Gamma_0^+$ and $\Gamma_0^-$, respectively. Observe that $\partial_\eta f_{2+} = 1-d-w$ in $\confD_+$ and $\partial_\eta f_{2-} = d+w$ in $\confD_-$. Finally, define
\be
\label{definition upm}
	u_\pm \circ f^{-1} \colonequals \phi_\pm - U x,
\ee
which measures the deviation of the relative potential from its the upstream limit in $\fluidD_\pm$, but pulled back to the flattened domain $\confD_\pm$. See Figure~\ref{front flattening figure} for an illustration.

The steady system~\eqref{steady two-phase Euler} is equivalent to the following problem for $(u,w)$:
\bse
\label{flattened equations}
\be
\left\{ \begin{aligned}
	\Delta^w u_+ & = (\rho_0+\delta) \mu  \mathfrak{D}_u^w  \left( \frac{1}{2}  |\nabla^w u_+ + U \mathbf{e}_1|^2 +  \frac{1}{F^2} ( (1-\eta)(w+d-1) +1 ) \right) &\quad& \textrm{in } \confD_+ \\
	\Delta^w u_- & = \mu  \mathfrak{D}_u^w  \left( \frac{1}{2}  |\nabla^w u_-+ U\mathbf{e}_1|^2 +  \frac{1}{F^2} (\eta+1)(w+d) \right) &\quad& \textrm{in } \confD_-
\end{aligned} \right.
\ee
where the differential operators above are given by
\[
	\nabla^w \colonequals (\mathcal{J}^{-1})^T \nabla_{(\xi,\eta)}, \qquad \mathfrak{D}_u^w \colonequals \nabla^w (u_\pm+U\xi) \cdot \nabla^w, \qquad \Delta^w \colonequals \nabla^w \cdot \nabla^w,
\]
and $\mathcal{J}$ is the Jacobian of $f$
\[
	\mathcal{J} = \mathcal{J}(\xi,\eta) 
		\colonequals Df(\xi,\eta) = \begin{pmatrix}
						1 & 0 \\
						\partial_\xi f_2 & \partial_\eta f_2
					\end{pmatrix}.
\]
Likewise, the kinematic conditions become
\be
\label{steady two-phase kinematic flat}
\left\{ \begin{aligned}
	\partial_\eta u_\pm & = 0 & \qquad & \textrm{on } \Gamma_0^\pm \\
	 \partial_{\nu_+}^w (u_++U\xi) = \partial_{\nu_-}^w (u_-+U\xi)  & = 0 & \qquad & \textrm{on } \Gamma_1,
\end{aligned} \right.
\ee
with $\partial_{\nu_\pm}^w$ being the pullback of the outward (non-normalized) normal derivative relative to $\fluidD_\pm$ along $\fluidS$:
\[
	\partial_{\nu_+}^w = \left( (\partial_\xi w) \partial_\xi - \frac{1+(\partial_\xi w)^2}{1-d-w} \partial_\eta \right)\Big|_{\Gamma_1}, \qquad 
	\partial_{\nu_-}^w = \left( -(\partial_\xi w) \partial_\xi + \frac{1+(\partial_\xi w)^2}{w+d} \partial_\eta \right)\Big|_{\Gamma_1}.
\]
The dynamic condition takes the form
\be
\label{steady two-phase dynamic flat}
	\frac{\rho_0+\delta}{2} |\nabla^w u_+ +U \mathbf{e}_1|^2 - \frac{1}{2} |\nabla^w u_-+U\mathbf{e}_1|^2 + \frac{\rho_0 -1+\delta}{F^2} w = \frac{\rho_0+\delta-1}{2} U^2  \quad \textrm{on } \Gamma_1,
\ee
and, thanks to our normalizations, the upstream asymptotic condition simplifies to
\be
\label{steady two-phase upstream flat}
\left\{
\begin{aligned} 
	\lim_{\xi \to -\infty} \partial_\xi u_\pm(\xi,\placeholder) & = 0, \\
	 \lim_{\xi \to +\infty} \partial_\xi u_\pm(\xi,\placeholder) & = U_\pm^\dagger-U, 
\end{aligned} \right. 
\qquad 
\left\{
\begin{aligned} 
	  \lim_{\xi \to -\infty} w(\xi)&  = 0, \\
	   \lim_{\xi \to +\infty} w(\xi) &= d^\dagger - d.
\end{aligned} \right. 
\ee
\ese
Note that we have used the upstream limit~\eqref{steady two-phase upstream flat} to evaluate the Bernoulli constant $Q$ explicitly in~\eqref{steady two-phase dynamic flat}. Moreover, from the conjugate flow equations of Corollary~\ref{conjugate flow corollary} and Remark~\ref{conjugate flow remark}, we see that for all $\mu$ sufficiently small, the downstream velocities $U_\pm^\dagger$ are uniquely determined by the upstream depth $d$, downstream depth $d^\dagger$, Froude number $F$, Mach number $\mu$, and $\delta$, keeping in mind that we have already normalized the upstream velocity $U = -1$.

\subsection{Function spaces and the abstract operator equation}

The most challenging aspect of adapting our methodology to the case of fronts is finding an appropriate functional-analytic framework. Note, for instance, that $u$ will necessarily grow linearly in the downstream limit $\xi \to +\infty$. There is no avoiding this issue, as regardless of the normalization, one of the far-field velocity fields will necessarily be nonzero for a front. On the other hand, any derivative of $u$ will be uniformly Hölder continuous with well-defined (constant) limits upstream and downstream. 

For a cylindrical subset $\Sigma = \mathbb{R} \times \Sigma^\prime$ of $\overline{\confD}$, integer $k \geq 0$, and $\beta \in (0,1)$, define 
\begin{align*}
	C_\het^{k+\beta}(\overline{\Sigma}) \colonequals \Big\{ v \in C_\bdd^{k+\beta}(\overline{\Sigma})  : & ~ \lim_{\xi \to -\infty} \partial^\gamma v = 0 \\
	&   ~\lim_{\xi \to +\infty} \partial^\gamma v \textrm{ exists } \quad \textrm{ for all } 0 \leq |\gamma| \leq k \Big\},
\end{align*}
with all the above limits uniform in $\eta$. It is easy to see that $C_\het^{k+\beta}$ is a Banach space --- indeed, a Banach algebra --- under the standard Hölder $C^{k+\beta}$ norm on $\Sigma$. These spaces should be compared to the so-called $\Xspace_\infty$ spaces considered in~\cite[Section 2]{chen2023global}; they are a natural setting for studying fronts. Of course, the relative potential $u$ itself does not have a well-defined downstream limit, so we introduce the space
\begin{align*}
	\dot C_{\het}^{k+\beta}(\overline{\confD_\pm})  \colonequals \Big\{ v \in C_\loc^{k+\beta}(\overline{\confD_\pm})  & :  ~v(0,\pm 1) = 0,  ~\partial_\xi v,\, \partial_\eta v \in C_\het^{k+\beta-1}(\overline{\confD_\pm}) \Big\}.
\end{align*}
Notice that we are prescribing the values of the functions at a point, which has the effect of removing (nonzero) constants. The result is again a Banach space under the obvious norm
\[
	\| \placeholder \|_{\dot C_{\het}^{k+\beta}({\confD_\pm})} \colonequals  \| \nabla (\placeholder) \|_{C^{k-1+\beta}(\confD_\pm;\mathbb{R}^2)}. 
\]
Finally, we set
\be
\label{definition X space and Y space front}
\begin{aligned}
	\Xspace_{\het} & \colonequals \left\{ (u,w) \in \left( \dot C_\het^{2+\beta}(\overline{\confD_+}) \cap \dot C_\het^{2+\beta}(\overline{\confD_-}) \right) \times  C_\het^{2+\beta}(\Gamma_1): ~\partial_\eta u_\pm|_{\Gamma_0^\pm} = 0 \right\} \\
	& \equalscolon \Xspace_{1 \het} \times \Xspace_{2\het}  \\
	\Yspace_{\het} &\colonequals \left( C_{\het}^{\beta}(\overline{\confD_+}) \cap C_{\het}^{\beta}(\overline{\confD_-}) \right) \times C_{\het}^{1+\beta}(\Gamma_1) \times C_{\het}^{1+\beta}(\Gamma_1) \times C_{\het}^{1+\beta}(\Gamma_1) \\
	& \equalscolon \Yspace_{1\het} \times \Yspace_{2\het} \times \Yspace_{3\het} \times \Yspace_{4\het}.
\end{aligned}
\ee
Here and below, we write subscripts of $\pm$ to denote the restriction of a function with domain $\confD$ to $\confD_\pm$.

The solutions of the flattened system~\eqref{flattened equations} can then be understood as the zero-set of the nonlinear mapping 
\[
	(\F, \G) = (\F, \G_1, \G_2, \G_3) \colon \mathcal{O}_\het \subset \Xspace_\het \times \mathbb{R}^2 \to \Yspace_\het
\]
with the first component encoding the bulk equations
\begin{align*}
	& \F(u,w; \mu,\delta)  \\
	& \quad \colonequals \left\{
		\begin{aligned}
			\Delta^w u_+ -  (\rho_0+\delta) \mu  \mathfrak{D}_u^w  \left( \frac{1}{2}  |\nabla^w u_+ + U \mathbf{e}_1|^2 +  \frac{1}{F^2} ( (1-\eta)(w+d-1) +1 ) \right) &\quad \textrm{in } \confD_+ \\
			\Delta^w u_- - \mu  \mathfrak{D}_u^w  \left( \frac{1}{2}  |\nabla^w u_-+ U\mathbf{e}_1|^2 +  \frac{1}{F^2} (\eta+1)(w+d) \right) &\quad \textrm{in } \confD_-,
		\end{aligned}
	\right.
\end{align*}
the second and third giving the kinematic boundary condition on $\Gamma_1$, 
\[
	\G_1(u,w; \mu,\delta) \colonequals  \partial_{\nu_+}^w u_+\Big|_{\Gamma_1}, \qquad \G_2(u,w; \mu,\delta) \colonequals  \partial_{\nu_-}^w u_-\Big|_{\Gamma_1},
\]
and the fourth being the dynamic condition on the interface 
\begin{align*}
	 \G_3(u,w; \mu,\delta) &  \colonequals \Big(\frac{\rho_0+\delta}{2} |\nabla^w u_+ +U \mathbf{e}_1|^2 - \frac{1}{2} |\nabla^w u_-+U\mathbf{e}_1|^2  \\
	& \qquad + \frac{\rho_0 -1+\delta}{F^2} w - \frac{\rho_0+\delta-1}{2} U^2\Big)\Big|_{\Gamma_1}.
\end{align*}
\subsection{Stream function formulation of the incompressible system}
\label{stream function formulation section}

Here, we recall the treatment of the incompressible front problem in~\cite{chen2022center}. An entirely technical --- but far from small --- nuisance is that that paper formulates the  system in terms of the relative stream function, which is the harmonic conjugate of the relative potential. We will need to show that nondegeneracy in this formulation implies nondegeneracy in ours. While this is reasonable to expect, actually verifying it requires some effort.

First, we return to the physical coordinates. Let the relative stream function $\psi \colon \overline{\fluidD} \to \mathbb{R}$ be defined so that its restriction $\psi_\pm= \psi|_{\fluidD_\pm}$ is a harmonic conjugate of $\phi_\pm$:
\bse
\label{stream function problem}
\be
\label{stream function bulk equation}
	\Delta \psi_\pm = 0, \quad \left( -\partial_y \psi_\pm, \, \partial_x \psi_\pm \right) \equalscolon \nabla^\perp \psi_\pm = \nabla \phi_\pm \qquad \textrm{in } \fluidD_\pm.
\ee
The kinematic conditions becomes the requirement that $\psi$ is constant on $\fluidS$ and $\fluidB_\pm$. For definiteness, say 
\be
\label{stream function kinematic condition}
\psi_\pm = 0 \quad \textrm{on } \fluidS, \qquad \psi_\pm = m_\pm \qquad \textrm{on } \fluidB_\pm,
\ee
with $m_+$ and $m_-$ being the (constant) volumetric mass flux through upper and lower layers. This choice ensures the kinematic condition is satisfied and selects a unique harmonic conjugate. The Bernoulli condition is simply
\be
\label{stream function dynamic condition}
	\frac{\rho_0}{2} |\nabla \psi_+|^2 - \frac{1}{2} |\nabla \psi_-|^2 + \frac{\rho_0 -1}{F^2} y = Q \qquad \textrm{on } \fluidS.
\ee
Unlike $\phi$, the relative stream function does not grow linearly but rather has the upstream and downstream limits
\be
\label{stream function far-field condition}
	\lim_{x\to-\infty} \psi_\pm(x,y) = -U(y-d), \qquad \lim_{x \to +\infty} \psi_\pm(x,y) =  -U_\pm^\dagger (y-d^\dagger).
\ee
\ese
Note that this means $m_- = Ud$ and $m_+ = -U(1-d)$.

The existence theory for front-type solutions of the incompressible problem~\eqref{stream function problem} begins with the work of Amick and Turner~\cite{amick1989small}, who constructed small-amplitude fronts for which $d \approx d^\dagger$ and $h$ is nearly constant. Through varying methods, similar solutions were found by Mielke~\cite{mielke1995homoclinic}, Makarenko~\cite{makarenko1992bore,makarenko1999conjugate}, and Chen, Walsh, and Wheeler~\cite{chen2022center}. More recently, families of large-amplitude fronts were constructed via global bifurcation theory in~\cite{chen2023global}, and later shown to include fronts arbitrarily close to overturning~\cite{chen2025extreme}. The main result of these two papers is summarized below.

\begin{theorem}[Large-amplitude incompressible fronts]
\label{large incompressible fronts theorem}
  Fix a $\beta \in (0,1)$ and density ratio $\rho_0 \in (0,1)$.  There exist a global curve $\km_\het$ of solutions to the incompressible internal wave problem \eqref{stream function problem} that admits the global $C^0$ parameterization
  \[
   	\km_\het = \left\{ \left(\psi(s), h(s), d(s) \right) : s \in \mathbb{R} \right\}
\]
with $(\psi(0), h(0), d(0)) = (0,0,1/(1+\sqrt{\rho_0}))$. 
  At each parameter value $s$, the Froude number and downstream depth and velocities are given by
  \be
  \label{incompressible conjugate flow relations}
  	F^2  = \frac{1-\sqrt{\rho_0}}{1+\sqrt{\rho_0}}, \qquad d^\dagger = \frac{1}{1+\sqrt{\rho_0}}, \qquad U_\pm^\dagger = U_\pm^\dagger(s) = U \frac{d_\pm(s)}{d_\pm^\dagger},
\ee
where $d_-(s) \colonequals d(s)$, $d_+(s) \colonequals 1-d(s)$, $d_-^\dagger \colonequals d^\dagger$, and $d_+^\dagger \colonequals 1 - d^\dagger$,
and
  \[ 
    \psi(s) \in C_\bdd^{2+\beta}(\overline{\fluidD_+(s)}) \cap C_\bdd^{2+\beta}(\overline{\fluidD_-(s)}) \cap C_\bdd^0(\overline{\fluidD(s)}), \qquad h(s) \in C_\bdd^{2+\beta}(\mathbb{R}), 
  \]
  where $\fluidD(s)$ denotes the fluid domain corresponding to $h(s)$ and $d(s)$.  Moreover, $h(s)$ is locally real analytic, and $\psi_\pm(s)$ admits an analytic extension to $\overline{\fluidD_\pm(s)}$.
  \begin{enumerate}[label=\rm(\alph*)]
  \item \label{global bore monotone part} \textup{(Strict monotonicity)} Each solution on $\km_\het$ is a strictly monotone front:
    \begin{equation}
      \begin{aligned}
        (\signum{s}) \partial_x h(s) &< 0 & \qquad &  \textrm{on } \mathbb{R}, \\
        (\signum{s}) \partial_x \psi_\pm(s) & < 0 & \qquad & \textrm{in } \overline{\fluidD_\pm(s)} \setminus \fluidB_\pm , \\
        \partial_y \psi_\pm(s) & < 0 & \qquad & \textrm{in } \overline{\fluidD_\pm(s)}.
      \end{aligned} \label{monotonicity Euler variables}
    \end{equation}
  \item \textup{(Limits)} \label{front limits part} In the limit along the elevation front portion of $\km_\het$, overturning occurs in that 
    \begin{equation}
      \limsup_{s \to -\infty} \| \partial_x h(s) \|_{L^\infty(\mathbb{R})} = \infty. \label{overturning} 
    \end{equation}
	 In the limit along depression front portion of $\km_\het$, either \emph{overturning} occurs or the interface comes arbitrarily close to the upper wall: 
  \begin{equation}
    \limsup_{s \to \infty} \| \partial_x h(s) \|_{L^\infty(\mathbb{R})} = \infty \quad \textrm{or} \quad \lim_{s \to \infty} d(s) = 1. \label{depression limit alternatives} 
  \end{equation}
  	\item \label{front discrete set part} There exist a discrete set $S \subset \mathbb{R}$ such that $(\psi(s), h(s), d(s))$ is nondegenerate for $s \not\in S$. 
  \end{enumerate}
\end{theorem}
\begin{proof}
The existence of $\km_\het$ is proved in \cite[Theorem 5.1]{chen2023global}. Real analyticity of the free surface profile $h(s)$ in the case of solitary internal waves was shown by Amick and Turner~\cite[Theorem 2.1]{amick1986global}, but their argument is entirely local and holds for fronts mutatis mutandis. The assertion in part~\ref{front limits part} that overturning necessarily occurs as $s \to -\infty$ was shown later in~\cite[Theorem 1.1]{chen2025extreme}. Finally, part~\ref{front discrete set part} comes, as usual, from the analytic global bifurcation theoretic machinery used in the construction of $\km_\het$; see the discussion following Theorem~\ref{homoclinic global ift}.
\end{proof}

Analogously to~\eqref{definition upm}, we define the corresponding unknown for the flattened problem to be
\be
\label{definition tilde u}
	\tilde u_\pm \circ f^{-1} = \psi_\pm + U( y-d),
\ee
which now measure the deviation of the relative stream function from its upstream limiting state. Note that $\tilde u$ vanishes on the rigid boundaries $\Gamma_0^\pm$, and, while it is continuous across $\Gamma_1$, it is generally not constant there. The natural function space setting is then 
\be
\label{definition tilde spaces}
\begin{aligned}
	\tilde{\Xspace}_\het & \colonequals \left\{ v \in C_\het^{2+\beta}(\overline{\confD_+}) \cap C_\het^{2+\beta}(\overline{\confD_-}) : ~ v|_{\Gamma_0^\pm} =0   \right\} \times \Xspace_{2\het} \equalscolon \tilde{\Xspace}_{1\het} \times \Xspace_{2\het} \\
	\tilde{\Yspace}_\het & \colonequals \Yspace_\het
\end{aligned}
\ee
and the abstract operator formulation becomes
\[
	(\tilde{\F}, \tilde{\G})(\tilde u, w; d) = 0
\]
for $(\tilde \F, \tilde\G_1, \tilde \G_2, \tilde \G_3) \colon \tilde{\mathcal{O}}_\het \subset \tilde{\Xspace}_\het \times (0,1) \to \tilde{\Yspace}_\het$ defined by
\be
\label{tilde F G definition}
\begin{aligned}
	\tilde{\F}(\tilde u, w; d) & \colonequals \Delta^w \tilde u \\
	\tilde{\G}_1(\tilde u,w; d) & \colonequals \partial_\xi \tilde u_+\Big|_{\Gamma_1} \\
	\tilde{\G}_2(\tilde u,w; d) & \colonequals \partial_\xi \tilde u_-\Big|_{\Gamma_1} \\
	\tilde{\G}_3(\tilde u,w; d) & \colonequals \left(\frac{\rho_0}{2} |\nabla^w \tilde u_+ -U \mathbf{e}_2|^2 - \frac{1}{2} |\nabla^w \tilde u_- - U\mathbf{e}_2|^2  + \frac{\rho_0 -1}{F^2} w - \frac{\rho_0-1}{2} U^2  \right)\Big|_{\Gamma_1}.
\end{aligned}
\ee
Here, $\tilde{\mathcal{O}}_{\het}$ is the open subset of $\tilde{\Xspace}_\het \times (0,1)$ derived in the obvious way from $\mathcal{O}_\het$. For consistency with the existence theory in \cite{chen2022center,chen2023global}, we are treating the upstream depth $d$ as the only varying parameter. From the conjugate flow equations in Corollary~\ref{conjugate flow corollary}, it can be shown that $d$ and $U$ determine $d^\dagger$ and $U_\pm^\dagger$ uniquely as in~\eqref{incompressible conjugate flow relations}.

\subsection{Formulation independence of nondegeneracy}

The central question addressed in this section is whether nondegeneracy of a solution to the incompressible problem implies nondegeneracy of solutions to the compressible problem at $\mu = 0$. One of course expects this to be true, but the fact that the two are expressed in terms of different quantities --- the stream function in the former case and the potential in the latter --- adds a major wrinkle. To transition between the stream function formulation and the velocity potential formulation amounts to taking the harmonic conjugate. 

With that in mind, for a fixed $w \in \Xspace_{2\het}$, consider the mapping
\[
	\mathcal{C}^w  \colon \Xspace_{1\het}  \to C_\bdd^{2+\beta}(\overline{\confD_+}) \cap C_\bdd^{2+\beta}(\overline{\confD_-})   
\]
whose restriction to $\confD_\pm$ is given by
\be
\label{definition Cw}
\begin{aligned}
	(\mathcal{C}^w u)_\pm(\xi,\eta)  & \colonequals -\int_{\pm 1}^\eta \left( \partial_\eta f_{2\pm}(\xi,\eta^\prime) \partial_\xi ( u_\pm (\xi,\eta^\prime)+ U\xi) - \partial_\xi f_{2\pm}(\xi,\eta^\prime) \partial_\eta ( u_\pm(\xi,\eta^\prime) + U\xi ) \right)  \, d\eta^\prime \\
		& \qquad + U( f_{2\pm}(\xi,\eta)-f_{2\pm}(\xi,\pm 1)).
\end{aligned}
\ee

Here, and in what follows, the dependence of $f$ on $w$ has been suppressed to improve readability. 

The motivation behind the formula~\eqref{definition Cw} is best seen when it is expressed in physical coordinates:
\[
	\begin{aligned}
		((\mathcal{C}^w u)_+ \circ f^{-1})(x,y) & =  -\int_{1}^y \left( \partial_x \phi_+(x,y^\prime) - U \right) \, dy^\prime & \\
		 ((\mathcal{C}^w u)_- \circ f^{-1})(x,y) & =  -\int_{0}^y \left(\partial_x \phi_-(x,y^\prime) - U \right)\, dy^\prime.
	 \end{aligned}
\]
In the incompressible case where $\phi_\pm$ is harmonic and its normal derivative vanishes on the interface, the integral on the right-hand side of~\eqref{definition Cw} is  the harmonic conjugate of $\phi_\pm$ pulled back to the flattened coordinates. The second term represents the limiting stream function upstream. In view of the definition of $\tilde{u}$ in~\eqref{definition tilde u}, we hope that when $(u,w)$ is a solution to the incompressible problem~\eqref{flattened equations}, then $(\mathcal{C}^w u,w)$ solves the stream function formulation of the problem.

It is easily verified that when $(u,w)$ represents an incompressible front, then $\mathcal{C}^w u$ vanishes on $\Gamma_0^\pm$, thanks to the boundary conditions satisfied by $u$ and the conjugate harmonic relations. Observe, moreover, that for any $(u,w)$, $\mathcal{C}^w u$ and its derivatives have well-defined upstream and downstream limits, and, in particular,
\[
	\lim_{\xi \to -\infty} (\mathcal{C}^w u)(\xi,\eta) = 0.
\]
As a consequence, we have that $\mathcal{C}^w u \in \tilde{\Xspace}_{1\het}$ when $(u,w) \in (\F,\G)(\placeholder, \placeholder; 0,0)^{-1}(0)$.

There is likewise an inverse mapping that converts the stream function formulation to the relative potential formulation when $\mu = 0$. It is found by writing the harmonic conjugate of the stream function in the flattened variables, then subtracting the appropriate function to have the desired upstream limit. For any $w \in \tilde{\Xspace}_{2\het}$, let  $\tilde{\mathcal{C}}^w \colon \tilde{\Xspace}_{1\het} \to \Xspace_{1\het}$ be the mapping whose restriction to $\confD_\pm$ is given by 
\be
\label{definition tilde Cw}
\begin{aligned}
	(\tilde{\mathcal{C}}^w \tilde u)_\pm(\xi,\eta)  & \colonequals   \int_{\pm1}^\eta \left( \partial_\eta f_{2\pm}(\xi,\eta^\prime) \partial_\xi \tilde u_\pm(\xi,\eta^\prime) - \partial_\xi f_{2\pm}(\xi,\eta^\prime) \partial_\eta \tilde u_\pm(\xi,\eta^\prime) \right) \, d\eta^\prime \\
	& \qquad -\int_0^\xi \frac{\partial_\eta \tilde u_\pm(\xi^\prime, \pm 1)}{\partial_\eta f_{2\pm}(\xi^\prime, \pm 1)} \, d\xi^\prime.
\end{aligned}
\ee

Again, it is best to understand the meaning of this function in the physical variables, where we see 
\[
	(\tilde{\mathcal{C}}^w \tilde u)_\pm \circ f^{-1}(x,y) + U x = -\int_{0}^x (\partial_y \psi_{\pm})(x^\prime, y_0^\pm) \, dx^\prime + \int_{y_0^\pm}^y (\partial_x \psi_{\pm} )(x,y^\prime)  \, dy^\prime,
\]
where $y_0^+ \colonequals 1$ and $y_0^- \colonequals 0$ represent the $y$-coordinate of the corresponding component of $\Gamma_0$.
When $\psi_\pm \circ f^{-1}$ is harmonic, the right-hand side is indeed the unique harmonic conjugate normalized so that it vanishes at the point $(0,y_0^\pm)$. We summarize the above observation in the next lemma.

 \begin{lemma}[Correspondence]
 \label{correspondence lemma}
For any $w \in \Xspace_{2\het}$, the mappings $\mathcal{C}^w$ and $\tilde{\mathcal{C}}^w$ given by~\eqref{definition Cw} and \eqref{definition tilde Cw}, respectively, are real-analytic and satisfy the following.
\begin{enumerate}[label=\rm(\alph*)]
	\item If $(u,w) \in \Xspace_\het$ and $\Delta^w u = 0$ in $\confD_+ \cup \confD_-$, then $(\mathcal{C}^w u,w) \in \tilde{\Xspace}_{\het}$ and 
	\[
		\Delta^w \mathcal{C}^w u = 0, \quad \nabla^w u = \begin{pmatrix} 0 & -1 \\ 1 & 0 \end{pmatrix} \nabla^w \mathcal{C}^w u \qquad \textrm{ in } \confD_+ \cup \confD_-.
	\] 
	Conversely, if $(\tilde u,w) \in \tilde{\Xspace}_{\het}$ and $\Delta^w \tilde u =0$ in $\confD_+ \cup \confD_-$, then $(\tilde{\mathcal{C}}^w \tilde u, w) \in \Xspace_\het$, and 
	\[
		\Delta^w \tilde{\mathcal{C}}^w \tilde u = 0, \quad \nabla^w \tilde{\mathcal{C}}^w \tilde u = \begin{pmatrix} 0 & -1 \\ 1 & 0 \end{pmatrix} \nabla^w \tilde u \qquad \textrm{ in } \confD_+ \cup \confD_-.
	\] 
 	\item They are inverses in the sense that
 	\be
 	\label{tilde u tilde v inverses}
 	\begin{aligned}
 		\tilde{\mathcal{C}}^w \mathcal{C}^w u & = u & \qquad & \textrm{for all } (u,w) \in \Xspace_\het, ~ \Delta^w u = 0 \textrm{ on } \confD_+ \cup \confD_- \\
		{\mathcal{C}}^w \tilde{\mathcal{C}}^w \tilde u & = \tilde u & \qquad & \textrm{for all } (\tilde u,w) \in \tilde{\Xspace}_\het, ~ \Delta^w \tilde u =0 \textrm{ on } \confD_+ \cup \confD_-.
	\end{aligned}
	 \ee 
	 
	 \item We have that 
	 \[
	 	(\F,\G)(u,w;0,0) = 0 \qquad \textrm{if and only if} \qquad (\tilde{\F},\tilde{\G})(\mathcal{C}^w u, w; d) = 0.
	 \]
\end{enumerate}
 \end{lemma}

In order to actually relate the linearized problems for the stream function and potential formulations, we make use of the ``good unknowns'' of Alinhac~\cite{alinhac1986paracomposition}, whose importance to the time-dependent water wave problem is well-known. For steady water waves, to the best of our knowledge, they were first used by Wahlén~\cite{wahlen2009critical}, though they were originally called the ``$\mathcal{T}$-isomorphism'' with the connection to Alinhac being realized later. 

The basic idea is as follows. To compute the linearized system, we can imagine taking a one-parameter family $(u(s), w(s))$ with $(u(0), w(0)) = (u,w)$ being the base point for the analysis. Then we may write
\[
	u(s) = \phi(\placeholder; s) \circ f(s) - U \xi,
\]
where $f(s)$ is the flattening transform determined from $w(s)$ via~\eqref{definition flattening transform} and $\phi = \phi(x,y; s)$ is the corresponding relative velocity potential in the physical coordinates. Denoting with a dot the derivative of a quantity in $s$ evaluated at $s =0$, we compute that 
\[
	\dot u = \dot \phi \circ f + (D_{(x,y)}\phi \circ f) \dot f = 
	\left\{
	\begin{aligned}
		\dot \phi \circ f + (1-\eta)\dot w (\partial_y \phi) \circ f & \qquad  \textrm{in } \confD_+ \\
		\dot \phi \circ f + (1+\eta)\dot w (\partial_y \phi) \circ f & \qquad  \textrm{in } \confD_-.
	\end{aligned}
	\right.
\]
Rearranging this identity and pulling it back to the flattened coordinates gives
\[
	\dot \phi_\pm \circ f = \dot u_\pm - \frac{(1\mp \eta)\partial_\eta u}{\partial_\eta f_{2\pm}}\dot w. 
\]

The point here is that the left-hand side is the naïve variation of the potential in the physical coordinates. It is a good unknown in the sense that $\Delta^w \dot \phi \circ f= 0$ in the incompressible case. Working with this quantity is considerably simpler, as it allows us to disentangle the variations of the potential from those of the coordinate change.  Wahlén~\cite{wahlen2009critical} used a version of this to compute the linearization of a one-fluid water wave problem with vorticity at the trivial solution; Varholm~\cite{varholm2020global} extended this idea to linearization at arbitrary solutions. The lemma below adapts their results to the systems~\eqref{flattened equations} and \eqref{tilde F G definition}. 

\begin{lemma}[Good unknowns] 
\label{good unknowns lemma}
Suppose that 
\[
	(u,w) \in (\F,\G)(\placeholder; 0,0)^{-1}(0), \qquad (\tilde u, w) \in (\tilde\F,\tilde\G)(\placeholder;d)^{-1}(0)
\]
are incompressible fronts in the relative potential and relative stream function formulation, respectively. 
\begin{enumerate}[label=\rm(\alph*)]
	\item \label{good unknown isomorphism part} The linear mapping $\mathcal{T}(u,w) \colon \Xspace_{\het} \to \Xspace_\het$ whose restriction to $\confD_\pm$ is given by
	\[
		\mathcal{T}_\pm(u,w) \begin{pmatrix} \dot v \\ \dot w \end{pmatrix}  \colonequals \begin{pmatrix} \dot v + \frac{ (1\mp \eta) (\partial_\eta u)}{\partial_\eta f_{2\pm}}  \dot w \\  \dot w \end{pmatrix}
	\]
	is an isomorphism. The same is true for the corresponding linear mapping $\tilde{\mathcal{T}}(\tilde u, w) \colon \tilde{\Xspace}_\het \to \tilde{\Xspace}_\het$ whose restriction to $\confD_\pm$ is given by
	\[
		\tilde{\mathcal{T}}_\pm(\tilde u,w) \begin{pmatrix} \dot{\tilde{v}} \\ \dot w \end{pmatrix} \colonequals \begin{pmatrix} \dot{\tilde v} + \frac{ (1\mp \eta) (\partial_\eta \tilde u)}{\partial_\eta f_{2\pm}}  \dot w \\   \dot w\end{pmatrix}.
	\]
	\item 
	\label{equivalence of linearizations part} For all $(\dot v,\dot w) \in \Xspace_\het$ with $\Delta^w \dot v = 0$ in $\confD$, we have the identity
	 \[
 		D_{(\tilde u,w)} ( \tilde{\F},\tilde{\G})(\mathcal{C}^w u,w;d)  \tilde{\mathcal{T}}(\mathcal{C}^w u,w) \begin{pmatrix} \mathcal{C}^w \dot v \\ \dot w \end{pmatrix}  = D_{(u,w)}( {\F},{\G})(u,w;0,0)  \mathcal{T}(u,w) \begin{pmatrix} \dot v \\ \dot w \end{pmatrix}. 
 	\]
	In particular,
	\begin{align*}
		D_{(u,w)} \F(u,w;0,0) \mathcal{T}(u,w)\begin{pmatrix} \dot v \\ \dot w \end{pmatrix} & = \Delta^w \dot v \\
		D_{(\tilde u,w)} \tilde\F(\tilde u,w;d) \tilde{\mathcal{T}}(\tilde u,w)\begin{pmatrix} \dot{\tilde v} \\ \dot w \end{pmatrix} & = \Delta^w \dot{\tilde v}
	\end{align*}
\end{enumerate}
\end{lemma}
\begin{proof}
Consider first the statements in part~\ref{good unknown isomorphism part}. It is easy to check that these are indeed bounded mappings between the stated spaces. In particular, for $(\dot v, \dot w) \in \Xspace_{\het}$ we have that $(\dot u, \dot w) \colonequals \mathcal{T}(u,w)(\dot v, \dot w)^T$ satisfies
\[
	\dot u(0,\pm 1) = \left( \dot v + \frac{ (1\mp \eta) (\partial_\eta u)}{\partial_\eta f_{2\pm}}  \dot w \right)(0,\pm 1) = 0, 
\]
as the second term vanishes identically on $\Gamma_0^\pm$ and the first is in $\Xspace_{1\het}$. Likewise,
\[
	\partial_\eta \dot u|_{\Gamma_0^\pm} = \left(\partial_\eta \dot v \mp \frac{\partial_\eta u}{\partial_\eta f_{2\pm}} \dot w \right)\Big|_{\Gamma_0^\pm} = 0,
\]
again thanks to the fact that $\dot v$ and $u$ are in $\Xspace_{1\het}$. In fact, $\mathcal{T}_\pm(u,w)$ has the explicit inverse
\[
	\mathcal{T}_\pm(u,w)^{-1} \begin{pmatrix} \dot u \\ \dot w \end{pmatrix} = \begin{pmatrix} \dot u - \frac{(1\mp \eta) (\partial_\eta u)}{\partial_\eta f_{2\pm}} \dot w \\ \dot w \end{pmatrix},
\]
which is clearly bounded. The same basic argument applies for $\tilde{\mathcal{T}}(\tilde u, w)$.  

The identities in part~\ref{equivalence of linearizations part} can be established through an elementary but long calculation following~\cite[Theorem 5]{varholm2020global}, so we omit the details.
\end{proof}

As a consequence of Lemma~\ref{good unknowns lemma}, nondegeneracy of a solution to the stream function formulation~\eqref{tilde F G definition} of the incompressible problem implies the linearized problem at the corresponding solution to $(\F,\G)(\placeholder; 0,0) = 0$ is likewise an isomorphism. 

\begin{lemma}[Nondegeneracy]
\label{nondegeneracy front lemma}
Suppose now that $(\tilde u, w) \in \tilde{\mathcal{O}}_\het$ is a nondegenerate solution of the stream function problem in that
\be
\label{definition nondegenerate stream function}
	(\tilde\F, \tilde\G)(\tilde u, w; d) = 0 ,\qquad D_{(\tilde u,w)} (\tilde\F,\tilde\G)(\tilde u,w;d)   \textrm{ is an isomorphism } \tilde{\Xspace}_\het \to \tilde{\Yspace}_\het. 
\ee
Then, $(u,w) \colonequals (\tilde{\mathcal{C}}^w \tilde u,w) \in \mathcal{O}_\het$ satisfies
\[
	(\F,\G)(u,w; 0,0) = 0, \qquad D_{(u,w)} (\F,\G)(u,w; 0,0) \textrm{ is an isomorphism } \Xspace_\het \to \Yspace_\het.
\] 
\end{lemma}
\begin{proof}
From Lemma~\ref{correspondence lemma} and Lemma~\ref{good unknowns lemma}, it follows that $(\F,\G)(u,w; 0,0) = 0$ and the kernel of $D_{(u,w)} (\F,\G)(u,w;0,0)$ is trivial. It remains only to show that $D_{(u,w)} (\F,\G)(u,w;0,0)$ is surjective. With that in mind, let $Y = (Y_1, Y_2, Y_3,Y_4) \in \Yspace_\het$ be given; we must prove that there exists $(\dot v, \dot w) \in \Xspace_\het$ so that 
\[
	D_{(u,w)} (\F,\G)(u,w;0,0) \mathcal{T}(u,w) 
		\begin{pmatrix} 
			\dot v \\ 
			\dot w 
		\end{pmatrix} 
		= Y. 
\]

First, we wish to solve the inhomogeneous problem in the interior with simplified boundary conditions. We claim that there exists $\dot v_1 \in \dot C_\het^{2+\beta}(\overline{\confD_+}) \cap  \dot C_\het^{2+\beta}(\overline{\confD_-})$ such that
\be
\label{inhomogeneous Poisson equation}
	\left\{ \begin{aligned}
		\Delta^w \dot v_1 & = Y_1 & \qquad & \textrm{in } \confD_+ \cup \confD_- \\
		\partial_\eta \dot v_1 & = 0 &\qquad& \textrm{on } \Gamma_0^+ \cup \Gamma_0^-.
	\end{aligned} \right.
\ee
To see this, consider first the problem in the physical variables. Since $w \in C^{2+\beta}_\het(\mathbb{R})$, we can extend $Y_{1\pm} \circ f^{-1}$ to a function $Z^\pm \in C_\bdd^{\beta}(\overline{\fluidD})$ vanishing in a neighborhood of $\fluidB_\mp$ and in the upstream limit $x \to -\infty$. The Laplacian in $\fluidD$ can be inverted explicitly using an integral kernel. Indeed, setting 
\[
	V^\pm(x,y)  \colonequals \int_{-\infty}^\infty \int_{0}^1  {G}^\pm(x,y; \, x^\prime, y^\prime) Z^\pm(x^\prime, y^\prime) \, d y^\prime \, dx^\prime,
\]
for the functions ${G}^\pm$ given by
\begin{align*}
	{G}^+(x,y; x^\prime, y^\prime) & \colonequals \frac{1}{4\pi} \log{\left( \frac{\cosh{(\frac{\pi}{2} (x-x^\prime))} - \cos{(\frac{\pi}{2} (y+y^\prime))}}{\cosh{( \frac{\pi}{2} (x-x^\prime))} - \cos{(\frac{\pi}{2} (y-y^\prime))}}\right)} \\
	{G}^-(x,y; x^\prime, y^\prime) & \colonequals \frac{1}{4\pi} \log{\left( \frac{\cosh{(\frac{\pi}{2} (x-x^\prime))} - \cos{(\frac{\pi}{2} (y+y^\prime))}}{\cosh{(\frac{\pi}{2} (x-x^\prime))} - \cos{(\frac{\pi}{2} (y-y^\prime))}}\right)},
\end{align*}
it is easy to confirm that $V^\pm \in C_\bdd^{2+\beta}(\overline{\fluidD})$,  and
\[
	\left\{ \begin{aligned}
		\Delta {V}^\pm &  = Z^\pm &\qquad& \textrm{in } \fluidD \\
		 V^\pm & = 0 &\qquad& \textrm{on } \mathscr{B}_\mp \\
		 \partial_y V^\pm & = 0 & \qquad & \textrm{on } \mathscr{B}_\pm.
	\end{aligned} \right.
\]
The exponential decay of the kernels $G^\pm$ in the $x$- and $x^\prime$-directions, moreover, ensures that
\[
	\lim_{x \to - \infty} \partial^\gamma V^\pm(x, \placeholder) = 0\quad \textrm{for all } |\gamma| \leq 2, \quad \textrm{and } \lim_{x \to +\infty} \partial^\gamma V^\pm(x,\placeholder) \textrm{ exists for all } |\gamma| \leq 2.
\]

Now, taking $v^\pm \colonequals V^\pm \circ f |_{\confD_\pm}$, we find that $\Delta^w v^\pm = Y_1$ in $\confD_\pm$ , and $\partial_\eta v^\pm = 0$ on $\Gamma_0^\pm$. 
Finally, let $\dot v_1$  be the function whose restriction to $\confD_\pm$ is $v^\pm - v^\pm(0,\pm1)$; this choice furnishes a solution to the Poisson problem~\eqref{inhomogeneous Poisson equation} lying in the space $\dot C_\het^{2+\beta}(\overline{\confD_+}) \cap  \dot C_\het^{2+\beta}(\overline{\confD_-})$. 

On the other hand, Lemma~\ref{good unknowns lemma}~\ref{equivalence of linearizations part} and the nondegeneracy of $(\tilde u,w)$ also ensures the existence of a unique $(\dot v_2, \dot w_2) \in \Xspace_\het$ such that 
\[
	 D_{( u, w)} ( \F,  \G)( u, w;0,0) {\mathcal{T}}(u, w) \begin{pmatrix} \dot{v}_2 \\ \dot w_2 \end{pmatrix} = \begin{pmatrix} 0 \\ Y_2 \\ Y_3 \\ Y_4 \end{pmatrix} 
	 - 
	 \begin{pmatrix} 
	 	0 \\
		D_{(u, w)} \G(u,w;0,0) {\mathcal{T}}( u, w) \begin{pmatrix} \dot v_1 \\ 0 \end{pmatrix} 
	\end{pmatrix}.
\]
Simply taking $\dot v = \dot v_1 + \dot v_2$ and $\dot w = \dot w_2$ completes the proof of the claim and therefore the lemma. 
\end{proof}

\subsection{Existence theory}

We are now prepared to prove our main theorem on the existence of lightly compressible fronts. First, we state this as a (rigorous) Janzen--Rayleigh expansion of a given nondegenerate solution of the incompressible problem~\eqref{stream function problem}. 

\begin{theorem}[Janzen--Rayleigh expansion for internal fronts]
\label{front theorem}
Suppose that for a fixed $\rho_0$ and $d$, $(u_0,w_0) \in \mathcal{O}_\het$ is a solution to the incompressible internal front problem~\eqref{flattened equations} which is nondegenerate in that $(\mathcal{C}^w u_0, w_0)$ is a nondegenerate solution of the corresponding stream function problem~\eqref{definition nondegenerate stream function}. There exists a two-dimensional real-analytic manifold $\mathscr{M}_\het^\loc$ of solutions to the compressible internal front problem \eqref{steady two-phase system} with the global parameterization
\[
	\mathscr{M}_\het^\loc = \left\{ (u^{\mu,\delta}, w^{\mu,\delta}; \mu,\delta) \in \Xspace_\het \times \mathbb{R} \times \mathbb{R} : |\mu| < \mu_0,~|\delta|< \delta_0 \right\} \subset \mathcal{O}_\het \cap (\F, \G)^{-1}(0),
\] 
for some $\mu_0 > 0$ and $\delta_0 > 0$. The solutions along $\mathscr{M}_\het^\loc$ admit the infinite power series representation 
\be
\label{JRE front}
\begin{aligned}
	 u^{\mu,\delta} & = u_0 + \sum_{m+n \geq 1} \mu^m \delta^n u_{m,n}, \qquad w^{\mu,\delta} = w_0 + \sum_{m+n \geq 1} \mu^m \delta^n w_{m,n} \\
	 (U_+^\dagger)^{\mu,\delta} & = U \frac{1-d}{1-d_0^\dagger} + \mu \frac{\rho_0 U (1-d)(d-d_0^\dagger)\left( F^2 U^2 (d+d_0^\dagger-2) - (1-d_0^\dagger)^2 \right)}{2 F^2 (1-d_0^\dagger)^3}  + O(\mu^2 + \delta) \\
	 (U_-^\dagger)^{\mu,\delta} & = U\frac{d}{d_0^\dagger} + \mu \frac{Ud \left( F^2 U^2 d^2 -F^2 U^2 (d_0^\dagger)^2 - d (d_0^\dagger)^2 + (d_0^\dagger)^3 \right)}{2F^2 (d_0^\dagger)^3} +O(\mu^2+\delta),
	 \end{aligned}
\ee
where $F^2 = \frac{1-\sqrt{\rho_0}}{1+\sqrt{\rho_0}}$ and $d_0^\dagger = 1/(1+\sqrt{\rho_0})$ are the Froude number and downstream depth for the incompressible case, and the leading-order terms $(u_{1,0}, w_{1,0})$ are the unique solutions in $\Xspace_\het$ of the linear problem 
\[
	D_{(u,w)}(\F,\G)(u_0,w_0;0,0) \begin{pmatrix} u_{1,0} \\ w_{1,0} \end{pmatrix} = -D_\mu (\F,\G)(u_0,w_0;0,0).
\]
\end{theorem}
\begin{proof}
The existence of the manifold $\mathscr{M}_\het^\loc$ follows from Lemma~\ref{nondegeneracy front lemma} and the (analytic) implicit function theorem. To derive the specific Janzen--Rayleigh asymptotics claimed in~\eqref{JRE front}, we do an expansion of the conjugate flow equations in Corollary~\ref{conjugate flow corollary} at the starting incompressible flow. 

Set $\delta = 0$, and let $\dot U_\pm^\dagger \colonequals \partial_{\mu}|_{\mu = 0} (U_\pm^\dagger)^{\mu,0}$. Consider first the conservation of mass flux in the lower layer in~\eqref{conservation of mass flux}, which here implies
\[
		 U e^{-\frac{\mu}{2} U^2  } \frac{ 1 - e^{- \frac{\mu d}{F^2}} }{1- e^{- \frac{\mu d^\dagger}{F^2} }}  =
	U_-^\dagger e^{-\frac{\mu}{2} (U_-^\dagger)^2 } \qquad \textrm{for } d^\dagger = (d^\dagger)^{\mu,0}, \quad U_-^\dagger = (U_-^\dagger)^{\mu,0},\quad 0 \leq \mu \ll 1.
\]
Note that we already know that $(U_-^\dagger)^{0,0} = U d/d_0^\dagger$ by the conjugate flow equations in the incompressible case; see Remark~\ref{conjugate flow remark}. Then, differentiating both sides of the above identity in $\mu$ and setting $\mu = 0$ (and using that $\partial_\mu (d^\dagger)^{\mu,0}|_{\mu=0} = 0$), we eventually arrive at 
\[
	U \frac{d}{d_0^\dagger} \left( \frac{d_0^\dagger -d}{2F^2} - \frac{U^2}{2} \right) = \dot U_-^\dagger - \frac{1}{2} \left(U \frac{d}{d_0^\dagger} \right)^3. 
\]
Rearranging this and simplifying leads to the $O(\mu)$ term for the expansion of $U_-^\dagger$ in~\eqref{JRE front}. The same type of manipulation applied to the conservation of mass flux in the upper layer gives the corresponding formula for $\dot U_+^\dagger$.
\end{proof}
\begin{remark}
In principle, one can also determine asymptotic expansions (to any order) for the downstream depth $(d^\dagger)^{\mu,\delta}$ of the corresponding fronts along $\mathscr{M}_\het^\loc$. However, because $d_0^\dagger$ is a double root of conservation of flow force equation when $\mu = 0$, it turns out that computing even the $O(\mu)$ term in the expansion of $(d^\dagger)^{\mu,\delta}$ requires expanding the compressible conjugate flow system to at least second order. The resulting equations are complicated enough as to be uninstructive, so we do not provide them here. In the same way, one can also expand all of the solutions in $\delta$.
\end{remark}

Applying this result along each distinguished arc of the global family $\km_\het$ of incompressible elevation and depression bores from \cite{chen2023global,chen2025extreme}, we immediately obtain the following existence result for large-amplitude lightly compressible fronts.

\begin{corollary}[Manifolds of lightly compressible fronts]
\label{front corollary}
Let $\{\mathscr{A}_j^\het\}$  be the (at most countable) distinguished arcs of the global family of incompressible fronts $\km_\het$ given in Theorem~\ref{large incompressible fronts theorem}. There exists a corresponding at most countable ordered list $\{ \mathscr{M}_j^\het\}$ of three-dimensional real-analytic manifolds of compressible fronts. Each $\mathscr{M}_j^\het$ has the structure of a (real-analytic) fiber bundle over the base $\mathscr{A}_j^\het$.
\end{corollary}

\section*{Acknowledgments}

The research of SW is supported in part by the National Science Foundation through DMS-2306243 and the Simons Foundation through award 960210.

\bibliographystyle{siam}
\bibliography{projectdescription}

\end{document}